\documentclass{amsart}

\usepackage{amsfonts,amssymb,amscd,amsmath,enumerate,verbatim,calc}
\usepackage{euscript}
\usepackage[all]{xy}

\newcommand{\tFAE}{the following conditions are equivalent:}
\newcommand{\TFAE}{The following conditions are equivalent:}
\newcommand{\CM}{Cohen-Macaulay}
\newcommand{\fg}{finitely generated \ }

\newcommand{\eF}{\EuScript{F}}
\newcommand{\ff}{\text{if and only if}}
\newcommand{\wrt}{with respect to}
\newcommand{\f}{\hat{f} }

\newcommand{\n}{\mathfrak{n} }
\newcommand{\m}{\mathfrak{m} }
\newcommand{\tf}{\mathfrak{t} }
\newcommand{\M}{\mathfrak{M} }
\newcommand{\N}{\mathfrak{N} }
\newcommand{\q}{\mathfrak{q} }

\newcommand{\A}{\mathfrak{a} }
\newcommand{\B}{\mathfrak{b} }
\newcommand{\C}{\mathfrak{c} }
\newcommand{\R}{\mathcal{R} }
\newcommand{\Pc}{\mathcal{P} }

\newcommand{\Z}{\mathbb{Z} }
\newcommand{\QQ}{\mathbb{Q} }
\newcommand{\nZ}{n \in \mathbb{Z} }
\newcommand{\iZ}{i \in \mathbb{Z} }

\newcommand{\spr}{\mathfrak{s} }
\newcommand{\GA}{G_{\mathfrak{a}}(A) }
\newcommand{\GM}{G_{\mathfrak{a}}(M) }
\newcommand{\GB}{G_{\mathfrak{b}}(B) }
\newcommand{\GT}{G_{\mathfrak{t}}(T) }
\newcommand{\ra}{\EuScript{R}_{\mathfrak{a}}(A) }
\newcommand{\ral}{\EuScript{R}_{\mathfrak{a}^l}(A) }
\newcommand{\rb}{\EuScript{R}_{\mathfrak{b}}(B) }

\newcommand{\eR}{\EuScript{R}}
\newcommand{\xb}{\mathbf{x}}
\newcommand{\ub}{\mathbf{u}}

\newcommand{\eG}{\EuScript{G}}

\newcommand{\Sc}{\mathcal{S} }
\newcommand{\rt}{\rightarrow}
\newcommand{\xar}{\longrightarrow}
\newcommand{\ov}{\overline}
\newcommand{\sub}{\subseteq}
\newcommand{\da}{\dagger}

\newcommand{\bX}{\mathbf{X} }
\newcommand{\bY}{\mathbf{Y} }

\newcommand{\Bcal}{\mathcal{B} }
\newcommand{\om}{\omega}
\newcommand{\dd}{*}

\newcommand{\wt}{\widetilde }

\newcommand{\Om}{\Omega}
\newcommand{\socle}{\operatorname{soc}}

\newcommand{\grade}{\operatorname{grade}}
\newcommand{\depth}{\operatorname{depth}}
\newcommand{\ann}{\operatorname{ann}}
\newcommand{\red}{\operatorname{red}}
\newcommand{\htt}{\operatorname{ht}}
\newcommand{\type}{\operatorname{type}}

\newcommand{\amp}{\operatorname{amp}}
\newcommand{\Hs}{\operatorname{\ ^* Hom}}
\newcommand{\Es}{\operatorname{\ ^* Ext}}

\newcommand{\Hom}{\operatorname{Hom}}
\newcommand{\Tor}{\operatorname{Tor}}
\newcommand{\Ext}{\operatorname{Ext}}

\newcommand{\K}{\mathcal{K}}
\theoremstyle{plain}

\newtheorem{thm}{Theorem}

\newtheorem{theorem}{Theorem}[section]
\newtheorem{corollary}[theorem]{Corollary}
\newtheorem{lemma}[theorem]{Lemma}
\newtheorem{proposition}[theorem]{Proposition}

\theoremstyle{definition}
\newtheorem{definition}[theorem]{Definition}
\newtheorem{defn}[thm]{Definition}
\newtheorem{s}[theorem]{}
\newtheorem{remark}[theorem]{Remark}
\newtheorem{remarkC}[thm]{Remark}
\newtheorem{example}[theorem]{Example}

\theoremstyle{remark}
\newtheorem{observation}[theorem]{Observation}
\newtheorem{question}[theorem]{Question}

\numberwithin{equation}{theorem}

\begin{document}

\title[CI-approximation]{Complete intersection Approximation, Dual Filtrations  \\ and Applications}
\author{Tony~J.~Puthenpurakal}
\date{\today}
\address{Department of Mathematics, Indian Institute of Technology, Bombay, Powai, Mumbai 400 076, India}

\email{tputhen@math.iitb.ac.in}
  \begin{abstract}
We give a two step method to study certain questions regarding associated graded module of a \CM \ (CM) module $M$ \wrt \ an $\m$-primary ideal $\A$ in a complete Noetherian local ring $(A,\m)$. The first step, we call it complete intersection approximation, enables us to reduce to the case when both $A$, $ \GA = \bigoplus_{n \geq 0} \A^n/\A^{n+1} $ are complete intersections and $M$ is a maximal CM $A$-module. The second step consists of analyzing the classical filtration $\{ \Hom_A(M,\A^n) \}_{\nZ}$ of the dual $\Hom_A(M,A)$. We give many applications of this point of view. For instance let $(A,\m)$ be equicharacteristic \& CM. Let $a(G_\A(A))$ be the $a$-invariant of $\GA$. We prove:
\begin{enumerate}
  \item   $a(G_\A(A)) = -\dim A$ \ff \ $\A$ is generated by a regular sequence.
  \item If $\A$ is integrally closed and  $a(G_\A(A)) = -\dim A + 1$ then $\A$ has minimal multiplicity.
\end{enumerate}
     We extend to modules a result of Ooishi relating symmetry of $h$-vectors. As another application
      we prove a conjecture of Itoh, if $A$ is a \CM \  local ring and
      $\A$  is a normal ideal with $e_3^\A(A) = 0$ then $G_\A(A)$ is \CM.
\end{abstract}
 \maketitle
\tableofcontents
\section{Introduction}
Let $(A,\m)$ be a Noetherian    local ring of dimension $d$,
 $\A$  an $\m$-primary ideal in $A$ and let $M$ be a \CM \ (CM)  module of dimension $r$. Set $G_{\A}(A) = \bigoplus_{n \ge 0} \A^n/\A^{n+1} $;  the
\textit{associated graded ring} of $A$ \wrt \ $\A$   and let $G_\A(M)= \bigoplus_{n \ge 0} \A^nM/\A^{n+1} M$ be the \textit{associated graded module} of $M$ \wrt \ $\A$.

In this paper we give a two step method to study certain questions regarding
$G_\A(M)$ when $A$ is complete.  Surprisingly this is also useful in the case of $M = A$. We call the first step \textit{complete intersection approximation}. This step enables us to reduce to the
case when $A$, $\GA$ are complete intersections and $M$ is a maximal \CM \ (MCM) $A$-module.\textit{ This reduction  surprised a few and attracted some
skepticism}. Nevertheless it is true.
The second step consists in studying the filtration
 $\{ \Hom_\A(M,\A^n) \}_{\nZ}$ of the dual $\Hom_A(M,A)$.  This filtration is
classical cf.  \cite[p.\ 12]{Ser}.  However it has not been used before in the study of blow-up algebra's (modules) of \CM \ rings (modules).

We begin by a general definition of our notion of  approximation.
Let $\mathcal{P}$ be a property of  Noetherian rings. For instance $\Pc =$ regular, complete intersection (CI), Gorenstein, \CM \ etc.
\begin{defn}
Let $(A,\m)$ be a Noetherian local ring and let $\A$ be a proper ideal
in $A$. We say $A$ admits a \emph{$\Pc$-approximation} \wrt \ $\A$ if there exists
a  local ring $(B,\n)$, an ideal $\B$ of $B$  and $\psi \colon B \rt A$,  a local ring homomorphism
with $\psi(\B)A =  \A$ such that the following \emph{three}
properties hold:
 \begin{enumerate}
\item
$A$ is a finitely generated $B$ module ( via $\psi$).
\item
$\dim A = \dim B$.
\item
$B$ and $\GB$  have property $\Pc$.
\end{enumerate}
 If the above conditions hold we say $[B,\n,\B,\psi]$ is a \emph{$\Pc$-approximation}
of $[A,\m,\A]$.
\end{defn}

\begin{remarkC}

\noindent $\bullet$  Regular approximations \textit{seems} to be  rare while Cohen-Macaulay approximations \emph{don't seem} to  have many applications.

  \noindent $\bullet$   For applications we will often insist that $\psi$, $\B$ satisfy some additional properties.
\end{remarkC}

When $A$ is complete we prove the following general result(see \ref{equiTT}).  Let $\spr(\A)$ denote the analytic spread of $\A$, see \ref{equimultiple-defn}.

\begin{thm}\label{introGAPP}
Let $(A,\m)$ be  a complete  Noetherian local ring and let $\A$ be a proper
 in $A$ such that $ \spr(\A) + \dim A/\A = \dim A$. Then
$A$  admits a  CI-approximation
\wrt \ $\A$.
\end{thm}

 Notice the hypothesis of Theorem \ref{introGAPP} is satisfied  if $\A$ is  $\m$-primary.
Another significant case when Theorem \ref{introGAPP} holds is when $A$ is equidimensional and $\A$ is equimultiple; see \cite[2.6]{HUO}.  For definition of an equimultiple ideal see \ref{equimultiple-defn}.

\begin{remarkC}
In most of our applications we only use that $A$ has a Gorenstein approximation. However in the proof of Itoh's conjecture we use the full power of CI-approximations.
\end{remarkC}

The proof of Theorem \ref{introGAPP} is fairly involved. The entire section
\ref{sectionGAPP} is devoted for a proof of it. The  following two special
cases are easy to
prove.
\begin{enumerate}
 \item
Let $(A,\m)$  be a  quotient of a regular local ring.  Then $[A,\m,\m]$ admits a CI approximation $[B,\n,\n, \phi]$ with $\phi$-onto. (see \ref{GorApprox-REG-QT}).
\item
 Let $(A,\m)$ be a complete
equicharacteristic local ring and let $\mathfrak{a}$ be an
$\m$-primary ideal.  Then $[A,\m,\mathfrak{a}]$ admits a CI
approximation $[B,\n,\mathfrak{b},\phi]$ with $\B = \n$ and $B/\n \cong A/\m$ and  $\mu(\n) =
\mu(\A)$. (see \ref{field}).
\end{enumerate}

Before proceeding further we need some notation.
Set $\A^n = A$ for $n \leq 0$.
Let $\ra = \bigoplus_{n \in \Z}\A^nt^n$  be the
\emph{extended Rees-algebra} of $A$
\wrt \ $\A$.
We consider $\ra$ as a subring of $A[t,t^{-1}]$. Set $\R(\A, M) =  \bigoplus_{\nZ} \A^nMt^n$.
 Clearly $\R(\A, M)$ is a finitely generated $\ra$-module
 respectively.

\begin{defn}
By the \textit{dual filtration} of $M^\dd = \Hom_A(M, A)$ \wrt \ $\A$ we mean the filtration
 $\{ \Hom_A(M,\A^n) \}_{\nZ}$.
\end{defn}

 We study the dual filtration in detail
when $(A,\m)$ is \textit{Gorenstein}, $\A$ is $\m$-primary, $M$ is MCM and $G_{\A}(A)$ is \textit{Gorenstein.}
Let $\eF = \{ \Hom_A(M,\A^n) \}_{\nZ}$ be the dual filtration
of $M^\dd$ \wrt \ $\A$. We prove that if $G_\A(M)$ is CM   then
$$G(\eF,M^{\dd}) \cong \Hs_{\GA}(G_{\A}(M),\GA); \quad \text{see \ref{corGor}}.    $$

Even when $G_\A(M)$ is not \CM \ we find the following relation:
Set $\red_\A(A) = $ reduction number of $A$ \wrt \ $\A$ (see \ref{redNO}), $a(G_\A(M)) =$  the $a$-invariant of $G_\A(M)$  (see \ref{a-invariant}), and let
$$ \alpha_{\mathfrak{a}} (M) = \min\{ n \mid \ G(\eF, M^\dd)_n \neq 0 \}. $$
We prove
\begin{equation*}
 a(G_\A(M)) \geq \red_\A(A) - \alpha_{\A}(M) - \dim A \quad \quad \text{(see \ref{a-iG})}.
\end{equation*}

\noindent  \textbf{Applications}

\noindent \textbf{I.} \textit{a-invariant of associated graded rings of $\m$-primary ideals}

 Let $\A$ be an $\m$-primary ideal in a
Noetherian local ring $(A,\m)$.
  It is easy to see that  $a(\GA) \geq -\dim A$.
We analyze the case when equality holds when $A$ is \CM.  If $\dim A = 1$, it follows immediately from a result
of T.Marley \cite[2.1(a)]{Mar-Proc} that $\A$ is prinicipal. When $d = 2$, it follows from Proposition \ref{Trung} and a result of Northcott \cite{North} that $\A$ is also
generated by regular sequence of length two.

We cannot use induction to analyze the case when $a(G_\A(A)) = -d$ since if $x$ is $A$-superficial \wrt \ $A$ then it
does not immediately follow that $a(G_{\A/(x)}(A/(x)) = - d + 1$

We prove that if $A$ is \textit{\CM} \ then
\[
 a(G_\m(A)) = - \dim A  \quad \text{\ff} \quad A \ \text{is regular} \ \ \text{(see \ref{regular-local}).}
\]
See \ref{counter-reg} for an example of \emph{a non-\CM }\ ring $A$ with  $a(G_\m(A)) = - \dim A$.

For arbitrary $\m$-primary ideals in a \CM \ local ring we unfortunately have to assume that
$A$ is \textit{equicharacteristic}. In this case we prove that
\[
a(\GA) = - \dim A \Longleftrightarrow \A \ \text{is generated by a regular sequence (see \ref{a-result})}.
\]
Next we try  to classify ideals such that $a(\GA) = - \dim A   + 1$. When $A$ is an
equicharacteristic \CM \ local ring and $\A$ is \textit{integrally closed} we prove
 \[
 a(\GA) = - \dim A + 1  \Longleftrightarrow \A \ \text{ has minimal multiplicity (see \ref{int-closed}).}
 \]
The result above need not hold if $\A$ is not integrally closed (see \ref{not-min-mult}).
When  $\dim A = 2$  we have been able to characterize $\m$-primary ideals $\A$ (not necessarily integrally closed) with $a(\GA) \leq - 1$.

\noindent \textbf{II.} \textit{Symmetric $h$-vectors.}

 Let $R = \bigoplus_{n \geq 0}R_n$ be a standard graded algebra over a field. For definition of  $h_R(z)$; the $h$-polynomial of $R$ see \ref{HilbertSeries-graded}. If $R$ is Gorenstein then  $h_R(z)$ is symmetric cf. \cite[4.4.6]{BH}.  The converse need not hold even if $R$ is \CM \ cf. \cite[4.4.7(b)]{BH}.  A celebrated theorem due to Stanley asserts that
if $R$ is \CM \ domain and $h_R(z)$ is symmetric then $R$ is Gorenstein
 (\cite[4.4]{Stanley-Adv}).

If $R = \GA$ the associated graded ring of an $\m$-primary ideal then Ooishi, \cite[1.6]{Oo-GOR}, proves that if
$A$ is Gorenstein, $\GA$ is \CM \  and $h_\A(z)$ is symmetric then $\GA$ is Gorenstein.
We generalize Ooishi's result as follows:

Let $(A,\m)$ be CM with a canonical module $\om_A$. Let $\A$ be an $\m$-primary ideal and let $M$ be a CM $A$-module. Set $c(M) = \dim A - \dim M$; $M^\da = \Ext^{c(M)}_A(M,\om_A)$ and $r=\red(\A,M)$. Assume
$\GA$, $G_\A(M)$ and $G_\A(M^\da)$ are CM. Let $\Omega_\A$ be the canonical module of $\GA$.  We prove (see Theorem \ref{symmetry})
$$  h_\A(M^\da,z) = z^rh(M,z^{-1})    \Longleftrightarrow     G_\A(M^\da) \cong \Es^{c(M)}_{\GA}(G_\A(M),\Om_\A) \ \text{ (up to a shift)}.   $$
The implication ($\Leftarrow$) follows from \cite[4.4.5]{BH}.  The  assertion $\Rightarrow$ is \textit{new} and  generalizes Ooishi's result.

In section \ref{determine-dual-MCM} we use of a MCM module $M$ to show that the dual filtration of a MCM module $M$ \wrt \ $\m$, over a Gorenstein local ring $A$ with $G_\m(A)$ Gorenstein, is a shift of the usual $\m$-adic filtration in the following cases.
\begin{enumerate}
 \item
$M$ is \textit{Ulrich} (i.e., $\deg h_M(z) = 0$; equivalently $e(M) = \mu(M)$).
\item
$\type(M) = e(M) -\mu(M)$ and $M$ has \textit{minimal multiplicity} and is \emph{not} Ulrich.
\end{enumerate}

\noindent \textbf{III.} \textit{$\m$-primary ideals $\A$ with $\mu(\A) = \dim A  + 1$}

\noindent Using Gorenstein approximation of $[A,\m,\A]$ when $A$ is complete and the structure of dual filtration of MCM modules over a hypersurface ring we prove (see \ref{mu-aE-d+1}):

Suppose $(A,\m)$ is an equicharacteristic Gorenstein  local ring of dimension $d \geq 0$ and $\A$ is an $\m$-primary ideal of $A$ with $\mu(\A) = d+1$. Then
$$ \GA \ \text{is Gorenstein \ff } \ h_\A(z) = \ell(A/\A)(1+z+\ldots+z^s) \ \text{for some} \ s\geq 1. $$
\noindent Here $\ell(-)$ denotes length.   The surprising thing is:   $h_\A(z)$ determines that $\GA$ is Gorenstein  \emph{without a priori assuming} $\GA$ is CM.

\noindent \textbf{IV.} \textit{Associated graded modules of the canonical
module}

\noindent Let $(A,\m)$ be CM  local ring with a canonical module $\om_A$.  Let
$\A$ be an $\m$-primary ideal \emph{such that $\GA$ is CM}.   Let $\Om_\A$ be the canonical module of $\GA$.   A natural question is when is $G_\A(\om_A)$ isomorphic to $\Om_\A$ up to a shift? Set $r = \red(\A, A)$.
Ooishi  proves that if $G_\m(A)$ is CM  then $G_\m(\om_A) \cong \Om_\m$ (up to a shift) \ff \ $G_\m(\om_A)$
is CM   and $h(\om_A,z) = z^{r} h(A,z^{-1})$; see  \cite[3.5]{Ooishi-cann}.

The hypothesis $G_\A(\om_A)$
is CM  is  difficult to verify. So there is a need to bypass this assumption. We first consider the case when $\mu(\A) = d + 1$. We prove
\[
G_{\mathfrak{a}} (\om_A) \cong \Om_{\A}^{A} \ \text{(up to a shift)} \Longleftrightarrow
h_{\mathfrak{a}}(A,z) = \ell(A/\mathfrak{a}) \big(1 + z + \ldots + z^{r}\big); \ \ \text{see \ref{mu-aE-d+1-cann} }.
\]

The assumption $G_\A(\om_A)$
is CM   \emph{holds}  \emph{automatically if } $\dim A = 0$. Surprisingly the following general result holds:
  Let $x_,\ldots,x_d$ be a maximal $A$-superficial sequence \wrt \ $\A$. Set $B =A/(\xb)$ and $\B = \A/(\xb)$. We prove (see \ref{cannonical-module-Mod}).
\[
\text{If} \ h(\om_B,z) = z^{r} h(B,z^{-1}) \quad  \text{then} \quad G_\A(\om_A) \cong \Om_\A \ \text{ up to a shift.}
\]
This result is quite practical. Using COCOA \cite{cocoa} \& MACAULAY \cite{M2}, we can determine whether $G_\A(\om_A)\cong \Om_\A$ up to a shift. See Example \ref{cann-example}.
Another application of \ref{cannonical-module-Mod} (see \ref{red2type2})  is: If $\red(A) = 2$  then

$G_\m(\om_A) \cong \Om_\m$(up to a shift) if and only if  \ $\type(A) = e_0(A) - h_1(A) - 1 $.

\noindent Ooishi  derives the same result essentially assuming $G_\m(\om_A) $ is CM, see \cite[4.4]{Ooishi-cann}.

\noindent \textbf{V.} \textit{ Lengths of duals of associated graded modules over Gorenstein Artin local rings.}

 Let $(A,\m)$ be a Gorenstein  Artin local ring, $\A$ an  ideal in $A$ and let $M$ be a finitely generated $A$-module. Set $G = \GA$ and $G(M) =G_\A(M)$. We prove
\begin{equation*}
\ell_{G} \left( G(M) \right )  \leq \ell_{G} \left( \Hs_{G}(G(M), G)\right )
\end{equation*}
This is surprising since $\GA$ need not be Gorenstein.

\noindent \textbf{VI.} \textit{Generalization of some results from CM  rings to CM modules.}

Let $(A,\m)$ be CM  local ring of dimension $d$ with infinite residue field and let $\A$ be an $\m$-primary ideal. Let
$\C $ be a minimal reduction  of $\A$. Set
$\delta = \sum_{n\geq 0}\ell(\A^{n+1}/\C\A^n)$. Then for $0\leq \delta \leq 2$
one has $\depth \GA \geq d-\delta$. This is due to
  Valabrega and Valla \cite[2.3]{VV} (for $\delta = 0$), Guerrieri \cite[3.2]{Gr94} (for $\delta = 1$) and Wang \cite[2.6,3.1]{Wang00} (for $\delta =2$).

We extend these results to CM  modules.   Using Gorenstein approximation we prove it by reducing it to the case of rings,
 see \ref{vv-guer-wang}.

\noindent \textbf{VII.} \textit{Relation between $e^{\mathfrak{a}}_{1} (M), e^{\A}_{0} (M)$ and $a(G_\A(M))$.}

Let $M$ be a CM  $A$ module of dimension $=2$.  Let $\A$
be an $\m$-primary ideal. For $i \geq 0$ let $e_{i}^{\A}(M)$ be the Hilbert-coefficients of $M$ \wrt \ $\A$. It can be easily proved that
\begin{equation}\label{reg2-Nee}
e_1^\A(M) \leq e_0^\A(A)(a(G_\A(A)) + 2)
\end{equation}
It is natural to investigate when equality holds occurs in \ref{reg2-Nee}. We prove that equality holds in \ref{reg2-Nee} if and only if $G_\A(M)$ is generalized \CM \ and $H^2(G_\A(A))^\vee$ is generalized Ulrich $G_\A(A)$-module.

\noindent  \textbf{VII.} \textit{ Proof of Itoh's conjecture}

It is well-known that the multiplicity of a Noetherian local ring $A$ is positive. When $A$ is \CM, \ Northcott \cite{North} proved that $e_1^\A(A) \geq 0$ with equality if and only if $\A$ is generated by a regular sequence. Narita showed that $e_2^\A(A) \geq 0$, see \cite{Nar}. Furthermore he gave an example which showed that $e_3^\A(A)$ can be negative.  Recall an ideal $\A$ is said to be normal if $\A^n$ is integrally closed for all $n \geq 1$. Itoh proved that if $\A$ is normal then $e_3^\A(A) \geq 0$. Furthermore he conjectured that if $e_3^A(A) = 0$ (and $\A$ is normal and $A$ is Gorenstein) then $G_\A(A)$ is \CM.  He proved it if $\A = \m$.  A major consequence
of our theory of complete intersection approximation, Dual filtration is a proof of Itoh's conjecture.
In fact we show more generally that if $A$ is \CM \ local and $\A$ is an $\m$-primary normal ideal with $e_3^\A(A) = 0$ then $G_\A(A)$ is \CM.

Here is an overview of the contents of this paper. In section two we introduce notation and discuss a few
preliminary facts that we need. In section 3 we prove the result on complete intersection approximations  when $\A = \m$ and $A$ is a quotient of a regular local ring.
We also prove it when
$A$ is complete, $\A$ is $\m$-primary  and contains a field. In section 4 we introduce dual filtrations. Furthermore application \textbf{V} is proved in this section. For applications \textbf{VI} and  \textbf{VII} see section 5.
In section  six we discuss on the initial degree of the dual filtration. We also relate it to the a-invariant of $G_{\A}(M)$.  For application \textbf{I} see section 7.
 Application \textbf{III} is discussed in   section 8. In section 9 we deal with symmetric $h$-vectors and prove generalization of Ooishi's result (i.e., application  \textbf{II}). In section 10 we use application \textbf{II}
 to get application \textbf{IV}. In section 11 we prove our result on complete intersection approximation of equimultiple ideals.
 In section twelve we discuss behaviour of dual filtration's \wrt \ superficial elements.
In section thirteen we discuss some preliminaries that we need to prove Itoh's conjecture
Finally in our last section we prove Itoh's conjecture.

\section{Notation and Preliminaries}\label{notation}
In this paper all rings are commutative Noetherian and all modules  (unless stated otherwise) are assumed to
be finitely generated. Let $A$ be a ring, $\A$ an ideal in $A$ and let $M$ be an $A$-module
 Recall an $\A$-\textit{filtration} $\eG = {\{M_n\}}_{n \in \Z}$ on $M$ is a
collection of submodules of $M$ with the properties
\begin{enumerate}
\item
$M_n \supset M_{n+1}$ for all $n \in \Z$.
\item
$M_n = M$ for all $n \ll 0$.
\item
$\A M_n \subseteq M_{n+1}$ for all $n \in \Z$.
\end{enumerate}
If
$\A M_n = M_{n+1}$ for all $n \gg 0$ then we say $\eG$ is $\A$-\emph{stable}.
We usually set $\eG_n = M_n$ for all $n \in \Z$.

\s \label{mod-N}
Let $\eG = {\{M_n\}}_{n \in \Z}$ be an $\A$-filtration on $M$ and let $N$ be a submodule of $M$. By the \emph{quotient filtration} on $M/N$ we mean the filtration   $\ov{\eG} = \{ (M_n + N)/N \}_{\nZ}$.
If $\eG$ is an $\A$-stable filtration on
$M$ then $\ov{\eG}$ is an $\A$-stable filtration on
$M/N$.\textit{ Usually} for us $N =\xb M$ for some $\xb = x_1,\ldots,x_s \in \A$.

\s If $\eG = \{M_n\}_{n \in \Z}$ is an $\A$ stable filtration on $M$ then  set
$\eR(\eG,M) = \bigoplus_{n \in \Z}M_nt^n$ the \emph{extended Rees-module} of $M$
\emph{\wrt }\ $\eG$.
 Notice that $\eR(\eG,M)$ is a finitely generated graded $\ra$-module. If $\eG$ is the usual $\A$-adic filtration then
set \\  $ \eR(\A, M) = \eR(\eG,M)$.

Set $G(\eG,M)=\bigoplus_{n \in \Z} M_n/M_{n+1}$, the \emph{associated graded
module }of $M$ \emph{\wrt} \ $\eG$. Notice $G(\eG,M)$ is a finitely generated graded module over
 $G_{\A}(A)$.
 Furthermore $\eR(\eG,M) /t^{-1}\eR(\eG,M) = G(\eG, M)$.

\s \label{shift-filt}
Let $\eG = {\{M_n\}}_{n \in \Z}$ be an $\A$-filtration on $M$ and let $s \in \Z$. By the $s$-\emph{th shift} of $\eG$, denoted by $\eG(s)$ we mean the filtration ${\{ \eG(s)_n\}}_{n \in \Z}$ where
$\eG(s)_n = \eG_{n+s}$. Clearly
$$G(\eG(s), M) = G(\eG, M)(s) \quad \text{and} \quad \R(\eG(s), M) = \R(\eG, M)(s). $$

\s All   filtration's in this paper $\eG = \{M_n\}_{n \in \Z}$ will be \textit{separated} i.e.,
 $\bigcap_{\nZ}M_n = \{0\}$.
This is automatic if $A$ is local,  $\A \neq A$ and $\eG$ is $\A$-stable.
If $m$ is a non-zero element of $M$ and if $j$ is the largest integer such that
$ m \in M_j$,
then we let $m^{*}_{\eG}$ denote the image of $m $ in $M_j  \setminus M_{j+ 1}$ and we call it the \textit{initial form} of $m$ \wrt \ $\eG$. If $\eG$ is clear from the context then we drop the subscript $\eG$.

The following result is well-known and easy to prove.
\begin{proposition}\label{filtEQUALgraded}
Let $(A,\m)$ be a  local ring, $\A$ an $\m$-primary ideal and let
$M$ be an $A$-module. Let  $\eF$ be an $\A$-stable  filtration of $M$ \wrt\ $\A$.
\TFAE
\begin{enumerate}[\rm (i)]
	\item
	$G(\eF, M) \cong G(M)$ up to a shift.
	\item
	$\eF $ is the $\A$-adic filtration up to a shift. \qed
\end{enumerate}
\end{proposition}

The following result enables us to prove a filtration is $\A$-adic in some cases.
\begin{lemma}
\label{a-stble=a-adic}
Let $(A,\m)$ be local, $\A$ an $\m$-primary ideal and let $M$ be a \CM \ $A$-module of dimension $r \geq 1$. Let $\eF = \{ M_n \}_{\nZ}$ be an
$\A$-stable filtration on $M$ such that
\begin{enumerate}[\rm (1)]
	\item
	 $M_n = M$ for $n \leq 0$ and $M_1 \neq M$.
\item
$G(\eF, M)$ is \CM.
\end{enumerate}
 Let $\xb = x_1,\ldots,x_s $ be an $M$-superficial sequence \wrt \ $M$ such that $\xb^* = x_1^*,\ldots,x_s^*$
 is $G(\eF, M)$-regular. Set $B = A/(\xb)$, $\B = \A/(\xb)$ and $N = M/\xb M$.  Let $\ov{\eF} = \{(M_n + \xb M)/\xb M \}_{\nZ}$ be the quotient filtration
 on $N$. If $\ov{\eF}$ is the $\B$-adic filtration on $N$ then $\eF$ is the $\A$-adic filtration on $M$.
\end{lemma}
\begin{proof}
Since
$M_0 = M$ we get $\A^iM \subseteq M_i$ for all $i \geq 1$. We prove by induction on $n$ that $M_n = \A^nM$ for $n \geq 1$.

\noindent\textit{The case $n =1$.}

$\ov{\eF}$ is $\B$-adic. So
\[
\frac{M_1+ \xb M}{\xb M} = \frac{\A M+ \xb M}{\xb M}
\]
Thus $M_1 + \xb M = \A M + \xb M$. But $\xb M \subseteq \A M \subseteq M_1$.
Therefore $M_1 = \A M$.

Assume the result for $n =p$ and prove for $n =p+1$.
By hypothesis on $\ov{\eF}$,
\begin{equation*}
\frac{M_{p+1}+ \xb M}{\xb M} = \frac{\A^{p+1} M+ \xb M}{\xb M}. \tag{*}
\end{equation*}
Notice
\begin{equation*}
\frac{M_{p+1}+ \xb M}{\xb M} \cong \frac{M_{p+1}}{\xb M \cap M_{p+1}} \quad \text{and} \quad
  \frac{\A^{p+1} M+ \xb M}{\xb M} \cong  \frac{\A^{p+1} M}{\xb M \cap \A^{p+1} M}.
 \tag{**}
\end{equation*}
Observe that
\begin{align*}
\xb \A^p M \subseteq \A^{p+1}M \cap \xb M &\subseteq M_{p+1} \cap \xb M \\
                         &= \xb M_p \quad \ \text{\cite[2.3]{CortZar}} \\
                            &= \xb \A^p M \quad \text{by induction hypothesis}
\end{align*}
So $ \xb M \cap M_{p+1} =  \xb \A^p M =  \A^{p+1}M \cap \xb M $. The result follows from (*) and (**).
\end{proof}
\s \emph{Veronese functor:}. Let $\eF = \{M_n \}_{n\in \Z}$ be an $\A$- stable filtration. For $l \geq 1$ set $\eF^{<l>}= \{M_{nd} \}_{n \in \Z}$. Notice $\eF^{<l>}$ is an $\A^l$-stable filtration. We also have an $\ral$-isomorphism $\eR(\eF,M)^{<l>} \cong \eR(\eF^{<l>}, M)$.

\s For definition and  basic properties of superficial sequences see \cite[p.\ 86-87]{Pu1}.

\s\label{AtoA'} \textbf{Flat Base Change:} In our paper we do many flat changes of rings.
 The general set up we consider is
as follows:

 Let $\phi \colon (A,\m) \rt (A',\m')$ be a flat local ring homomorphism
 with $\m A' = \m'$. Set $\A' = \A A'$ and if
 $N$ is an $A$-module set $N' = N\otimes A'$. Set $k = A/\m$ and $k' = A'/\m'$.

 \textbf{Properties preserved during our flat base-changes:}

\begin{enumerate}[\rm (1)]
\item
$\ell_A(N) = \ell_{A'}(N')$. So
 $H_{\A}(M,n) = H_{\A'}(M',n)$ for all $n \geq 0$.
\item
$\dim M = \dim M'$ and  $\grade(K,M) = \grade(KA',M')$ for any ideal $K$ of $A$.
\item
$\depth G_{\A}(M) = \depth G_{\A'}(M')$.
\item
For each $i \geq 0$ we have $\Tor^{A}_{i}(k,k) \otimes_A A'  \cong \Tor^{A'}_{i}(k', k')$. In particular $A$ is regular local \ff \ $A'$ is.

\end{enumerate}

\textbf{Specific flat Base-changes:}

\begin{enumerate}[\rm (a)]
\item
$A' = A[X]_S$ where $S =  A[X]\setminus \m A[X]$.
The maximal ideal of $A'$ is $\n = \m A'$.
The residue
field of $A'$ is $k' = k(X)$. Notice that $k'$ is infinite.
\item
$A' =  \widehat{A}$  the completion of $A$ \wrt \ the maximal ideal.
\item
Applying (a) and (b) successively ensures that $A'$ is complete with $k'$ infinite.
\item
$A' = A[X_1,\ldots,X_n]_S$ where $S =  A[X_1,\ldots,X_n]\setminus \m A[X_1,\ldots,X_n]$.
The maximal ideal of $A'$ is $\n = \m A'$.
The residue
field of $A'$ is $l = k(X_1,\ldots,X_n)$. Notice that if $I$ is integrally closed then $I'$ is
also integrally closed. Recall an ideal $K$ is said to be asymptotically normal if  $K^n$ is integrally closed for all $n \gg 0$.
When $\dim A \geq 2$ and $I$ is asymptotically normal, say $I = (a_1,\cdots,c_n)$ then in \cite[Corollary 2]{Ciu} it is proved that
the element $y = \sum_{i = 1}^{n} a_iX_i$ in $A'$ is $I'$-superficial and  when $A$ is \CM \
the $A'/(y)$ ideal $J = I'/(y)$ is asymptotically normal. We call $A'$ a general extension of $A$.
\end{enumerate}

\s \label{lying-above} We will need the following result. Let $\phi \colon (B,\n) \rt (A.\m)$ be a local map such that $A$ is finitely generated $B$-module (via $\phi$). Consider
$R = B[X_1,\cdots,X_n]$ and $S = A[X_1,\cdots. X_n]$ and let $\psi \colon R \rt S$ be the map induced by $\phi$. Note $S$ is a finite $R$-module via $\psi$.
Let $P = \n R$ and $Q = \m S$. Then $S_P = S_Q$. To see this note that $Q$ is the only prime in $S$ lying over $P$. The result follows from an exercise problem
in \cite[Exercise 9.1]{Ma}.

\s\label{equimultiple-defn}  The \emph{fiber cone} of $\A$ is the
$k$-algebra
$F(\A) = \bigoplus_{n\geq 0}\A^n/\m \A^{n}.$
Set $\spr(\A) = \dim F(\A)$,
the \emph{analytic spread} of $\A$.
 Recall $\dim A \geq \spr(\A) \geq \htt \A$.
We say $\A$ is\textit{ equimultiple} if
$\spr(\A) = \htt \A$. Clearly $\m$-primary ideals are equimultiple.
The number $\spr(\A) - \htt \A$ is called the analytic deviation of $\A$.

\s \label{a-invariant}
 Assume $\A$ is $\m$-primary and $\dim M = r$.
 Then
    $H^{i}_{ G_{\mathfrak{a}}(A)_ +}
\left( G_{\mathfrak{a}} (M)\right) = 0$ for $i > r$.
Recall that
the $a$-\emph{invariant} of $G_\mathfrak{a} (M)$ is
\[
 a\left( G_\mathfrak{a} (M) \right) = \max \{n \mid H^{r}_{\GA_+}(G_{\A}(M))_n \neq 0 \}.
\]
For $i \geq 0$ set
$a^{\mathfrak{a}}_{i} (M) = \max\{ n \mid  H^{i}_{ G_{\mathfrak{a}}(A)_ +}
\left( G_{\mathfrak{a}} (M)\right)_n \neq 0 \}.$

\s \label{redNO} We assume $k = A/\m$ is infinite. Let $\C = (x_1,\ldots,x_l)$ be  a minimal reduction of
$\A$ \wrt \ $M$.  We denote by  $\red_\C(\A, M) := \min \{ n  \mid \  \C \A^n M= \A^{n+1}M \} $
the \textit{reduction number} of $\A$ \wrt \ $\C$ and $M$. Let
 $$\red(\A, M) = \min \{ \red_\C(\A, M) \mid \C
\text{\ is  a minimal reduction of \ }  \A \}$$
 be the \textit{reduction number } of $M$ \wrt \ $\A$. Set
$\red_\A(A) = \red(\A, A)$.

\s \label{HilbertSeries-graded} Let $R = \bigoplus_{n\geq 0}R_n$ be a standard algebra over an Artin local ring $(R_0,\m_0)$ i.e., there exists $x_1,\ldots x_s \in R_1$ such that $R = R_0[x_1,\ldots,x_s]$.
Let $M = \bigoplus_{\nZ}M_n$ be an $R$-module.
By Hilbert-Serre theorem
\[
 H(M,z) = \sum_{\nZ}\ell_{R_0} (M_n)z^n  = \frac{h_M(z)}{(1-z)^{\dim M}} \quad \text{here } \ h_M(z) \in \Z[z,z^{-1}] \ \text{\&} \  h_R(1) \neq 0.
\]
We call $H_M(z)$ the \textit{Hilbert series }of $M$ and
$$h_M(z)= h_{-p}z^{-p}+ \ldots + h_{-1}z^{-1}+ h_0 + h_1z+ \ldots + h_sz^s \quad  \text{  (with $h_s \neq 0$ and $h_{-p} \neq 0$)} $$
the \textit{$h$-polynomial}
 of $M$. Notice  $h_M(z) \in \Z[z]$ \ff \ $M_n = 0$ for all $n < 0$.

\s\label{symmetric-h-polynomial}
 We say $h_M(z)$ is \emph{symmetric} if $h_{s-i} = h_{-p+i}$ for $i = 0,1,\ldots,s+p$.

\s \label{coefficents}
Let $R$, $M$, $h_M(z)$ be as in \ref{HilbertSeries-graded}. We also assume that
$M_n =0$ for $n < 0$. If $f$ is a polynomial we use $f^{(i)}$ to denote the $i$-th formal derivative of $f$. The integers $e_i(M)= h_M^{(i)}/i! $ for all $i\geq 0$ are called the \emph{Hilbert coefficients} of $M$.  The number $e_0(M)$ is also called the \textit{multiplicity }of $M$.
\s \label{HilbertSeries-local}
Let $(A,\m)$ be a local ring,  $\A$ an  $\m$-primary ideal and let $M$ be an $A$-module.
Set
$ H_\A(M,z) = H(G_\A(M)(z),z)$ and $h_{\A}( M , z) = h_{G_\A(M)}(z)$.  For
$i \geq 0$ set $e_{i}^{\A}(M) = e_i(G_\A(M))$. If $\A = \m$ we usually drop the "label" $\m$ and write $h(M,z)= h_{\m}(A,z)$, $e_{i}(M) = e_{i}^{\m}(M)$ and $\red(A) = \red(\m, A)$.

\section{CI-Approximation: Some special cases}\label{Gorapp-special}
In this section we prove that if $(A,\m)$ is a quotient of a regular local ring then
$[A,\m,\m]$ has a CI-approximation $[B,\n,\n, \phi]$ with $\phi$ onto. In Theorem \ref{equi} we  state our main result regarding CI-approximations. This will be proved in section 10. Finally in Theorem \ref{field} we prove a result regarding CI-approximation of an $\m$-primary ideal $\A$ in a complete  equicharacteristic
local ring $(A,\m)$.
The following Lemma  regarding annihilators is crucial.

\begin{lemma}
\label{annihilator}
Let $M$ be an $A$-module and let $\A$ be an ideal. Set $G = \GA$.
 Suppose there exists
 $\xi_1,\ldots,\xi_s \in \A$ such that $\xi_{i}^{*} \in \ann_{G} G_{\A}(M)$
for each $i$. Also assume that
 $\xi_{1}^{*},\ldots,\xi_{s}^{*}$ is a $G$-regular sequence.
 Then there exists $u_1,\ldots,u_s \in \A$ such that
\begin{enumerate}[\rm (1)]
\item
$u_1,\ldots,u_s$ is an $A$-regular sequence.
\item
$u_i \in \ann M$ for each $i = 1,\ldots,s$.
\item
For each $i \in \{1,\ldots,s \}$ there exists $n_i \geq 1$ such
 that $u_{i}^{*} = (\xi_{i}^{*})^{n_i}$.
\item
 $u_{1}^{*},\ldots,u_{s}^{*}  \in  \ann_G G_\A(M).$
\item
 $u_{1}^{*},\ldots,u_{s}^{*} $ is a  $G$-regular sequence.
\end{enumerate}
\end{lemma}
\begin{proof}
Suppose we have constructed $u_1,\ldots,u_s$ satisfying (2) and (3) then
 $u_1,\ldots,u_s$ satisfy all the remaining properties. This can be seen as follows:

(4) follows from (3); since $\xi_{i}^{*} \in \ann_{G} G_{\A}(M)$ for each $i$.

(1) and (5). As  $\xi_{1}^{*},\ldots,\xi_{s}^{*}$ is a $G$-regular sequence
 we also get
$(\xi_{1}^{*})^{n_1},\ldots,(\xi_{s}^{*})^{n_s}$ is a $G$-regular sequence
\cite[16.1]{Ma}.
Thus $u_{1}^{*},\ldots,u_{s}^{*} $ is a  $G$-regular sequence. It follows from   \cite[2.4]{VV} that
$u_1,\ldots,u_s$ is an $A$-regular sequence.
Thus it suffices to show there exists $u_1,\ldots,u_s$ satisfying (2) and (3).

Fix $i \in \{1,\ldots,s \}$.
Set $\xi = \xi_i$.
  Say $\xi \in \A^r \setminus \A^{r+1}$ for some $r \geq 1$.
Since $\xi^* \in \ann_G G_{\A}(M)$ we have $\xi M \subseteq \A^{r+1}M$.
Set $\q = \A^{r+1}$.
By the \emph{determinant trick}, \cite[2.1]{Ma}, there exists a monic polynomial $f(X) \in A[X]$ such that
\begin{align*}
f(X) &= X^{n} + a_1X^{n-1} + \ldots + a_{n-1}X + a_n \quad \text{with} \quad a_i \in \q^i, \text{\ for} \ i = 1,\ldots n \\
\text{and} \ u = f(\xi) &= \xi^{n} + a_1\xi^{n-1}+ \ldots +  a_{n-1}\xi + a_n \in \ann M.
\end{align*}
As $\xi^*$ is $G$-regular we have $\xi^n \in \A^{nr} \setminus \A^{nr+1}$.
 However  for each $i \geq 1$ we
 have
\[
a_i\xi^{n-i} \in \q^i\A^{(n-i)r} = \A^{i(r+1)}\A^{nr - ir} = \A^{nr + i} \subseteq
 \A^{nr +1}.
\]
Thus $u^* = (\xi^{n})^{*} = (\xi^*)^n$ (since $\xi^*$ is $G$-regular).
Set $u_i = u$ and $n_i = n$.
\end{proof}

The following Corollary is useful.
\begin{corollary}\label{redTOmaxDIM}
 Let $(A,\m)$ be \CM \ local ring, $\A$ an $\m$-primary ideal and let $M$ be an $A$-module. Set
$c = \dim A - \dim M$.  If $\GA$ is \CM \ then there exists $u_1,\ldots, u_c \in \A$ such that
\begin{enumerate}[\rm (1)]
 \item
$ u_1,\ldots, u_c \in \ann_A M $.
 \item
 $ u_{1}^*,\ldots, u_{c}^* \in \ann_{\GA} G_\A(M)$.
 \item
 $ u_1,\ldots, u_c $ is an $A$-regular sequence.
\item
$ u_{1}^*,\ldots, u_{c}^*$ is a $\GA$-regular sequence.
\item
 $u_1 t,\ldots, u_c t \in \ann_{\ra } \R(M).$
\item
$u_1 t,\ldots, u_c t$ is a $\ra$-regular sequence.
\end{enumerate}
\end{corollary}
\begin{proof}
 Since $\GA$ is \CM, we have that
\begin{align*}
  \grade  \ann_{\GA} G_\A(M)  &= \htt \ann_{\GA} G_\A(M) \\
                           &=  \dim \GA  - \dim G_{\A}(M),  \quad \text{since $\GA $ is *-local} \\
                           &= \dim A - \dim M .
\end{align*}
Therefore (1), (2), (3) and (4) follow from Lemma \ref{annihilator}.
The assertion (5) is clear.

(6). Since $t^{-1}$ is $\ra$-regular and $u_1^*,\ldots, u_c^*$ is $\GA = \ra/t^{-1}\ra$-regular sequence it follows that $t^{-1}, u_1t,\ldots, u_ct$ is  a $\ra$-regular sequence. Notice
$\ra$ is a *-local. It follows that $ u_1t,\ldots, u_ct$ is a $\ra$-regular sequence.
\end{proof}

We state our general result regarding CI-Approximations
\begin{theorem}\label{equi}
Let $(A,\m)$ be a complete
 local ring and let $\A$ be a proper  ideal in $A$ with $\dim A/\A + \spr(\A)
 = \dim A$. Then
$A$ has a CI-approximation \wrt \ $\A$. Furthermore if $\A = (x_1,\ldots,x_m)$ then we may
choose
a CI-approximation $[B,\n,\B,\varphi]$ of $[A,\m,\A]$ such that there exists $u_1,\ldots,u_m \in \B$ with $\varphi(u_i) = x_i$ for  $i = 1,\ldots, m$.
\end{theorem}
Theorem \ref{equi} is proved in section 10.
\begin{remark}\label{equiREM}
 The hypothesis of the theorem holds  when $\A$ is $\m$-primary. It is also satisfied when $\A$ is equimultiple and
$A$ is equidimensional \cite[34.5]{Nag}.  Our hypothesis ensures that $\GA$ has a homogeneous system of parameters,
  \cite[2.6]{HUO}.
\end{remark}

For the case when $\A = \m$ and $A$ is a quotient of a regular local ring $T$ we have:
\begin{theorem}\label{GorApprox-REG-QT}
 Let $(A,\m)$  be a local ring such that $A$ is a quotient of a regular local ring $(T,\tf)$. Then
$[A,\m,\m]$ has a CI-approximation $[B,\n,\n, \phi]$ with $\phi$-onto.
\end{theorem}
\begin{proof}
 Set $k = A/\m = T/\tf$. Let $\psi \colon T \rt A$ be the quotient map. We consider $A$ as  a
$T$-module. So $ \GA = G_{\tf}(A)$. Set $c = \dim T - \dim A = \dim \GT - \dim \GA$.
Its well-known  $\GT \cong k[X_1,\ldots,X_n]$; where $n = \dim T$. In particular $\GT$ is
\CM. Let $u_1,\ldots u_c$ be as in Corollary \ref{redTOmaxDIM}.

Set $(B,\n) = (T/(\ub),\tf/(\ub))$ and let $\phi \colon B \rt A$ be the map induced by
$\psi$. So $\phi$ is onto. Clearly
$[B,\n,\n, \phi]$ is a CI-approximation  of $[A,\m,\m]$.
\end{proof}

\noindent For  $\m$-primary ideals in a complete  equicharacteristic complete local ring we prove:
\begin{theorem}\label{field}
 Let $(A,\m)$ be a complete
equicharacteristic local ring and let $\mathfrak{a}$ be an
$\m$-primary ideal.  Then $[A,\m,\mathfrak{a}]$ has a CI-approximation $[B,\n,\mathfrak{b},\phi]$ with $\B = \n$ and $B/\n \cong A/\m$ and  $\mu(\n) =
\mu(\A)$.
\end{theorem}
\begin{proof}
  $A$  contains its residue field $k = A/\m.$; see \cite[28.3]{Ma}.  Let
$\mathfrak{a} = (x_1, \ldots, x_n).$  Set $S = k [|X_1,\ldots X_n|]$
and let $\psi : S \rightarrow A$ be the natural map which sends
$X_i$ to $x_i$ for each $i.$
Since $A/(\underline{X}) A = A/\mathfrak{a}$ has finite
length (as an $S$-module) it follows that $A$ is a finitely
generated as an S-module, see \cite[8.4]{Ma}.  Set $\eta = (X_1, \ldots
X_n).$ Notice $\psi (\eta)A = \mathfrak{a}$ and $G_{\eta} (S) = k
[X^{*}_{1}, \ldots X^{*}_{n}]$ the polynomial ring in $n$-variables.

If $\dim A = \dim S$ then set $[B,\n,\B,\phi] = [S,\eta,\eta,\psi].$
Otherwise set $c = \dim S-\dim A = \dim G_n (S) - \dim
G_{\mathfrak{a}} (A).$
Let $u_1,\ldots u_c$ be as in Corollary \ref{redTOmaxDIM}.
Set $B = S/(\ub)$, $\n = \eta/(\ub)$ and $\B = \n$. Clearly $G_{\n} (B) = G_{\eta}
(S)/(\ub^*)$ is CI.
 The map $\psi$ induces $\phi : B \rightarrow A.$  It can
be easily checked that $[B,\n,\n,\phi]$ is a CI-approximation of $[A,\m,\mathfrak{a}]$.
Finally by our construction its clear that $\mu(\n) = \mu(\A)$.
\end{proof}

%End 1111111111111111111111111111111111111111111111111111
\section{A classical filtration of the dual}\label{section-classical}
Let $M$ be an $A$-module.
The following  filtration of the dual of $M$, i.e, $M^{\dd} = \Hom_A(M, A)$ is classical cf.  \cite[p.\ 12]{Ser}.
Set
\[
M_{n}^{\dd} = \{ f \in M^{\dd} \mid f(M) \subseteq \A^n \} \cong \Hom_A(M,\A^n).
\]
It is easily verified that   $\eF_M = \{ M_{n}^{\dd} \}_{n \in \Z}$ is an $\A$-stable filtration on $M^{\dd}$.
We call $\eF_M$ to be the \emph{dual-filtration} of $M$ \wrt $\A$.
Set $\R = \R(\A) = \bigoplus_{\nZ}\A^nt^n$  and
$\R(M) = \R(\A,M) = \bigoplus_{\nZ}\A^nMt^n$. Let $f \in  M_{n}^{\dd}$.  We show that $f$
induces a\textit{ homogeneous} $\R$-\textit{linear map} $\hat{f}$ \textit{of degree} $n$ \textit{from} $\R(M)$ to $\R$; (see \ref{mapAUG}).
Theorem  \ref{main}  shows that $\Psi_M$ is a $\R$-linear isomorphism.
\begin{align*}
\Psi_M \colon \R(\eF,M^{\dd}) &\xar \Hs_\R(\R(M),\R ) \\
f &\mapsto \f
\end{align*}
It is easily seen,  see \ref{initialobs},  that $\Psi_M$ induces a natural map
$$\Phi_M \colon G(\eF,M^{\dd}) \xar \Hs_{\GA}(G_{\A}(M),\GA). $$
We prove $\Phi_M$ is injective and as a consequence deduce application \textbf{V} as stated in the introduction.
In Corollary \ref{corGor} we give a sufficient condition for $\Phi_M$ to be an isomorphism.

The following result is well-known.
\begin{proposition}
\label{stablefilt}
Let $(A,\m)$ be local, $\A$ an ideal in $A$ and let $M$ an  $A$-module.
Then $\A \Hom_{A}(M, \A^n) = \Hom_A(M,\A^{n+1})$ for all $n \gg 0$. \qed
\end{proposition}

\s \label{mapAUG} Let $f \in M_{n}^{\dd}$.

\noindent\textbf{Claim:} $f$ induces a homogeneous $\R$-linear map $\hat{f}$ of degree $n$ from $\R(M)$ to $\R$.

Since $f(M) \subseteq \A^n$, for each $j \geq 0 $ we have $f(\A^jM) \subseteq \A^{n+j}$. Furthermore for $j < 0 $ as $\A^jM = M$ and notice that $f(M) \subset \A^n \subset \A^{n+j}$.
Thus
\begin{equation*}
f(\A^jM) \subseteq \A^{n+j} \quad \text{for each $j \in \Z$}. \tag{$*$}
\end{equation*}
  This enables us to define
\begin{align*}
\hat{f} \colon \R(M) &\xar \R \\
\sum_{j \in \Z} m_jt^j &\mapsto \sum_{j \in \Z} f(m_j)t^{n+j}.
\end{align*}

Next we prove that $\hat{f}$ is $\R$-linear.
Clearly $\hat{f}$ is $A$-linear. Notice
\begin{align*}
 \f\left( (x_jt^j)\bullet(m_it^i) \right) &= \f(x_jm_it^{i+j}) = f(x_jm_i)t^{i+j+n} = x_jt^jf(m_i)t^{i+n}  \\
(x_jt^j)\bullet \f(m_it^i) &= x_jt^jf(m_i)t^{i+n}.
\end{align*}
Thus $\f$ is $\R$-linear.

\s \label{main-1} We define a map
\begin{align*}
\Psi_M \colon \R(\eF,M^{\dd}) &\xar \Hs_\R(\R(M),\R ) \\
f &\mapsto \f
\end{align*}

\begin{theorem}\label{main}
$\Psi_M \colon \R(\eF,M^{\dd}) \xar  ^*\Hom_{\R}(\R(M), \R)$ is a $\R$-linear isomorphism.
\end{theorem}
\begin{proof}
If $a \in \A^j$ and $f \in M_{i}^{\dd}$ then
$af \in M^{\dd}_{i+j}$. Check  that $(at^i)\bullet \hat{f} = \hat{af}$.
Thus $\Psi_M$ is $\R$-linear.
Clearly  if $\hat{f} = 0$ then $f = 0$. Thus $\Psi_M$ is injective.

We show $\Psi_M$ is surjective. Let $g \in \Hs_\R(\R(M),\R )_n$. We write $g$ as
\[
g\left(\sum_{j \in \Z}m_jt^j\right)  = \sum_{j \in \Z}g_j(m_j)t^{n+j}.
\]
For each $j \in \Z$ the map $g_j \colon \A^jM \rt \A^{n+j}$ is $A$-linear.
Define $f = i\circ g_0$ where $i \colon \A^n \rt A$ is the inclusion map. Clearly $f \in M_{n}^{\dd}$.

\textit{Claim:} $\hat{f} = g$.

Let $m_jt^j \in \A^jMt^j$.

Case 1. $j = 0$.

Set $m = m_0$. Notice
$$g(m t^0) = g_0(m)t^n = \hat{f}(m t^0).$$
The last equality above holds since $\hat{f}$ is $\R$-linear.

\noindent\textit{In the next two cases we use Case 1 and the fact that  $\hat{f}$ is $\R$-linear.}

Case 2. $j< 0$.

Notice $m_jt^j =  t^j \bullet m_jt^{0}$. So
$$g(m_jt^j) = t^j\bullet g(m_jt^0) =  t^j\bullet \hat{f}(m_jt^0) = \hat{f}(m_jt^j).$$

Case 3. $j >0$.

Set $m_j = \sum_{l = 1}^{s}u_{jl}n_l$ where $u_{jl} \in \A^j$ and $n_l \in M$.
Fix $l$.
Notice
$$ u_{jl}n_lt^j = u_{jl}t^j\bullet n_lt^0.$$
Therefore for $l = 1,\ldots, s$,
$$g(u_{jl}n_lt^j) = u_{jl}t^j\bullet g(n_l t^0) = u_{jl}t^j \bullet \hat{f}(n_lt^0) = \hat{f}(u_{jl}n_lt^j).$$
 Note the last equality above is since $\hat{f}$ is $\R$-linear. So we have
$$ g(m_jt^j) = g\left(\sum_{l = 1}^{s}u_{jl}n_l\right)
           = \sum_{l =1}^{s}g(u_{jl}n_l)
           = \sum_{l = 1}^{s}\hat{f}(u_{jl}n_l)
           = \hat{f}(m_jt^j). $$
  Again   the last equality above holds since $\hat{f}$ is $\R$-linear.
\end{proof}

\begin{observation}\label{initialobs}
$f \in M_{n}^{\dd}$ induces $\hat{f} \colon\R(M) \rt \R$ which is homogeneous of degree  $n$. So $\hat{f}$ induces a map $\tilde{f}\colon G_{\A}(M) \rt \GA$ which is also
homogeneous of degree $n$. Clearly $\tilde{f} = 0$ \ff \
$f \in M_{n+1}^{\dd}$. So we have
\begin{align*}
\Phi_M \colon G(\eF,M^{\dd}) &\xar \Hs_{\GA}(G_{\A}(M),\GA) \\
f + M_{n+1}^{\dd} &\mapsto \tilde{f}. \\
\text{Clearly} \quad \Phi_M &= \Psi_M\otimes \frac{\R}{t^{-1}\R}.
\end{align*}
\end{observation}

\begin{corollary}\label{basicCor}
Set $G = \GA$.
There is an exact sequence
\[
0 \xar G(\eF,M^{\dd}) \xrightarrow{\Phi_M} \Hs_{G}(G_{\A}(M),G) \xar \Es^{1}_{\R}(\R(M),\R )(+1)
\]
\end{corollary}
\begin{proof}
The map $ 0\rt \R(+1) \xrightarrow{t^{-1}} \R \rt G \rt 0$ induces the exact sequence
\begin{align*}
0&\xar \Hs_{\R}(\R(M),\R)(+1) \xrightarrow{t^{-1}} \Hs_{\R}(\R(M),\R) \xar \Hs_{\R}(\R(M),G)\\
\ &\xar \Es_{\R}^{1}(\R(M),\R)(+1)
\end{align*}
Note that $\Hs_{\R}(\R(M),G) \cong \Hs_{G}(G_{\A}(M),G)$. The result follows by using
Theorem \ref{main} and Observation \ref{initialobs}.
\end{proof}
A consequence of Corollary \ref{basicCor}
is Application \textbf{V}.
\begin{corollary}
\label{App5}
Let $(A,\m)$ be a Artin Gorenstein local ring, $\A$ a proper ideal in $A$ and let $M$ be a finitely generated $A$-module. Then
$$  \ell(G_\A(M)) \leq \ell \left( \Hs_{G_{\A}(A)}(G_{\A}(M), \GA)    \right).  $$
\end{corollary}
\begin{proof}
Set $G = \GA$ and  $G(M) = G_\A(M)$.  Notice that
\begin{align*}
\ell (G_\A(M)) &= e_0(G_\A(M)) = \ell(M) \\
\ell (G_\eF(M^*)) &= e_0(G_\A(M^*))  = \ell(M^*).
\end{align*}
Since $A$ is self-injective,  $\ell(M) = \ell(M^*)$. The result follows from
\ref{basicCor}.
\end{proof}
Another important consequence of  Corollary \ref{basicCor} is the following:
\begin{corollary}\label{corGor}
Let $(A,\m)$ be a Gorenstein local ring and let $\A$ be an  ideal such that $\GA$ is  Gorenstein. Let $M$ be a maximal \CM \ $A$-module with $G_{\A}(M)$-\CM. Then
$$ G(\eF,M^{\dd}) \cong \Hs_{\GA}(G_{\A}(M),\GA).$$
\end{corollary}
\begin{proof}
Notice $\R$ is also a Gorenstein ring and $\R(M)$ is maximal \CM. It follows that
$\Es_{\R}^{1}(\R(M),\R) = 0$. The result follows from Corollary \ref{basicCor}.
\end{proof}

\s \label{mod-sup-gg} In the theory of Hilbert functions over local rings the notion of superficial elements plays an important role. In our case, first we assume for convenience that  $(A,\m)$ is also Gorenstein and  $M$ is a $MCM$ A-module. Let $\xb = x_1,\ldots, x_r$ be a $M \bigoplus A$ superficial sequence. Set $B = A/(\xb )$, $\B = \A/(\bX)$,  $N = M/\xb M$ and
$$
 \eF_M = \left\{\Hom_A (M ,
\mathfrak{a}^n)\right\}_{\nZ}  \quad \text{and} \quad \eF_N = \left\{\Hom_B (M ,\mathfrak{n}^n)\right\}_{\nZ}.
$$
We may ask when does
\begin{equation} \label{sup-eqn}
\frac{G (\eF_M , M^\dd)}{\xb^* G(\eF_M , M^\dd)} \cong G (\eF_N ,
N^\dd).
\end{equation}
We prove
\begin{theorem}\label{modFilt}[with hypothesis as in \ref{mod-sup-gg}]
If  $\GA$ is Gorenstein and $G_{\mathfrak{a}}(M)$ is \CM \ then \ref{sup-eqn} holds.
\end{theorem}
See Theorem \ref{modFiltX}(2) for a proof of
Theorem \ref{modFilt}.

We now discuss dual filtration's and the Veronese functor. We show
\begin{proposition}
  \label{dual-Ver-prop}
  Let $(A,\m)$ be  local, $\A$ an ideal in $A$ and let $M$ be an $A$-module. Let $\eF$ be the dual filtration of $M$ \wrt $\A$. Then for all $l\geq 1$;  $\eF^{<l>}$ is the dual filtration of $M$ \wrt \ $\A^l$.
\end{proposition}
\begin{proof}
Fix $l \geq 1$. Let $\eG$ be the dual filtration of $M$ \wrt  \ $\A^l$.
  We note that $\eF_{nl} = \Hom_A(M, \A^{nl}) = \eG_{n}$. The result follows.
\end{proof}
\section{Some preliminary applications of Gorenstein approximation}\label{section-preliminary}
In this section we give give  Applications \textbf{VI} and \textbf{VII} as stated in the introduction.
First we generalize to modules
some results of  Valabrega and Valla\cite[2.3]{VV}, Guerrieri \cite[3.2]{Gr94} and Wang \cite[2.6,3.1]{Wang00}.  If
 $\dim M = 2$ and $\mathfrak{a}$ an ideal of definition for
$M$, then it is easy to prove
$ e^{\mathfrak{a}}_{1} (M) \leq e^{\mathfrak{a}}_{0} (M)(a_2(G_\A(M)) + 2)$.
We give a complete classification of when equality holds.

\begin{theorem} \label{vv-guer-wang}[Appl. \textbf{VI} ]
 Let $(A,\m)$ be a local ring with infinite residue field,
$M$ a Cohen-Macaulay A-module of dimension $r$ and $\mathfrak{a}$ an ideal of
definition for $M.$  Let $\C = (x_1,\ldots x_r)$ be a minimal
reduction of $\A$ \wrt \  $M$.
Set $$\delta =  \sum^{\infty}_{n = 0} \ell \left(\frac{\mathfrak{a}^{n +
1} M \cap \C M}{\C \mathfrak{a}^{n} M}\right).$$
  If
 $\delta = 0,1,2$
then $\depth G_\mathfrak{a} (M) \geq d - \delta.$
\end{theorem}
\begin{proof}
 We may go mod $\ann M.$  Thus we
may assume $\dim M = \dim A $ and $\A$ is $\m$-primary. Since $M$ is a faithful $A$-module it can be easily checked that $\C$ is a minimal reduction of $\A$ \wrt \ $A$.
  Our hypothesis nor conclusion
change under completion so we may assume $A$ is complete.

Let $[B,\n,\B,\varphi]$ be a Gorenstein approximation of
$[A,\m,\mathfrak{a}].$  We do base change and consider $M$ as a
$B$-module. Notice $G_{\B}(M) = G_{\A}(M)$. Let $y_i \in \B$ be such that $\varphi(y_i) = x_i$.  Set
$R = B \propto M$, the idealization of $M$ and consider the ideal $\q = \B \propto M$. Then
$G_{\q}(R) = G_{\B} (B) \propto G_{\mathfrak{a}} (M) (- 1).$

As $G_{\B} (B)$ is \CM, $ \depth G_{\q} (R) = \depth G_{\A} (M)$.
Set $z_i = (y_i,0)$ for $i = 1,\ldots,r.$
Notice  $\q^{i} = \B^{i} \propto \mathfrak{a}^{i - 1} M$ for $i \geq 0$.
$$
\text{Therefore} \quad \q^i \cap (\mathbf{z}) = (\B^i \cap (\mathbf{y}) ,
\mathfrak{a}^{i - 1} M \cap \C M) \quad \text{and}
$$
$$
\frac{\q^i \cap (\mathbf{z})}{\q^{i - 1} (\mathbf{z})} = \left(0 ,
\frac{\mathfrak{a}^{i - 1} M \cap \C M}{\mathfrak{a}\C^{i -2}M}
\right).
$$
The result now follow from the case of rings.
\end{proof}

The previous result did not use the fact the Gorenstein property
of $G_{\B}(B)$.  All we used was that $G_{\B} (B)$ is \CM.
The next result uses the fact that $G_{\B} (B)$ is
Gorenstein. However we need
the following  elementary result .
\begin{proposition}\label{Trung}
Let $R = \bigoplus_{R_n}$ be a standard graded Noetherian ring with $R_0$ Artin local.
Let $E$ be a graded $R$-module of dimension two.
Then
\[
e_1(E) \leq e_0(E)(a_2(E) + 2).
\]
\end{proposition}
\begin{proof}[Sketch]
We may assume that $H^0_{R_+}(E) = 0$.  Set $t = a_2 + 1$. By the Grothendieck-Serre formula
\[
H_E(t) - \{ e_0(t+1) - e_1 \} = - \ell(H^1(M)_t)
\]
Hence $e_1(E)  \leq e_0(E)(a_2 + 2)$
\end{proof}
It is important to understand when equality holds above. In general we cannot say much. However if $M$ is \CM \ of dimension two over a local ring $(A,\m)$ then we can give a characterization when equality holds above. Before proceeding we need the following notation:
If $D$ is a graded Artin $G_\A(A)$-module then $D^\vee$ denotes the dual of $D$ \wrt \
to the injective hull of $k$ as a $G_\A(A)$-module (note that if $\A$ is $\m$-primary then $G_\A(A)$ is *-complete (see \cite[3.6.16]{BH}). Furthermore a finitely generated $C$ over $G_\A(A)$ is said to generalized Ulrich if the $h$-polynomial of $C$ is of the form $h_C(z) = a z^l$ for some $l \in \Z$.

We give application \textbf{VII} of our notion of Gorenstein approximation.
\begin{theorem}\label{dim2e1}
 Let $(A,\m)$ be a
 local
ring, $\mathfrak{a}$  an $\m$-primary ideal in $A$ and  $M$ an $A$-module with $\dim M = 2$.
Then the following are equivalent:
\begin{enumerate}[\rm (i)]
\item
$
e^{\mathfrak{a}}_{1} (M) = e^{\mathfrak{a}}_{0} (M) \left(a^{\mathfrak{a}}_{2} (M) + 2\right).
$
\item
$G_\A(M)$ is generalized \CM \ and $H^2(G_\A(A))^\vee$ is Generalized Ulrich $G_\A(A)$-module.
\end{enumerate}
\end{theorem}
\begin{proof}
We may assume $A$ is
complete, $\mathfrak{a}$ is $\m$-primary and $\dim M = \dim A.$  Let
$[B,\n,\B,\varphi]$ be a Gorenstein approximation of
$[A,\m,\mathfrak{a}]$.
Notice
$$
a^{\B}_{i}(M) = a^{\A}_{i}(M) \;\mbox{and}\;
e^{\mathfrak{b}}_{i} (M) = e^{\mathfrak{a}}_{i} (M) \ \text{for all} \ i.
$$
Thus we may assume $(A,\m)$ is Gorenstein local and
$G_{\mathfrak{a}} (A)$ is Gorenstein. We may further assume $A$ has an infinite
residue field. This we do.
 Set
$$G = G_{\A}(A),\quad   D = H^{2}_{G_+}  \left( G_{\mathfrak{a}}(M)\right)  = \bigoplus_{n \leq \alpha} D_n \quad\text{where}\ \  \alpha =
a^{\mathfrak{a}}_{2} (M).$$
 Set $(-)^{\vee} = \Hs_G (-,H^2_{G_+}(G))$ the Matlis dual functor over $G$.
Notice $D^{\vee} = \bigoplus_{n \geq -\alpha} D_{n}^{\vee}$ where $D^{\vee}_{i} = \Hom_{A/\A}(D_{-i}, E_{A/\A}(k))$, and
$E_{A/\A}(k)$ is the injective hull of $k$ as an $A/\A$-module. By Matlis duality $D^{\vee}$ is a finitely generated
$G$-module.

By Local duality we have
$$
D^{\vee} = H^2 (G_{\mathfrak{a}} (M))^{\vee} \simeq \Hs_G(G_{\mathfrak{a}}(M), G(a))
$$
(Here $a = \red_{\mathfrak{a}} (A) - 2$, the $a$-invariant of
$G$).

{\bf Claim  1:} $D^{\vee}$ is a Cohen-Macaulay $G$-module of dimension $2$.\\
\noindent{\bf Proof of Claim 1:}
By \cite[17.1.10]{BS} it follows that the function $n \mapsto \ell(D^{\vee}_{n})$ is polynomial of degree $1$. So
$\dim D^\vee = 2$.
 Let $F_{1} \rightarrow F_{0} \rightarrow
G_{\mathfrak{a}} (M) \rightarrow 0$ be a free presentation of
$G_{\mathfrak{a}} (M)$ as a $G$-module.
Notice
$$
0 \rt \Hs_{G} (G_{\mathfrak{a}} (M) ,G) \rightarrow \Hs_{G} (F_0,
G) \rightarrow \Hs_{G} (F_1,G)
$$
As $G$ is \CM \  of $\dim 2$ we get Claim 1.

{\bf Claim 2:}  $ \ell(D_n) = C_0 (\alpha - n + 1) -
C_1$ for all $n << 0$ where  $C_0 > 0$ and $C_1 \geq 0$. \\
{\bf Proof of Claim 2:} $D^{\vee} = \bigoplus_{m \geq - \alpha}
D^{\vee}_{m}$. Notice $ \ell(D^{\vee}_{m}) = \ell(D_{-m})$ for all $m$. Set
Notice $D^{\vee} (-\alpha)_j = 0$ for $j<0$.
So $D^{\vee} (-\alpha)$ has Hilbert series
$$
\frac{h_0 + h_1z + \ldots \alpha h_s z^z}{(1-z)^2}.
$$
As $D^{\vee} (-\alpha)$ is \CM \  we have $h_i \geq 0$ for all $i$.   Thus there exists  $C_0 > 0$ and $C_1 \geq 0$
 such that $ \ell(D^{\vee} (-\alpha)_{n}) = C_0 (n + 1) - C_1$ for all
$n \gg 0$. Therefore
\[
\ell (D_n) = \ell(D^{\vee}_{-n}) =  C_0 (-n + \alpha + 1) - C_1  \quad \text{for all $n \ll 0$.}
\]
This proves Claim 2.

By a result due to Serre cf. \cite[4.4.3]{BH}
$$
H(G_{\A} (M), n) - P(G_{\A }(M), n) =
\sum^{2}_{i = 0} (-1)^i \ell(H^i(G_{\A }(M))_n).
$$
  Clearly $H^0  (G_{\mathfrak{a}} (M)))_n = 0$ for $n < 0.$ Notice $\ell(H^1 (G_{\A} (M))_n) = \xi$, a constant  for all $n \ll 0$, see \cite[17.1.9]{BSh}.
So we get
$$
-\left[e^{\mathfrak{a}}_{0} (M) (n + 1) - e^{\mathfrak{a}}_{1}
(M)\right] = C_0 (\alpha + 1 - n) - C_1 - \xi
$$
Therefore $C_0 = e^{\mathfrak{a}}_{0} (M).$
Comparing  coefficients we get
$$
e^{\mathfrak{a}}_{1} (M) = e^{\mathfrak{a}}_{0} (M)
(\alpha + 2) - \xi - C_1.
$$
The result follows.
\end{proof}

\section{The initial degree of the dual filtration and the  $a$-invariant }\label{section-initial-degree}
  Let $(A,\m)$ be a Gorenstein local ring, $\A$ an equimultiple ideal in $A$ and let $M$ be an MCM  $A$-module. Let $\alpha_\A(M)$ be the \emph{order} of $M$ \wrt \ $\A$ (see \ref{alpha-setup}). We
prove that $\alpha_\A(M) \leq \red_\A(A)$.
If $\A$ is $\m$-primary then we prove that
$$ a(G_\A(M)) \geq \red_\A( A ) - \alpha_{\A}(M) - \dim A. $$
We also prove that $\alpha_\m(M) = \red(A)$ \ff \ $M$
is an Ulrich module.

\s \label{alpha-setup} Let $(A,\m)$ be a local ring, $\A$ an ideal in $A$ and $M$ and $A$-module. Set $M^\dd = \Hom_A(M,A)$. Let  $\eF =  \{M^\dd_n = \Hom_A(M,\A^n) \}_{\nZ}$  be the dual filtration of $M^*$ \wrt  \ $\A$. If $f\neq 0 \in M^\dd$ set $\alpha_{\mathfrak{a}} (f) =
\max \{n \mid f (M) \subseteq \mathfrak{a}^n\}$. Set
$$\alpha_{\mathfrak{a}} (M) = \min\{\alpha_{\mathfrak{a}} (f) \mid f \in M^\dd , f \neq 0 \}.$$
 \[
 \text{Notice} \ \ \alpha_{\mathfrak{a}} (M) =  \max \{ n \mid M^\dd_n = M^\dd \} = \min \{ n \mid \ G(\eF, M^\dd)_n \neq 0 \}.
                           \]
\s \label{alpha-map}  Let $x_1, \ldots , x_r \in \A$ be a sequence. Set $(B , \n)
= (A/(\xb ) , \m/(\xb))$ and $ N = M/\xb M$.
If $f \in \Hom_A (M , A)$ define $\psi (f) : N \rightarrow B$, by
 $$  \psi (f) (m + \xb M) = {f (m)} + \xb A.$$
One can check readily $\psi (f) \in \Hom_B (N , B)$.
Define
\begin{align*}
 \Upsilon_{\bX} \colon \Hom_A (M , A) &\rightarrow \Hom_B(N , B) \\
\Upsilon_{\bX}(f) &= \psi(f).
\end{align*}
\begin{proposition}\label{alphaSUP} Let $(A,\m)$ be a Gorenstein  local ring and let $M$ be a MCM $ A$-module.
 If $x_1, \ldots , x_r \in \A$ is an $A$-regular sequence and  then $\Upsilon_{\bX}$
 is onto.
\end{proposition}
\begin{proof}
 It  suffices to prove for $r = 1.$  The exact sequence
$
0 \rightarrow A \xrightarrow{x} A \xrightarrow{\pi} B
\rightarrow 0
$
yields $0 \rightarrow M^\dd \xrightarrow{x} M^\dd \rightarrow \Hom_B (N
, B) \rightarrow \Ext^{1}_A(M , A)$.
Notice $\Ext^{1}_A(M , A) = 0$.
\begin{align*}
\text{The map: } \quad  \quad \quad  \Hom (M , A) &\xrightarrow{\Hom (M , \pi)} \Hom_B (N ,
B) \\
f &\mapsto \ov{f}
\end{align*}
 is nothing but $\Upsilon_{x_1}.$
Thus $\Upsilon_{x_1}$ is onto.
\end{proof}

For equimultiple ideals we have
\begin{corollary}\label{boundALPHA}[with hypothesis as in \ref{alphaSUP}]
If  $\A$ is an equimultiple ideal then
 $$
\alpha_\mathfrak{a} (M) \leq \red_\mathfrak{a} (A).
$$
\end{corollary}
\begin{proof}
 Let $J = (x_1,\ldots,x_s)$ be a minimal reduction of $\A$  such that
$$
r = \red_{\mathfrak{a}}(A) = \min \{n \mid \mathfrak{a}^{n + 1} = J\mathfrak{a}^n  \}.
$$
By  hypothesis $x_1,\ldots,x_s$ is a regular sequence.
Consider $\Upsilon_{\bX} \colon  \Hom_A (M , A) \rightarrow
\Hom_{A/J} (M/JM , A/JM)$ as defined earlier.  \\ If $\alpha (M) > r$
then note that $\psi (f) = 0$ for each $f \in M^\dd$.  By \ref{alphaSUP}, $\Upsilon$ is onto.
So  $\Hom_{A/J} (M/JM , A/J) = 0$, a contradiction.
\end{proof}

\s Set $r = \red(\A) $ and $d = \dim A$.
Recall that if $\GA$ is Gorenstein then  $\Om_\A = G (r-d)$ is the canonical module of $\GA$.
 Using graded local duality we get
\begin{equation}\label{a-m-primary}
  a(G_\A(M)) = -\min\{n \mid \Hs_{\GA}\left(G_\A(M), \GA( r - d )\right)_n \neq 0 \}.
\end{equation}

The following result  gives a lower bound on $a(G_\A(M)$ in terms of  $\alpha_{\A}(M)$.

\begin{proposition}\label{a-iG}
 Let $(A, \m)$ be a $d$-dimensional Gorenstein local ring, $\A$ an $\m$-primary ideal with $\GA$ Gorenstein and let $M$ be a MCM $A$-module.
$$ a(G_\A(M)) \geq \red_\A( A ) - \alpha_{\A}(M) - d. $$
\end{proposition}
\begin{proof}
 Set $G (M) = G_\mathfrak{a} (M) , G = G_\mathfrak{a} (A) , r =
\red_\mathfrak{a} (A)$ and $M^\dd = \Hom_A (M , A)$.   By \ref{basicCor} we have
an inclusion
$
0 \rightarrow G (\eF , M^\dd) \rightarrow \Hom_G (G (M) , G).
$
So
\begin{equation*}
 0 \rightarrow G (\eF , M^\dd) (r - d) \rightarrow \Hom_G (G (M) , G (r
- d)). \tag{*}
\end{equation*}
$$ \text{Notice} \ \ G (\eF , M^\dd) ((r - d))_{\alpha_{\mathfrak{a}} (M) - (r -d)} = G (\eF , M^\dd)_{\alpha_{\mathfrak{a}} (M)} \neq 0.$$
Looking at initial degrees of sequence $(*)$ and  using \ref{a-m-primary} we get
 $$
 - a(G_\A(M)) \leq    \alpha_{\mathfrak{a}} (M) - (r - d).
$$
The result follows.
\end{proof}

\begin{remark}\label{alpha-Ainvar}
If $\A$ is $\m$-primary then by Proposition \ref{a-iG} and \ref{boundALPHA} it follows that
$a(\GM) \geq - \dim A$. But this also follows easily from other
 arguments.
 \end{remark}

 \noindent \textbf{The case when} $\A = \m$. \\
\noindent \emph{Notation:}  Set $G(-) = G_\m(-)$ and $\alpha (-) = \alpha_{\m} (-).$ If we have to specify the ring then we write $\alpha_A(-)$.
\emph{Clearly $\alpha(M) = 0$ \ff  \  $M$ has a free summand.}
\begin{lemma}\label{ulrichLEM}
Let $(A ,\m)$ be a Gorenstein local ring with $G(A)$ Gorenstein.
 Let $x \in \m\setminus\m^2$ be $M \bigoplus A$-superficial
\wrt \  $\m.$ Then
$$
\alpha_A (M) \leq \alpha_{A/(x)} (M/xM).
$$
\end{lemma}
\begin{proof}
Set $N = M/xM$ and $(B,\n) = (A/(x), \m/(x))$. By Proposition \ref{alphaSUP}
$$
\Upsilon : \Hom_A (M , A) \rightarrow \Hom_{B} (N , B)
$$
is onto.  Set $\alpha = \alpha_A (M).$  Then $f (M) \subseteq
\m^\alpha$ for all $f \in M^\dd$. Therefore
$
 \Upsilon (f) (N) \subseteq \n^\alpha$
for all $f \in M^\dd$.
Since $\Upsilon$ is onto, we get $g (N) \subseteq \n^\alpha$
for all $g \in N^\dd$.   Thus $\alpha_{B} (N) \geq \alpha =
\alpha_A (M).$
\end{proof}
\s {\bf Reduction to dimension zero :}

Let $\xb = x_1 \ldots x_d$ be a \emph{maximal} $M \oplus A$-superficial sequence.  Set $(B , \n) = (A/(\xb) ,
\m/(\xb))$ and $N = M/\xb M$.
\begin{remark}\label{ulrichREM}
Recall an MCM $\A$-module $M$ is called \emph{Ulrich }if $e(M) = \mu(M)$.
 \begin{enumerate}[\rm (a)]
  \item
$M$ is Ulrich $ \Longleftrightarrow $ $N$ is Ulrich.  Notice $N$ is Ulrich $ \Longleftrightarrow $  $N \simeq k^{\mu(N)}.$
\item
 Since $G(A)$ is \CM \ we have $\red (A) = \red (B)$
\item
If $\alpha_A(M) = \red(A)$ then  $\alpha_B (N) = \red (B).$ This is so, since
$$
\red(A) = \alpha_A(M) \leq \alpha_B(N) \leq \red(B) \quad \text{(inequalities due to \ref{ulrichLEM} \& \ref{boundALPHA})}.
$$
\item
If $\alpha_B (N) = \red (B)$ then $N = k^s$ for some $s \geq 1$.  \emph{The converse also holds} since $B$ is Gorenstein.
 \end{enumerate}
\end{remark}
\begin{proposition}\label{baby-Ulrich}
 Let $(A ,\m)$ be a Gorenstein local ring with $G_{\m} (A)$ Gorenstein. Also assume that $A$ has infinite residue field.
Let $M$ be a MCM $A$-module.
\TFAE
\begin{enumerate}[\rm (i)]
 \item
$\alpha (M) = \red (A)$.
\item
$M$ is an Ulrich A-module.
\item
$a(G(M)) = - \dim A$.
\end{enumerate}
\end{proposition}
\begin{proof}
Set $r = \red(A)$, $\alpha = \alpha_A(M)$ and $d = \dim A$.
 Let $\xb = x_1, \ldots ,x_d$ be a maximal $M \oplus A$-superficial sequence.  Set $(B , \n) = (A/(\xb),
\m/(\xb))$ and $N = M/\xb M.$
(i) $\Rightarrow$ (ii):
If  $\alpha (M) = \red (A) $ then by \ref{ulrichREM}(c), $ \alpha (N) = \red (B) = r $.
 We have $\alpha (N) = r.$  So
$N^\dd = N^\dd_r = \Hom_B (N , \n^r).$
Since $B$ is Gorenstein  with $\n^r \neq 0$ and $\n^j = 0$ for $j >r$ we get $\n^r \cong k$.
So $N^\dd = \Hom_A (N , k) = k^{\mu (N)}$.
Therefore
$ N \cong N^{\dd \dd} = k^{\mu (N)}.$  So $N$ is Ulrich.
 By \ref{ulrichREM}(a) $ M$ is  an Ulrich
$A$-module.

(ii) $\Rightarrow$ (iii): Trivial.

(iii) $\Rightarrow$ (i):  By Proposition \ref{a-iG} we get
\[
 a(G(M)) \geq r - \alpha -d.
\]
So we obtain $\alpha \geq r$. But $\alpha \leq r$ always (see \ref{boundALPHA}). So $\alpha = r$.
\end{proof}

%End 22222222222222222222222222222222222222222222222222222222222222222222222222222222222222222222222222222

%33333333333333333333333333333333333333333333333333333333333333333333333333333333
\section{$a$-invariant: Borderline cases}\label{section-boderline}
Let $(A,\m)$ be CM.  It is well-known that  $a(G_\m(A)) \geq - \dim A$. We
prove that $a(G_\m(A)) = - \dim A$ \ff \ $A$ is regular local (see \ref{regular-local}). If $A$ is \emph{equicharacteristic}
we prove
\begin{itemize}
 \item $a(\GA) = - \dim A$ \ff \ $\A$ is generated by a regular sequence.
\item (assume   $\A$ is also integrally closed).  $a(\GA) = - \dim A + 1$ \  \ff \
$\A$ has minimal multiplicity.
\end{itemize}

\s \emph{Discussion:} We can easily show that if $\dim A = 1, 2$, then  $a(\GA) = - \dim A$ \ff \ $\A$ is generated by a regular sequence. When $d = 1$ it easily follows from a result of  Marley \cite[2.1(a)]{Mar-Proc}. When $d = 2$ it follows from  \ref{Trung} that $e_1^\A(A) = 0$. It follows from a result of Northcott \cite{North} that $\A$ is generated by a regular sequence. We cannot use induction to analyze the case when $a(G_\A(A)) = -d$ sinice if $x$ is $A$-superficial \wrt \ $A$ then it
does not immediately follow that $a(G_{\A/(x)}(A/(x)) = - d + 1$

We begin our investigations when $\A = \m$ the maximal ideal of $A$.
\begin{theorem}\label{regular-local}
 Let $(A,\m)$ be a \CM \ local ring.
\TFAE
\begin{enumerate}[\rm (i)]
 \item
$A$ is regular local.
\item
$a(G(A)) = - \dim A$.
\end{enumerate}
\end{theorem}
The example below shows that the assumption,  $A$ is CM in \ref{regular-local}, is crucial.
\begin{example}\label{counter-reg}
Let $(A,\m)$ be a regular local ring of dimension $ d > 0$. Let $M$ be an $A$-module with
$\dim M < d$. Let  $R = A \propto M$, the idealization of $M$. Notice
$R$ is local with maximal ideal $\n = \m \propto M$. Also $\dim R = \dim A = d$. Furthermore
$\depth R = \min \{ \depth M, \depth A \} < d$. So
$R$ is \emph{not} CM.

\textit{Claim :} $a(G_\n(R)) = - d$.

Set $T = G_\m(A), N = G_\m(M)(-1)$ and
$S = G_\n(R)$.
Notice $ S = T \propto N$,
the idealization of $N$.
  Since $N^2 =0$ in $S$ we have that $H^i_{S_+}(-) = H^i_{T_+S}(-)$. Furthermore as $S$
is a finite $T$-module we get that $H^i_{T_+}(S) = H^i_{T_+ S}(S)$ as $T$-modules. Since $ S = T \oplus N$
as $T$-modules we get that $H^i_{T_+}(S) = H^i_{T_+}(T) \oplus H^i_{T_+}(N)$ for each $i \geq 0$.
Since $\dim N = \dim M < d$ we get that $H^d_{T_+}(S) = H^d_{T_+}(T)$.
Thus $H^d_{S_+}(S) = H^d_{T_+}(T)$ as $T$-modules. The result follows.
\end{example}

\begin{proof}[Proof of Theorem \ref{regular-local}]
 The assertion (i) $\implies$ (ii) is clear. To prove the converse
we may assume that $A$ is complete with infinite residue field $k$; see \ref{AtoA'}(c) and \ref{AtoA'}(5).

Since $A$ is complete, $A$ is the quotient of a regular local ring. So by \ref{GorApprox-REG-QT}
$[A,\m, \m]$ has a Gorenstein  approximation $[B,\n, \n, \phi]$ with
$\phi$ onto. We consider $A$ as a $B$-module. Notice $G_\n(A) = G_\m(A)$ as a $G_\n(B)$-module. It is also be
easily seen that for each $i \geq 0$ we have $H^i_{G_\n(B)_+}(G_\n(A)) \cong H^i_{G_\m(A)_+}(G_\m(A))$ as a $G_\n(B)$-module.

Since $\dim B = \dim A = d$(say), we get $a(G_\n(A)) = - \dim B$. By Theorem \ref{baby-Ulrich} we get that $A$ is Ulrich as a
$B$-module. Let $J = (u_1, \ldots,u_d)$ be a minimal reduction of $A$ \wrt \ $\n$. Set $x_i = \phi(u_i)$ for
$i = 1,\ldots, d$. As $A$ is an Ulrich $B$-module we get $JA = \n A$. It follows that
$\m = (x_1,\ldots, x_d)$. So $A$ is regular local.
\end{proof}

We prove an analogue of \ref{regular-local} for $\m$-primary ideals. Unfortunately we have to assume that $A$ is equicharacteristic (i.e., it contains a field).
\begin{theorem}\label{a-result}
 Let $(A , \m)$ be an equicharacteristic $CM$
local ring of $\dim d \geq 1.$  Let ${\mathfrak{a}}$ be
an $\m$-primary ideal.  Then $a(G_\mathfrak{a} (A)) =
-\dim A  \Longleftrightarrow  \A $ is a parameter ideal.
\end{theorem}
\begin{proof}
We may assume, without loss of any generality that  $A$ is complete and has an infinite residue field.  As $A$ is equicharacteristic we choose
a Gorenstein approximation $[B , \n , \B , \varphi]$ of $[A , \m ,{\mathfrak{a}}]$ with $\B = \n$.   Notice $G_{\n}(A) = \GA$ and notice
$H^{i}_{G_{\n}(B)_+}(\GA) = H^{i}_{\GA_+}(\GA)$. If $a(G_\mathfrak{a} (A)) = -d$
then $A$ is Ulrich as an $B$-module.
So $\n A = J A$ where $J = (y_1,\ldots,y_d)$ is a minimal reduction of $\n$ \wrt \ $A$. But $\n A = \A A$. Set
$x_i = \varphi(y_i)$. We get $\A = (x_1,\ldots, x_d)$. It follows that $\A$ is a parameter ideal.

Conversely if $\A$ is a parameter ideal then its clear that $a(\GA) = -d$.
\end{proof}

We might also ask what happens when $a(G_{\A}(A)) = -d +1$.  We show
\begin{theorem} \label{int-closed}
Let $(A , \m)$ be an equicharacteristic $CM$
local ring of $\dim d \geq 1.$  Let ${\mathfrak{a}}$ be
an $\m$-primary ideal.
If $\A$ has minimal multiplicity then  $a(G_{\A}(A)) = -d +1$. The converse holds if
\begin{enumerate}[\rm (1)]
  \item  $\A$ is integrally closed.
  \item $\dim A = 1$.
\end{enumerate}
 \end{theorem}
\begin{proof}
If $\A$ has minimal multiplicity then $G_{\A}(A)$ is \CM. The description of Hilbert series of $G_{\A}(A)$ gives
$a(G_{\A}(A)) = -d +1$.
To prove the converse we may assume that $A$ is complete with an infinite residue field.

\emph{Case (1): $\A$ is integrally closed. } \\
\noindent As $A$ is equicharacteristic we choose
a Gorenstein approximation $[B , \n , \B , \varphi]$ of $[A , \m ,
{\mathfrak{a}}]$ with $\B = \n$.
From Lemma \ref{ulrichLEM} we have
$$
\alpha_{\n} (A) \leq \alpha_{\n/(\xb)} (A/\xb A)
$$
where $\xb = x_1,\ldots,x_d$ is a maximal  $A \bigoplus B$ superficial sequence \wrt \ $\n$.

By \ref{a-iG} we have $\alpha_{\A}(A) \geq r -1$ where $r = \red B = \red B/(\xb)$.
Therefore $\alpha_{\n/(\xb)} (A/\xb A)  \geq r-1$. It follows that $\ov{\n}^2\ov{A} = 0$.
Thus $\n^2 A \subseteq \xb A$. Let $y_i = \varphi(x_i)$ for all $i$. Set $\q = (y_1,\ldots,y_d)$. Then
we have $\A^2  \subseteq J$. Since $\A$ is integrally closed we have $\A^2 = \q \A$.
So $\A$ has minimal multiplicity.

\emph{Case (2): $\dim A = 1$. } \\
\noindent By hypothesis $a_1(\GA) = a(\GA) =0$. If $\GA$ is CM then we have nothing to show.
We assert that $\GA$ is CM. If not then
by a result of Marley \cite[2.1(a)]{Mar-Proc}  we get $a_0(\GA) < a_1(\GA )$ a contradiction, since $H^0_{\GA_+}(\GA)$ is concentrated in non-negative degrees.
\end{proof}
When $\dim A =2$ there exists $\m$-primary ideals $\A$ which do not have minimal multiplicity but have  $a(G_{\A}(A)) = -1$ (see \ref{not-min-mult}).
In fact when  $\dim A =2$  we give the following characterization of $\m$-primary ideals $\A$ with $a(\GA) \leq - 1$.
\begin{proposition} \label{dim2-hoa} Let $(A,\m)$ be a \CM \ local ring with $\dim A = 2$. Let $\A$ be an $\m$-primary ideal.
\TFAE
\begin{enumerate}[\rm (i)]
\item $a(\GA) \leq - 1$.
  \item $\red_{\A^n}(A) = 1$ for all $n\gg 0$.
  \item $e_{2}^{\A}(A) = 0$.
\end{enumerate}
\end{proposition}
\begin{proof}
By \cite[2.6]{Hoa} we get $a(G_{\A}(A)) < 0$ \ff \ $\red(\A^n) = \dim A -1 $ for all $n \gg 0$.
Thus the assertions (i) and  (ii) are equivalent. By  \cite{Nar}, we get that (ii) and (iii) are equivalent.
\end{proof}

The example below was constructed by Marley \cite[4.1]{Mar}  for a  different purpose.
\begin{example}\label{not-min-mult}
 Let $A = k[X,Y]_{(X,Y)}$ and let $\A = (X^7, X^6Y, XY^6, Y^7) $. Using COCOA \cite{cocoa} one verifies that
\[
h_\A(A,z) =  38 + 3z + 3z^2 + 3z^3 + 3z^4 + 3z^5 - 4z^6.
\]
Notice $e_{2}^{\A}(A) = 0$.  So by \ref{dim2-hoa} we get $a(G_{\A}(A)) < 0$. Since $a(G_{\A}(A)) \geq -2$ and $\A$ is not
generated by a regular sequence it follows from \ref{a-result} that $a(G_{\A}(A)) = -1 $.
\end{example}
\begin{question}
 What are all the $\m$-primary ideals in a \CM \ local ring $(A,\m)$ having $a(\GA) = - \dim A + 1$?
\end{question}

%End 444444444444444444444444444444444444444444444444444444444444444444444444444444444444444444444444444
\section{Dual Filtrations of MCM modules over hypersurface rings}\label{section-hypersurface}
Let $(A,\m)$ be a complete equicharacteristic hypersurface ring and let $M$ be a MCM $A$-module. When $G_{\m}(M)$ is CM   we give a necessary and sufficient condition for the dual filtration, $\eF_M =\{ \Hom_A(M,\m^n)\}_{\nZ}$, on $M^*$ to be a shift of the $\m$-adic filtration on $M^*$: see \ref{mainHyper}. This result along with  Gorenstein approximation  is then used to deduce results about associated graded  rings of $\q$-primary $\A$
in a \CM \ local ring $(R,\q)$ with $\mu(\A) = \dim R +1$. We give proof of application \textbf{III} and half of application \textbf{IV}, stated in the introduction: see \ref{mu-aE-d+1} and \ref{mu-aE-d+1-cann}.
\s \label{setup}\textbf{ Setup:} Let $Q = k[[X_0, X_1,X_2\ldots, X_d]]$ where $k$ is an infinite field and let $\n$ be the maximal
ideal of $Q$. Let $(A,\m) = (Q/(f), \n/(f))$ where $f \in \n^e \setminus \n^{e+1}$ and $e \geq 2$. Notice $e_0(A) =e$ and
$G_{\m}(A) = G_{\n}(Q)/(f^*).$  So $G_\m(A)$ is Gorenstein.

\s \textbf{Some  invariants of a MCM module over a hypersurface ring. }

{\bf Case 1 :} $\dim A = 0.$  So $Q$ is a DVR. Let  $\n = (\Pi).$
As a  $Q$-module
\begin{equation}\label{0-dim-rep}
 M \simeq \bigoplus^{\mu(M)}_{i = 1} Q/(\Pi^{a_i (M)}) \quad \text{where} \quad 1 \leq a_1(M) \leq \ldots \leq a_{\mu(M)}(M) \leq e.
\end{equation}
Thus the decomposition above is as $A$-modules.
\s \label{decomp} The  Hilbert function of $M$ is
$$
H (M,z) = \sum^{\mu (M)}_{i = 1} \left(1 + z + \ldots +
z^{a_i(M) - 1}\right).
$$
\begin{remark} \label{Hilb-det}
Notice $H (M,z)$ completely determines $a_1(M), \ldots, a_{\mu(M)}(M)$  in the case when $\dim M = 0$.
\end{remark}
{\bf Case 2:} $d = \dim A \geq 1.$ \\
By  an argument similar to \cite[7.6]{Pu2} we get that for sufficiently
general $x_1, \ldots x_d \in \m\setminus \m^2,$ the Hilbert function of
$M/(x_1, \ldots x_d) M$ is constant. Using \ref{Hilb-det} we can define $a_i(M) = a_i(M/(\mathbf{x})M)$ for  sufficiently
general $\xb$.
\begin{definition}
We call the numbers $a_1(M), \ldots, a_{\mu(M)}(M)$  the \textit{generic superficial  invariants} of $M$.
\end{definition}
\begin{corollary}\label{e1-interms-ai}
 (with assumptions  as above)
$$e_0 (M) = \sum^{\mu (M)}_{i = 1} {a}_i (M)$$
\end{corollary}
\begin{proof}
 Let $\xb = x_1,\ldots,x_d$ be sufficiently general.  Set
$N = M/\xb M.$  Then
$$
e_0 (M) = e_0 (N) = \sum^{\mu (M)}_{i = 1} {a}_i (N)
= \sum^{\mu (M)}_{i = 1} {a}_i (M).
$$
\end{proof}

\s Assume $ 0 \rt Q^{n} \xrightarrow{\phi}Q^{n} \rt M \rt 0$ is a minimal
 presentation  of $M$. Set
\begin{equation*}
 i_{\phi} = \max\{i  \mid \ \text{all entries of $\phi$ are in } \  \n^i \}.
\end{equation*}
It is well known that  $i_{\phi}$ is an invariant of $M$.
We set $i(M)= i_{\phi}$.

\begin{proposition}\label{i=a1}[with hypothesis as above]
 $i(M) = a_1(M)$.
\end{proposition}
\begin{proof}
When $\dim M = 0$ the its clear that $i(M) = a_1(M)$.
 We can choose $\xb = x_1,\ldots, x_d$ sufficiently general such that $i(M) = i(M/\xb M)$
(see \cite[4.4]{Pu2}). The result follows from the zero-dimensional case.
\end{proof}

\begin{remark} \label{cann-d2} If $E = M \oplus N$ then
$$
E_{i} = \Hom_A (M \oplus N, \m^i)= M_{i}^{\dd} \oplus N_{i}^{\dd}.
$$
So $G (\eF_E , E^\dd) = G (\eF_M , M^\dd) \oplus G (\eF_N , N^\dd).$
\end{remark}

 \noindent{\bf The Dual Filtration in dimension zero.}
\begin{proposition}\label{dualDIMzero}
 Let $(Q,\Pi)$ be a DVR. Set  $A = Q/(\Pi^e)$ for some $e \geq 2$ and let $M$ be an $A$-module.
\begin{enumerate}[\rm (1)]
 \item \label{cann-d1}
If  $M = Q/(\Pi^{i})$  then $ \ G (\eF_M , M^\dd) \simeq G (M^\dd) (e - i).$
\item \label{cann-d3}
 If $M = \bigoplus_{i =1}^{\mu(M)} Q/(\Pi^{a_i(M)})$ then
$$G(\eF_M, M^\dd) = \bigoplus_{i =1}^{\mu(M)} G(Q/(\Pi^{a_i(M)}))(e -a_i(M)).  $$
\item \label{cann0}
 $G(\eF_M, M) \simeq G(M^\dd)$ (up to a shift) \ff \ ${a}_1(M)
 = \dots ={a}_{\mu (M)}(M).$
\end{enumerate}

\end{proposition}
\begin{proof}
 $\rm{(1) }$. Set
\[
 M_{j}^{\dd} = \Hom_{A} (M,\m^j) = \Hom_{Q/(\Pi^e)}\left( Q/(\Pi^{i}) , (\Pi^j) /(\Pi^e)\right).
\]
Clearly $    M_{j}^{\dd} = M^\dd$ for  $ j = 0, 1, \ldots, e-i$.

For $j > e - i$ let $f \in M_{j}^{\dd}$.
 Let
 $f(\ov{1}) = \alpha \ov{\Pi^{j}}  = \Pi^{j-(e -i)} \cdot (\alpha \Pi^{e - i}).$

Define $g \in M^\dd$ by $g(\ov{1}) = \alpha\ov{\Pi^{e - i}}$.
Clearly
$g$ is $A$-linear. Also  $g \in M_{e - i} = M^\dd$.
 Thus $ f(\bar{1}) =  \Pi^{j - (e - i)} g (\ov{1}) \in  \Pi^{j - (e - i)} M_{e-i}$.
Therefore
$
M_{j}^{\dd} \subseteq \m^{j - (e - i)} M^\dd$.  But $\m^{j - (e - i)} M^\dd\subseteq M_{j}^{\dd}$ always. So $M_{j}^{\dd} = \m^{j - (e - i)} M^\dd.$
Thus
$$ \ G (\eF_M , M^\dd) \simeq G (M^\dd) (e - i). $$
$\rm{(2) }$. This follows from  \ref{cann-d2} and (\ref{cann-d1}).

$\rm{(3) }$. This follows from  (\ref{cann-d3}).

\end{proof}
\begin{theorem}\label{mainHyper}
 (with hypothesis as in \ref{setup})  Assume
$G(M)$ is \CM. Then
\tFAE
\begin{enumerate}[ \rm (i)]
 \item
$G(M^\dd) \simeq \Hom (G (M) , G (A))$ up to a shift.
\item
${a}_1 (M) =  \ldots =
{a}_{\mu(M)}(M)$.
\item
$h_M (z) = \mu(M) (1 + z + \ldots + z^{i(M) -1})$.
\item
$h_M (z) = \mu(M) (1 + z + \ldots + z^{s -1})$ for some $s \geq 1$.
\end{enumerate}
\end{theorem}
We need the following
\begin{lemma}\label{lemmahyper}(with hypothesis as in \ref{setup})
   Set $\alpha = i (M)$.
Then
$$
h (M , z) = \mu (M) (1 + z + \ldots + z^{\alpha - 1}) +
\sum_{i\geq \alpha} h_i (M) z^i.
$$
Furthermore  $h_{\alpha} (M) < \mu (M)$.
\end{lemma}
\begin{proof}
 We prove the result
by induction on dimension of $A.$
When $\dim A = 0$ then the result follows by \ref{decomp}.

 If $\dim A > 0$  then let $x \in \m \setminus \m^2$ be sufficiently general.   Set $N = M/xM$. We
may choose $x$ such that $i (M) = i (N) = \alpha$, see \cite[4.4]{Pu2}.  By induction hypothesis
$$
h (N , z) = \mu (N) (1 + z + \ldots + z^{\alpha - 1}) + \sum_{i
\geq \alpha} h_i (N) z^i \quad \text{and} \quad h_\alpha (N) < \mu (N).
$$
By Singh's Lemma \cite[Theorem 1]{BS}
\begin{equation*}
H (M , z) = \frac{H (N , z)}{(1-z)} -  \sum_{i \geq 0}\ell\left(\frac{\m^{i + 1} M \colon x}{\m^i M}\right)z^i.
\end{equation*}
By \cite[4.6]{Pu2} we get  $(\m^{i + 1} M \colon x) =  \m^i M$  for all $ i = 0, 1, \ldots ,\alpha - 1$.
\[
\text{So}\  \quad  h_M (z) =  h(N , z)  - (1 - z)^d \bullet \sum_{i \geq \alpha} \ell\left(\frac{\m^{i + 1} M \colon x}{\m^i M}\right) z^i.
\]
It follows that $h_i(M) = h_i(N)$ for $ i = 0,\ldots, \alpha -1$ and
\[
 h_\alpha (M) = h_\alpha (N) -  \ell\left(\frac{\m^{\alpha + 1} M \colon x}{\m^\alpha M}\right)\leq h_\alpha (N) <
\mu (N) = \mu(M).
\]
The result follows.
\end{proof}
\begin{corollary}\label{corhyper}
 If $h (M , z) = \mu (M) (1 + z + \ldots + z^{s - 1})$ than $i(M)
= s$ and $G(M)$ is  \CM.
\end{corollary}
\begin{proof}
 The assertion $i(M) = s$ follows from Lemma \ref{lemmahyper}.  The fact that
$G(M)$ is \CM \  follows from \cite[Theorem 2]{Pu2}.
\end{proof}

We now give a
\begin{proof}[Proof of Theorem \ref{mainHyper}]
(iii) $\Leftrightarrow$  (ii) follows from \cite[Theorem 2]{Pu2}.

(iii) $\implies $(i).
Let $\eF_M$ be the dual filtration on $M^\dd$ \wrt \ $\m$.
As $G(M)$ is \CM \ it follows from Corollary \ref{corGor} that
$G(\eF,M^\dd) \cong \Hom (G(M) , G(A)).$
Let $\xb = x_1,\ldots,x_d \in \m \setminus \m^2$ be sufficiently general. Set $N = M/\xb M$. Then by
\ref{modFilt} we have $$G(\eF_M,M^\dd)/\xb^* G(\eF_M,M^\dd) \cong G(\eF_N,N^\dd).$$

Notice $a_1(M) =  \ldots = a_{\mu(M)}(M) =i(M)$. By \ref{dualDIMzero}(\ref{cann-d1}) it follows that the Hilbert series of $G(\eF_N,N^\dd)$ is $\mu (M) z^{e - i(M)}(1 + z + \ldots + z^{i(M) - 1})$ and that $G(\eF_N, N^\dd) = G(N^\dd)(e- i(M))$.  It follows from  Lemma \ref{a-stble=a-adic} that $G(\eF,M^\dd) \cong G(M^\dd)(e- i(M))$.
Therefore $G(M^\dd) \simeq \Hom (G(M), G(A))$ up to a shift.
\\
(i) $\implies$ (iii) By hypothesis $G (M^\dd) \simeq G (\eF,M^\dd)$ (up to a shift).  Thus $G(M^\dd)$ is \CM. By
\ref{modFilt}
we go mod a maximal regular sequence to reduce to dimension zero case. Here the assertion is true by \ref{dualDIMzero}(\ref{cann0}).

(iii) $\implies$ (iv) Nothing to show. \\
(iv) $\implies$ (iii) Follows from Corollary \ref{corhyper}.
\end{proof}

\s \textbf{The case when $(A,\m)$ is \CM \ (need not be a hypersurface ring) but $\A$ is an $\m$-primary ideal with
$\mu(\A) = d + 1$.}

\begin{remark}\label{today}
Assume $A$ is also complete and equicharacteristic.
By \ref{field} we get that $[A,\m,\A]$ has a
Gorenstein approximation $[B,\n,\n,\psi]$ with $\mu(\n) = d +1$. It follows that $B$ is a hypersurface ring.  Also $A$ is  MCM as a $B$-module. Clearly $\GA = G_\n(A)$. It can be easily checked that
$\om_A \cong \Hom_{B}(A, B) = A^\dd$ and $\Om_{\A}^{A} = \Hs_{G(B)}(G_\n(A), \Om_{\n}^{B})$. Also note that $G_\n(B) \cong \Om_{\n}^{B}$ up to a shift.
\end{remark}

\begin{theorem}\label{mu-aE-d+1-cann}
  Let $(A,\m)$ be a \CM \  equicharacteristic local
ring and a canonical module $\om_A$.  Let $\mathfrak{a}$ be an $\m$-primary
ideal with $\mu(\mathfrak{a}) = d + 1$ and
$G_{\mathfrak{a}}(A)$  \CM \ with canonical module $\Om_{\A}^{A}$.
\TFAE \
\begin{enumerate}[\rm (i)]
 \item
$G_{\mathfrak{a}} (\om_A) \simeq \Om_{\A}^{A}  $ (up to a shift).
\item
$h_{\mathfrak{a}} (A,z) = \ell(A/\mathfrak{a}) (1 + z + \ldots +
z^{s-1})$ for some $s \geq 2.$
\end{enumerate}
\end{theorem}
\begin{proof}
 This follows from \ref{today} and \ref{mainHyper}.
\end{proof}
\begin{theorem}\label{mu-aE-d+1}
 Let $(A,\m)$ be a equicharacteristic Gorenstein local ring and
$\mathfrak{a}$ an $\m$-primary ideal with $\mu(\mathfrak{a}) =
d + 1.$
\TFAE
\begin{enumerate}[\rm (i)]
 \item
$G_{\mathfrak{a}} (A)$ is Gorenstein.
\item
$h_{\mathfrak{a}} (A,z) = \ell(A/\mathfrak{a}) (1 + z + \ldots +
z^{s - 1})$ for some $s \geq 2.$
\end{enumerate}
\end{theorem}
\begin{proof}
 The assertion (i) $\implies$ (ii) follows from  \ref{mu-aE-d+1-cann}.

\noindent (ii) $\implies$ (i). We use notation as in \ref{today}. Notice $A$ is a MCM $B$-module and
$h_\A(A,z) = h_\n(A,z)$. So by \ref{corhyper} we get that $G(A) = \GA$ is \CM. The result follows from
\ref{mu-aE-d+1-cann}.
\end{proof}
\begin{question}
Are the results of Theorem \ref{mu-aE-d+1}  and Theorem \ref{mu-aE-d+1-cann} true when $A$ does not contain a field?
\end{question}

%End 99999999999999999999999999999999999999999999999999999999999999
%End 99999999999999999999999999999999999999999999999999999999999999
\section{A generalization of a result due to Ooishi}\label{section-Ooishi}
\s \label{canFilt-setup} \textbf{Introduction \& Setup: } Let $(A,\m)$ be a CM  local ring with a canonical module $\om_A$. Let $\A$ an $\m$-primary ideal and $M$ a CM $A$-module. Set $M^\da = \Ext^{\dim A - \dim M}_{A}(M,\om_A)$. Assume $\GA$, $G_\A(M)$ are CM.

 Let $\Om_\A$, $\K_\A$ be the canonical module of $\GA$, $\ra$ respectively.
Notice that $\Om_\A \cong \K_\A/t^{-1}\K_\A(-1)$(see \cite[3.6.14]{BH}).

In Theorem \ref{cannFiltDual} we show that if $A$ is complete then there is a $\A$-stable filtration $\eF$ on $M^\da$ such that
\begin{align}
 \label{formula-dual}\R(\eF, M^\da) &\cong \Es_{\ra}^{\dim A - \dim M}\left( \R_\A(M), \K_\A \right)\\
\label{formula-dual-G}\text{equivalently} \quad  G(\eF(-1), M^\da) &\cong \Es_{\GA}^{\dim A - \dim M}\left( G_\A(M), \Om_\A \right).
\end{align}
As an application we give a generalization of a result due to
Ooishi.

\begin{definition}\label{def-canonical}[With hypotheses as in \ref{canFilt-setup}]
 If there exists a $\A$-stable filtration $\eF$ on $M^\da$ such that \ref{formula-dual} (equivalently \ref{formula-dual-G}) holds then
we say $\eF$ is a  \emph{canonical filtration} on $M^\da$ \wrt \ $\A$.
\end{definition}
Two canonical filtrations (say $\eF, \eG$)
  on $M^\da$ are
equivalent, i.e., there exists an $A$-linear isomorphism $\sigma \colon M^\da \rt M^\da$ such that
$\sigma(\eF_n) = \eG_n$ for all $\nZ$.  This is due to the following
 well-known
 result.
 \begin{theorem}\label{unique}
Let $M$ be an $A$-module and let $\A$ be an ideal in $A$. Suppose
$\eF = \{ \eF_n \}_{\nZ}$ and $\eG = \{ \eG_n \}_{\nZ}$ are two $\A$-stable filtration's such that
$\eR(\eF, M) \cong \eR(\eG, M)$ as $\ra$-modules. Then $\eF, \eG$ are equivalent filtration's.
\end{theorem}

We need the following graded version of a result of Rees \cite{rees-Hom}.
\begin{remark}\label{change-base-Ext}
Let $R = \bigoplus_{\nZ} R_n$ be a graded ring and let  $M, N$ be graded $R$-modules.  Let $x$ be a homogeneous element in $R$. If $x$ is both $M$ and $R$-regular element and $x\cdot N = 0$ then
\[
 \Es_{R/xR}^{i}(N, M/xM) \cong \Es^{i+1}_{R}(N,M)(-\deg x) \quad \text{for all} \ \ i \geq 0.
\]
The proof is along the same lines as given (for example in) \cite[3.1.16]{BH}
\end{remark}

\begin{theorem}\label{cannFiltDual}[With hypotheses as in \ref{canFilt-setup}] Further assume $A$ is complete. Then $M^\da$ has a canonical filtration.
  Furthermore  if $\eF, \eG$ are two such filtration's then they are equivalent as filtration's on $M^\da$.
\end{theorem}
\begin{proof}
Uniqueness of the filtration up to equivalence follows
by \ref{unique}.

\textbf{Case 1:} \textit{ $M$ is a MCM $A$-module.} \\
Let $[B,\n,\B,\phi]$ be a Gorenstein approximation of $(A,\m,\A)$.  Consider the dual filtration  $\eF_M = \{ \Hom_B(M, \B^n) \}_{\nZ}$ on $M^\da$.
Since $\GB$ is Gorenstein \&  $G_\B(M) = G_\A(M)$ is CM,  we  get
\begin{equation}\label{v-2}
 G(\eF_M,M^{\da}) \cong \Hs_{G_{\B}(B)}(G_{\A}(M),\GB) \quad \quad \text{(see \ref{corGor})}.
\end{equation}
The canonical module of $\GB$ is $\GB(s)$ for some $s \in \Z$. So
\[
 \Om_\A \cong \Hs_{\GB}\left( \GA, \GB(s) \right) \quad \quad \text{(see \cite[3.6.12]{BH})}.
\]
Notice that we have the following isomorphisms of $G = \GA$-modules
\begin{align*}
 \Hs_{G}\left( G_\A(M),  \Om_\A \right) &\cong \Hs_{G}\left( G_\A(M), \Hs_{\GB}\left(G, \GB(s) \right) \right) \\
  &\cong \Hs_{\GB}\left( G_\A(M),  \GB(s) \right) \\
   &\cong G(\eF_M(s),M^{\da}),  \quad \text{using \ref{v-2} and \ref{shift-filt} }.
\end{align*}

\textbf{Case 2:} \textit{  $M$ is a \CM  \ $A$-module but not necessarily MCM} \\
By case 2 we may assume that $c = \dim A - \dim M > 0$. Let $u_1,\ldots u_c$ be as in Corollary \ref{redTOmaxDIM}. Recall
\begin{enumerate}[\rm (1)]
 \item
$ u_1,\ldots, u_c \in \ann_A M $ and
 $ u_{1}^*,\ldots, u_{c}^* \in \ann_{\GA} G_\A(M)$.
\item
$ u_{1}^*,\ldots, u_{c}^*$ is a $\GA$-regular sequence.
\end{enumerate}
Set $A' = A/(\ub), \A' = \A/(\ub)$.
Then $M$ is a $A'$-module. Also notice that $G_{\A'}(M) = G_\A(M)$.
Since $\Om_\A$ is a  MCM $\GA$ module, it follows that $u_1^*,\ldots, u_c^*$ is $\Om_\A$-regular.
Notice also that $ u_1,\ldots, u_c \in \ann_A M^\da $ and $ u_{1}^*,\ldots, u_{c}^* \in \ann_{\GA} G_\A(M^\da)$.

Set $w =\sum_{i=1}^{c} \deg u_{i}^{*}$.  Notice  $\Om_\A/(\ub^*\Om_\A)\left( w \right)$ is the canonical
module of $G_{\A'}(A')$; cf. \cite[3.6.14]{BH}.  Using \ref{change-base-Ext} repeatedly we get
\begin{equation}\label{v-3}
 \Hs_{G_{\A'}(A)}\left( G_{\A'}(M), \Om_\A/(\ub^*\Om_\A) \right) \cong \Es_{G_{\A}(A)}^{c}\left( G_{\A }(M), \Om_{\A} \right)\left( - w \right).
\end{equation}

By case 1 we have a $\A'$-stable filtration $\eF$ on $M^\da$ such that
$$G(\eF, M^\da ) \cong \Hs_{G_{A'}(A')}\left( G_{\A'}(M), \Om_\A/(\ub^* \Om_\A )\left( w \right) \right).$$
Notice $\eF$ is an $\A$-stable filtration on $M^\da$. Also $G_{\A'}(M) = G_\A(M)$.
By \ref{v-3} we get
$$ G(\eF, M^\da) \cong \Es_{\GA}^{\dim A - \dim M}\left( G_\A(M), \Om_\A \right).$$
\end{proof}

We now state our generalization of Ooishi's result.

\begin{theorem}\label{symmetry}
[with hypotheses as in \ref{canFilt-setup}] Set $r = \red(\A, M)$ and assume further that $G(M^\da)$ is \CM.
Then the following assertions are equivalent:
\begin{enumerate}[\rm (i)]
\item
$h(M^\da,z) = z^r\cdot h(M,z^{-1})$.
\item
$G(M^\da) \cong \Es_{\GA}^{\dim A - \dim M}(G_{\A}(M), \Om_\A)$ (up to a shift).
\end{enumerate}
\end{theorem}
\begin{proof}
The assertion (ii) $\implies$ (i) follows from \cite[4.4.5]{BH}.

To prove the converse, we note that using the argument Case 2 in Theorem \ref{cannFiltDual} we may assume that
$M$ is a MCM $A$-module. Furthermore we may assume that $A$ is complete with infinite residue field.

We first consider the case
when $\dim A = 0.$ Consider the usual $\A$-adic filtration:
 $$M^\da \supseteq \mathfrak{a}M^\da \supseteq \mathfrak{a}^2M^\da \supseteq
\ldots \supseteq \mathfrak{a}^rM^\da \supseteq \mathfrak{a}^{r + 1}M^\da =
0$$
Let the canonical filtration of $M^\da$ (up to a shift) be :
$$M^\da \supseteq \mathfrak{q}_1 \supseteq \mathfrak{q}_2 \supseteq
\ldots \supseteq \mathfrak{q}_s \supseteq \mathfrak{q}_{s + 1} =
0 \quad \text{with} \ q_1 \neq M^\da.$$
Let the Hilbert series of $M^\da$ \wrt \ $\A$ be
$$
H_{\A} (M,z) = h_0 + \ldots + h_r z^r.
$$
Then  the Hilbert series of 
$G_{\mathfrak{a}} (M)^\da$ is
$$
H_\mathfrak{a} (M,z^{-1}) = h_r z^{-r} + h_{r - 1} z^{-r + 1} + \ldots + h_1 z^{-1} + h_0
$$
cf.   \cite[4.4.5]{BH}. It follow that $s = r.$  Also
$
\ell (\mathfrak{q}_r) = h_0 = h_r = \ell (\mathfrak{a}^rM^\da).
$
Since $\mathfrak{q}_r \supseteq \mathfrak{a}^rM^\da$ we have
$\mathfrak{a}^rM^\da = \mathfrak{q}_r.$\\
By downward induction we show $\mathfrak{q}_i = \mathfrak{a}^iM^\da$ for
$i \leq r.$ Notice $\mathfrak{q}_i \supseteq \mathfrak{a}^iM^\da$.

For $i = r$ we just showed $\mathfrak{a}^rM^\da = \mathfrak{q}_r$.  Assume for $i =j+1$ and prove
for $i = j$.
Notice that
$$
\ell \left(\frac{\mathfrak{q}_j}{\mathfrak{q}_{j + 1}}\right) = h_{r -
j} = h_j =  \ell \left(\frac{\mathfrak{a}^jM^\da}{\mathfrak{a}^{j +
1}M^\da}\right).
$$
But $\mathfrak{q}_{j + 1} = \mathfrak{a}^{j + 1}M^\da$ and $\mathfrak{q}_j
\supseteq \mathfrak{a}^jM^\da.$ So $\mathfrak{q}_j
=\mathfrak{a}^jM^\da.$
The result follows.

Next we consider the case when $d = \dim A > 0$.  By hypothesis $\GA, G_\A(M)$ are \CM. Let $x_1,\ldots,x_d \in \A/\A^2$ be such that
 $x_{1}^{*},\ldots,x_{d}^{*}$ is $\GA$ and $G_\A(M)$ regular sequence. Set $B = A/(\xb)$, $N = M/\xb M$ and $\B = \A B$.
 Furthermore $\GB = \GA/(\xb^*)$ and $G_\B(N) = G_\A(M)/\xb^* G_\A(M)$. So the $h$-vector of $G_\B(N)$ is symmetric.
  By the previous case we
 have that the $\B$-adic filtration on $N^\da$ is the canonical filtration on $N^\da$ up to a shift. By \ref{a-stble=a-adic}
it follows that the  $\A$-adic filtration on $M^\da$ is the canonical filtration on $M^\da$ up to a shift. Therefore
$$G(M^\da) \cong \Es_{\GA}^{\dim A - \dim M}(G_{\A}(M), \Om_\A) \quad  \text{(up to a shift).} $$
\end{proof}

%End 99999999999999999999999999999999999999999999999999999999999999
%\section{Determination of some dual filtrations}\label{determine-dual-MCM}
%**some intro**

\section{Application of   our generalization of Ooishi's result}\label{determine-dual-MCM}
We use Theorem \ref{symmetry}  to
to show that the dual filtration \wrt \ $\m$ of a MCM module $M$, over a Gorenstein local ring $A$ with $G_\m(A)$ Gorenstein, is a shift of the usual $\m$-adic filtration in the following cases:
\begin{enumerate}
 \item
$M$ is \textit{Ulrich} (i.e., $\deg h_M(z) = 0$; equivalently $e(M) = \mu(M)$).
\item
$\type(M) = e(M) -\mu(M)$ and $M$ has \textit{minimal multiplicity} and not Ulrich.
\end{enumerate}
Assume  $(R,\n)$ is  CM local with a canonical module $\omega_R$ and   $G_\q(R)$ is CM  for some $\n$-primary ideal $\q$.
We  use Theorem \ref{symmetry} to  give an easily verifiable  condition on whether  $G_\q(\om_R) $ is the canonical module of $G_\q(R)$ (up to a shift).

 The following  criterion is useful to show dual filtrations are $\A$-adic up to a shift.
\begin{proposition}\label{Ooishi--Gor-easy}
Let $(A,\m)$ be Gorenstein, $\A$ an $\m$-primary ideal with $\GA$ Gorenstein. Let $M$ be an MCM $A$-module with $G_\A(M)$ \CM. Let $\eF$ be the dual filtration of $M$ \wrt \ $\A$. Set $c = \red(\A, M)$. Let $\xb = x_1,\ldots,x_d$ be an $A$ superficial sequence \wrt \ $\A$. Set $(B,\n) = (A/(\xb) , \m/(\xb)), \B =A/(\xb)$ and $N = M/\xb M$. Set $N^* = \Hom_B(N,B)$.
\TFAE
\begin{enumerate}[\rm (i)]
\item
$h(N^\dd,z) = z^c\cdot h(M,z^{-1})$.
\item
$G_\A(M^*) \cong G(\eF, M^*)$ (up to a shift).
	\item
	$\eF$ is the $\A$-adic  filtration of $M^*$ (up to a shift).
\end{enumerate}
\end{proposition}
\begin{proof}
(iii) $\implies$ (ii). Is trivial

(ii) $\implies$ (iii). Follows from \ref{filtEQUALgraded}.

(i) $\implies$ (iii)
 Let $\eG$ be the dual filtration of $N^*$ \wrt \ $\B$.
Notice $h(N,z) = h(M,z)$, as $G_\A(M)$ is CM.
 Using \ref{symmetry} we get that $\eG$ is $\B$-adic up to a shift.

Notice $\xb^* = x_1^*, \ldots, x_d^*$ is a $\GA$-regular sequence,  \cite[8]{Pu1}. So  $\xb^*= x_1^*, \ldots, x_d^*$ is a  $G_\A(M)$-regular sequence.
Using  \ref{modFilt} we get that   $\eG$ is the quotient filtration of $\eF$.   By \ref{a-stble=a-adic}  it follows that $\eF$ is $\A$-adic \emph{up to a shift}.

(ii) $\implies$ (i).
By \ref{corGor},  $G_\A(M^*)$ is a MCM $\GA$-module. So $h(N^\dd,z) = h(M^\dd,z)$. We also have $h(N,z) = h(M,z)$. The result follows from \ref{symmetry}.
\end{proof}

\noindent \emph{Applications:}

\noindent\textbf{Dual filtrations of Ulrich modules:}
\begin{theorem}\label{Ulrich}
 Let $(A ,\m)$ be a Gorenstein local ring with $G_{\m} (A)$ Gorenstein. Assume $A$ has infinite residue field.  Let $r = \red(A)$.
Let $M$ be a MCM $A$-module.
\TFAE
\begin{enumerate}[\rm (i)]
\item
$M_{r+i}  = \m^i M^\dd$ for all $\iZ$.
\item
$ G (\eF , M^\dd) \simeq G (M^\dd) (-r).$
\item
$M$ is an Ulrich A-module.
\end{enumerate}
\end{theorem}
\begin{proof}
By \ref{filtEQUALgraded} (i) and (ii) are equivalent.

(ii) $\implies$ (iii)
 Notice
$\alpha(M) =r= \red(A)$. So
$M$ is Ulrich; see \ref{baby-Ulrich}.

(iii) $\implies$ (ii)
 Let $N = M/\xb M$ where $\xb = x_1, \ldots ,
x_d$ is a maximal superficial $A$-sequence.  Notice $N = k^{e_0 (M)}$. Set $(B , \n) =
(A/(\xb) , \m/(\xb))$. Notice $N^{\dd}= \Hom_{B}(N,B) \cong N$ since $B$ is Gorenstein.
Clearly $h(N^*,z) = e_0(M) =z^0h(M,z^{-1})$. So by \ref{Ooishi--Gor-easy} we get  that
$G(\eF, M^*) \cong G(M^*)$ up to a shift; say $t$.

\emph{Determining $t$}:  Since $G(M)$ is \CM, we may reduce to the zero-dimensional case,
see \ref{modFilt}. So $M =k^{\mu(M)}$. In this case it is easy to show that $G(\eF, M^*) \cong G(M^*)(-r)$. So $t =r$.
\end{proof}

 \noindent\textbf{Dual filtration of non-Ulrich Modules  having  minimal multiplicity}

\begin{remark}\label{min-basic-remark} Let $\type(M) = $ \CM-type of $M$  $= \dim_{k} \Ext^{d}_{A} (k , M)$.  If $M$ has minimal multiplicity then
 $\tau (M) \geq e_0 (M) - \mu (M).$ \\
{\bf Proof :}
Let $x_1, \ldots , x_d$ be a maximal $M$-superficial sequence. Set
$N = M/\xb M.$  Note that  $e_0 (M) = e_0 (N) = \ell(N).$
Set $(B ,\n) = (A/(\xb) , \m/(\xb)).$  Since $M$ has minimal multiplicity we get  $\n^2 N = 0$. So
$\n N \subseteq \socle(N)$.  As
$G_{\n} (N) = N/\n N \oplus \n N/(0).$
So
$
\ell(\n N) = e_0 (N) - \mu(N) = e_0 (M) - \mu (M).
$
Since $\n^2 N = 0$ we get $\n N \subseteq \socle (N)$. So
 $e_0 (M) - \mu (M) \leq \ell(\socle (N)) = \type (N) = \type(M).$
\end{remark}

\begin{theorem}\label{minmultDual}
 Let $(A , \m)$ be a Gorenstein local ring with an infinite residue field. Assume $G_{\m} (A)$ is
Gorenstein.    Let $M$ be a MCM $A$-module with minimal
multiplicity \wrt \ $\m$  and $M$ is NOT Ulrich.
\TFAE
\begin{enumerate}[\rm (i)]
 \item
$\type (M) = e_0 (M) - \mu (M)$.
\item
$M^\dd$ has minimal multiplicity and $G (\eF_M , M^\dd) \simeq G_{\m}(M^\dd)$ up to a shift.
\item
$G (\eF_M , M^\dd) \simeq G_{\m}(M^\dd)$ up to a shift.
\end{enumerate}
\end{theorem}
\begin{proof}
 Set $G = G_\m(A)$, $G(M) = G_\m(A)$.

 $\rm (i) \implies \rm (ii)$. Since $M$ has minimal multiplicity we get that
 \[
 h(M,z) = \mu(M) + (e_0(M)- \mu(M))z
 \]
  Notice $M^\dd$ is MCM.
Let $\xb = x_1,\ldots,x_d \in \m \setminus \m^2$ be sufficiently general. Set $(B,\n) = (A/(\xb), \m/(\xb))$, $N = M/\xb M$.  Also note that
 $N^* \cong M^*/\xb M^*$.
 Since $\n^2 N = 0$, we get $\n^2 N^* = 0$.
 $
\text{So}\ \ H (N^\dd,z) = \mu (N^\dd) + \ell(\n N^\dd)z.
$
 By hypothesis it follows that $\mu(N^*)= \type (N) = \type(M) = e_0 (M) - \mu(M)$.
Since $\ell(N^*) = e_0(N^*) = e_0(N) = e_0(M)$ we get that
\[
h(N^*,z)= (e_0 (M) - \mu(M)) + \mu(M)z = zh(M,z^{-1}).
\]
So by \ref{Ooishi--Gor-easy} we get $G(\eF_M, M^*) \cong G(M^*)$ up to a shift.
 So $G(M^*)$ is \CM \ (by \ref{corGor}). Since $N^* \cong M^*/\xb M^*$ and $N^*$ has
 minimal multiplicity it follows that $M^*$ has minimal multiplicity.

$\rm (ii) \implies \rm (iii)$ is clear

$\rm (iii) \implies \rm (i)$
 Using \ref{Ooishi--Gor-easy} and \ref{corGor} it follows that
$h(M^*,z) = zh(M,z^{-1})$. It follows that $\type(M) = e_0(M)- \mu(M)$.
\end{proof}

\textbf{Associated graded module of the Canonical module}

\s \label{canonical module-setup} Let $(A,\m)$ be \CM \ local ring with a canonical module $\om_A$.  Let
$\A$ be an $\m$-primary ideal \emph{such that $\GA$ is \CM}.  Let $r = \red(\A, A)$. Let $\Om_\A$ be the canonical module of $\GA$. Let $\type(A)= \ell(\Ext{^d}_{A}(A/\m, A))$ be the \CM \ type of $A$.

\begin{theorem}\label{cannonical-module-Mod}[hypothesis as in \ref{canonical module-setup}]
  Let $x_,\ldots,x_d$ be a maximal $A$-superficial sequence \wrt \ $\A$. Set $B =A/(\xb)$ and $\B = \A/(\xb)$.
If   $h(\om_B,z) = z^{r} h(B,z^{-1})$ then $G_\A(\om_A)$ is \CM \ and isomorphic to $\Om_\A$ up to a shift.
\end{theorem}
\begin{proof}
Notice that we may assume that $A$ is complete.
Then by
\ref{lift-sup-Gor-App} there exists
$[R,\n,\q,\phi]$,  a Gorenstein approximation of $[A,\m,\A]$ with the property that
there exists $y_1,\ldots,y_d \in \q$ such that $\phi(y_i) =x_i$ for $i = 1,\ldots,d$. Furthermore $y_1^*,\ldots,y_d^*$ is $G_\q(R)$-regular.   Notice
that $\om_A \cong \Hom_R(A,R)$, i.e., the dual of the MCM $R$-module $A$. The result follows from \ref{Ooishi--Gor-easy}.
\end{proof}
Theorem \ref{cannonical-module-Mod} is quite practical as shown by the following:
\begin{example}\label{cann-example}
Let $A = \QQ[t^{13},t^{18},t^{23},t^{28},t^{33}]$ and $\m = (t^{13},\ldots,t^{33})$. Set $B = A/(t^{13})$. By \cite[1.1]{TamM}; $G_\m(A_\m)$ is CM.
Set $R = \QQ[x,y,z,w] $.
Consider the natural map
$\phi \colon R \rt A$ with $\phi(x) = t^{13},\ldots, \phi(w) = t^{33}$.
Using MACAULAY we can compute $\q = \ker \phi$.  Set $S = R/\q$ and $\n = (x,y,z,w)$. Also set $T = S/(x)$. Using  COCOA we get
\[
h(A,z) = h(S,z) = 1 + 4z + 4z^2 + 4z^3 \quad \text{and} \quad
h(B,z) = h(T,z) = 1 + 4z + 4z^2 + 4z^3
\]
This proves that $G_\n(S_\n)$ is CM and $x^*$ is $G_\n(S_\n)$-regular. (Notice this also proves  $G_\m(A_\m)$ is CM and
${t^{13}}^*$ is $G_\m(A_\m)$-regular).

We consider $R$ graded with
$\deg x = 13, \deg y = 18, \deg z = 28$ and $\deg w = 33$. Set $D = \Ext^{5}_{R}(R/(\q + (x)), R)$. Note that
\[
D \cong \om_{T_{\n}} \cong \om_{S_{\n}}/x \om_{S_{\n}} \cong \om_{A_{\m}}/t^{13} \om_{A_{\m}}.
\]
Using MACAULAY we get that
\[
h(D,z) = 4 + 4z + 4z^2 + z^3 = z^3h(B,z^{-1})
\]
By \ref{cannonical-module-Mod} we get that $G(\om_{A_\m})$ is the canonical module of $G_\m(A_\m)$ up to a shift.
\end{example}
An easy consequence of \ref{cannonical-module-Mod} is the following:
\begin{corollary}\label{red2type2}
Let $(A,\m)$ be \CM \ local ring with $\red(A) = 2$. Then $G_\m(\om_A) \cong \Om_\m$(up to a shift) if and only if  \ $\type(A) = e_0(A) - h_1(A) - 1 $.
\end{corollary}
\begin{proof}
Notice $G_\m(A)$ is CM; \cite[2.1]{Sa2}. Set $h(A, z) = 1 +  h_1z + h_2z^2$ and $e= e_0(A) = e_0(\om_A)$.

($\Rightarrow$) This implies $ h(\omega_A, z) = z^2h(A, z^{-1})$. It follows that
$$ \type(A) = \mu(\omega) = h_2 =  e_0(A) - h_1 - 1. $$

($\Leftarrow$) We may assume $A$ has infinite residue field. Notice $h_2 = \type(A)$.  Let $\xb = x_1,\ldots, x_b$ be a maximal superficial $A$-sequence. Set $(B,\n) = (A/(\xb), \m/(\xb)$. Notice $\om_B \cong \om_A/(\xb)\om_A$. Check that
$$ h(\omega_B, z) = \type(B) + (e-\type(B) - 1)z + z^2 = z^2h(A, z^{-1}). $$
The result follows from \ref{cannonical-module-Mod}.
\end{proof}

\section{CI approximation: The general case} \label{sectionGAPP}
In this section we prove  our general result regarding  complete intersection (CI)
approximation; see  \ref{equiTT}. In  \ref{field} we showed that every $\m$-primary ideal $\A$ in an equicharacteristic
local ring $A$ has a
CI-approximation. Even if we are interested only in the case of $\m$-primary ideals; we have to take some care for
dealing with the case of local rings with mixed characteristics. \emph{The essential point is to show existence of
homogeneous regular sequences in certain graded ideals.} We also prove a Lemma regarding lifting of superficial sequence
along a Gorenstein approximation; see \ref{lift-sup-Gor-App}.

This section is divided into two subsections. In the first subsection $R = \bigoplus_{n\geq 0}R_n$ be  a standard graded algebra
over a local ring $(R_0,\m_0)$. When $R$ is \CM \ and $M =\bigoplus_{n\geq 0}M_n$ is a \fg graded $R$-module generated by elements in $M_0$,
we give conditions to ensure a regular sequence $\xi_1,\ldots,\xi_c \in R_+ \cap\ann_R M$
where $c= \dim R - \dim M$ (see Theorem \ref{hann}).
In the second subsection we prove Theorem \ref{equiTT}.

\textbf{Homogeneous regular sequence }

Let  $R = \bigoplus_{n\geq 0}R_n$ be  a standard  algebra
over a local ring $(R_0,\m_0)$ and let $M =\bigoplus_{n\geq 0}M_n$ be a \fg graded $R$-module.
 In general a graded module $M$
need not have homogeneous  regular sequence $\xi_1,\ldots,\xi_c \in R_+ \cap\ann_R M$ where $c = \grade M$ (see \ref{eqEx}).
 When $R$ is \CM \ and $M =\bigoplus_{n\geq 0}M_n$ is a \fg \ graded $R$-module generated by elements in $M_0$,
we give conditions to ensure a regular sequence $\xi_1,\ldots,\xi_c \in R_+ \cap\ann_R M$
where $c= \dim R - \dim M$.

We adapt an example from \cite[p.\ 34]{BH} to show that a homogeneous  regular sequence of length
$\grade M$
in $R_+ \cap\ann_R M$ need not exist.

\begin{example}\label{eqEx}
 Let $A = k[[X]]$ and $T = A[Y]$. Set $R = T/(XY)$. Notice that $R_0 = A$. Set $\M = (X,Y)R$ the unique graded maximal ideal of $R$ and let $M = R/\M = k$. Clearly $\grade M = \grade(\M , R) = 1$. It can be easily checked that
every homogeneous element in $R_+$ is a zero-divisor.
\end{example}

The following result regarding existence of homogeneous regular sequence in $(\ann M) \cap R_+$ (under certain conditions) is crucial in our proof of Theorem \ref{equiTT}.
\begin{theorem}
\label{hann}
Let $(R_0,\m_0)$ be a  \ local ring and let $R = \bigoplus_{n\geq 0}R_n$ be a
standard graded $R_0$-algebra and let $M = \bigoplus_{n\geq 0}M_n$ be a finitely generated
$R$-module. Assume $R$ has a homogeneous s.o.p. Set $c = \dim R - \dim M$.
Assume
\begin{enumerate}[ \rm (1)]
\item
$R_0$ is \CM \ and $R$ is \CM.
\item
$M$ is generated as an $R$-module by some elements in $M_0$.
\item
There exists
$y_1,\ldots,y_s \in R_0$  a system of parameters for both $M_0$ and $R_0$ and a  part of a homogeneous system of parameters for both $M$ and $R$.
\end{enumerate}
Then there exists a homogeneous regular sequence
  $\xi_1,\ldots,\xi_c \in R_+$ such that $\xi_i \in \ann_R M$ for each $i$.
\end{theorem}
\begin{proof}
We prove this by induction on $s = \dim R_0$.
When $s = 0$ $R_0$ is Artinian. So $\grade(\ann M,R) = \grade(\ann M \cap R_+, R)$. As $R$ is CM \ we have $\grade(\ann M, R) = c$.
 The result  follows from \cite[1.5.11]{BH}.

We  prove the assertion for $s = r +1$
assuming it to be valid for $s=r$. Set $S = R/y_1R = \bigoplus_{n\geq 0} S_n$ and $N = M/y_1M$. Note that $N$ is also generated in degree $0$.  Let $\ov{y_i} = y_i \mod (y_1)$ with $2 \leq i \leq s$. Then $\ov{y_2} \ldots \ov{y_s}$
is a system of parameters for both $S_0$ and $N_0 = M_0/y_1M$. Notice $c = \dim R - \dim M = \dim S - \dim N$. Also note that $S$ has h.s.o.p. By induction hypothesis
there exists a $S$-regular sequence
  $\ov{\xi_1},\ldots,\ov{\xi_c} \in S_+$ such that $\ov{\xi_i} \in \ann_S N$ for each $i$.

$R$ is CM. So $y_1$ is  $R$-regular.
Therefore $y_1,\xi_1,\xi_2\ldots,\xi_c$ is an $R$-regular sequence.
Set $\q = y_1R_+$. Let $M$ be generated as an $R$-module by $u_1,\ldots,u_m \in M_0$.
Fix $i \in \{1,\dots,c \}$ and set $\xi = \xi_i$. By construction $\xi u_j \in y_1 M$ for each
 $j$. However $\xi \in R_+$ and $y_1 \in R_0$ . So $\xi u_j \in \q M$ for each
 $j$. By the determinant trick (observing that each $u_j$ has  degree zero)
 there exists a \textit{ homogeneous} polynomial
\[
\Delta(\xi) = \xi^{n} + a_{1}\xi^{n-1} + \ldots +a_{n-1}\xi + a_n \in \ann_R M \quad \text{and} \ a_i \in \q^i \ \text{for} \ i =1,\dots,n.
\]
Clearly $\Delta(\xi) \in R_+$.
Notice that $\Delta(\xi) = \xi^{n} \mod (y_1)$. So
$y_1,\Delta(\xi_1),\ldots,\Delta(\xi_c)$ is a regular sequence, \cite[16.1]{Ma}.
Therefore $ \Delta(\xi_1),\Delta(\xi_2)\ldots,\Delta(\xi_c)$
is an $R$-regular sequence. This proves the assertion when $s =r+1$.
\end{proof}

\textbf{CI approximation}

In this subsection we prove our general result regarding  CI-approximation.

\begin{theorem}\label{equiTT}
Let $(A,\m)$ be a complete
 local ring and let $\A$ be a proper  ideal in $A$ with $\dim A/\A + \spr(\A)
 = \dim A$. Then
$A$ has a CI-approximation \wrt \ $\A$. Furthermore if $\A = (x_1,\ldots,x_m)$ then we may
choose
a CI-approximation $[B,\n,\B,\varphi]$ of $[A,\m,\A]$ such that there exists $u_1,\ldots,u_m \in \B$ with $\varphi(u_i) = x_i$ for  $i = 1,\ldots, m$.
\end{theorem}
\begin{proof}
Notice $\GA$ has a homogeneous s.o.p (see proof of Proposition 2.6 in \cite{HUO}).
 Let $ y_1,\ldots,y_s \in A$ such that $ \ov{y_1},\ldots,\ov{y_s}$ (their images in $A/\A$ ) is an
\emph{s.o.p} of $A/\A$.

\textbf{Case 1.} \emph{A contains a field:}
 As $A$ is a complete
$A$ contains  $k = A/\m$.
  Set $S = k[[ Y_1,\ldots, Y_s, X_1,\ldots,X_n ]]$. Define
$\phi \colon S \rt A$ which maps $Y_i$ to $y_i$ and $X_j$ to $x_j$ for $i= 1,\ldots, s$
and $j = 1,\ldots, n$.
 Set $\q = (Y_1,\ldots, Y_s)$.
Notice
\[
\frac{A}{(\q, X_1,\ldots,X_n)A} = \frac{A}{(\A, y_1,\ldots,y_s)} \quad \text{has finite length.}
\]
As $S$ is complete we get that $A$ is a finitely generated $S$-module, cf. \cite[8.4]{Ma}. Since
$\psi((\bX))A = \A$,  we get  $\GA = G_{(\bX)}(A)$ is a finitely generated
$G_{(\bX)}(S)$-module.

Notice $G_{(\bX)}(S) \cong
S/(\bX)[X_{1}^{*},\ldots, X_{n}^{*}]$ and $S/(\bX) \cong k[| Y_1,\ldots, Y_s|]$.
Furthermore
 $G_{(\bX)}(A) = \GA$ is generated in degree zero as a $G_{(\bX)}(S)$-module.
Set $c =  \dim S - \dim A  = \dim G_{(\bX)}(S)  - \dim \GA $.  Notice  that $ Y_1,\ldots, Y_s$ is a s.o.p of  $G_{(\bX)}(S)_0$ and $\GA_0$.
We apply Theorem \ref{hann} to get $\xi_{1},\ldots,\xi_{c} \in (\bX)  $ such that
 $\xi_{1}^{*}, \ldots, \xi_{c}^{*} \in G_{(\bX)}(S)_+  $ is a \emph{regular sequence}   in  $G_{(\bX)}(S)$ and
$\xi_{i}^{*} \in \ann_{G_{(\bX)}(S)} \GA$ for all $i$. We now apply  Lemma
\ref{annihilator} to get $u_1,\ldots, u_c  \in (\bX)$ an $S$-regular sequence in
$\ann_S (A)$ such that $u_{1}^{*},\ldots, u_{c}^{*}$ is a $G_{(\bX)}(S)$-regular
 sequence in $\ann_{G_{(\bX)}(S)} \GA $.  Set $\mathbf{u} = (u_1,\ldots, u_c) $,   $B = S/\mathbf{u} $, $\B = (\bX +\mathbf{u})/\mathbf{u}$.  Let $\n$ be the maximal ideal in $B$ and let $\ov{\phi} \colon B \rt A$
be the map induced by $\phi \colon  S \rt A$. Then it is clear that
$(B,\n,\B, \ov{\phi})$ is a CI approximation of $A$ \wrt \ $\A$.

\textbf{Case 2.} \textit{$A$ does not contain a field:} There exists a DVR $(D,\pi)$ and a local
ring homomorphism $\eta \colon D \rt A$ such that $\eta$ induces an isomorphism
 $D/(\pi) \rt A/\m$.  Set $S = D[| Y_1,\ldots, Y_s, X_1,\ldots,X_n|]$. Define
$\phi \colon S \rt A$, with  $\phi(d) =  \eta(d)$ for each $d \in D$ and maps  $Y_i$ to $y_i$ and $X_j$ to $x_j$ for $i= 1,\ldots, s$
and $j = 1,\ldots, n$.
 Set $\q = (Y_1,\ldots, Y_s)$.
Notice
\[
\frac{A}{(\q, X_1,\ldots,X_n)A} = \frac{A}{(\A, y_1,\ldots,y_s)} \quad \text{has finite length.}
\]
As $S$ is complete we get that $A$ is a finitely generated $S$-module.

Set $T = S/(\bX) = G_{\bX}(S)_0$. Notice $\dim T$ is one more than that of $A/\A = \GA_0$.
To deal with this situation we  use an argument from \cite{PuEuler}, which we repeat for the convenience of the reader. Let $\mathfrak{t}$ be the maximal ideal of $T$. Set $\ov{Y_i} = $ image of $Y_i$ in $T$. Set $V= A/\A$.
Note that
\[
\mathfrak{t} = \sqrt{\ann_T(V/\ov{\bY} V)} = \sqrt{\ann_T(V) + (\ov{\bY})}.
\]
So there exists $\ov{f} \in \ann_T(V) \setminus (\ov{\bY})$ such
that $\ov{f},\ov{Y_1},\ldots,\ov{Y_s}$ is an s.o.p. of $T$. Since $T$ is CM; $f,\ov{Y_1},\ldots,\ov{Y_s}$ is a $T$-regular
sequence.
So $X_1,\ldots,X_n,f,Y_1,\ldots,Y_s$ is a $S$-regular
sequence. Since $\ov{f} \in \ann_T(A/\A)$  we get $f A \subseteq \A A = (\bX)A$. Using the
determinant trick
 there exists
\[
\Delta = f^m + \alpha_{1}f^{m-1} + \ldots + \alpha_{m-1}f + \alpha_{m} \in \ann_S A \quad \& \alpha_i \in (\bX)^i; \ 1\leq i \leq m.
\]
 Notice that $\Delta = f^m \ \text{mod}( \bX)$. Thus
 $X_1,\ldots,X_n, \Delta,Y_1,\ldots,Y_s$ is a $S$-regular
sequence. So $\Delta$ is $S$-regular. Set $U = S/( \Delta)$ and $\C = ((\bX) + ( \Delta))/( \Delta)$.
 Notice that $A$ is a finitely generated $U$-module and $G_{\C}(A) = \GA$. Furthermore
$G_{\C}(U) \cong T/(f^m)[X_{1}^{*},\ldots, X_{n}^{*}]$ is  CI.

 Note that as a
 $G_{\C}(U)$ module $\GA $ is generated in degree zero. Furthermore $\dim G_{\C}(U)_0 = \dim \GA_0$. Set $c =  \dim U - \dim A  = \dim G_{\C}(U)  - \dim \GA $.
The subsequent argument is similar to  that of Case 1.
\end{proof}
\begin{remark}\label{high-red} It follows from the proof of Theorem \ref{equiTT} that if $\A$ is $\m$-primary there exists CI-approximations with arbitrary high reduction numbers. To see this note that
there exist a complete intersection $R$ (with $G(R)$ a complete intersection) and a   map $R \rt A$ and elements $f_1,\cdots, f_m \in  \ann_R A$ such that $f_1^*,\cdots, f_m^* \in \ann_{G(R)} \GA$ and $f_1^*,\cdots, f_m^*$ is $G(R)$-regular. Furthermore $B = R/(f_1,\cdots, f_m)$. Also note that $\deg f_i^* > 0$ for all $i$. We may simply take $B'$ to be
$R/(f_1^l,\cdots, f_m^l)$ for large $l$.

\end{remark}

The following result is needed for Theorem \ref{cannonical-module-Mod}.
\begin{lemma}\label{lift-sup-Gor-App}
Let $(A,\m)$ be a complete \CM \ local ring of dimension $d$ and let
$\A = (x_1,\ldots,x_d,w_1,\ldots, w_s)$ be a $\m$-primary ideal. Assume that $x_1,\ldots,x_d$ is an $A$-superficial
sequence. Then there exist $[B,\n,\B,\phi]$, a Gorenstein
approximation of $[A,\m,\A]$ such that
\begin{enumerate}
\item
there exists $v_1,\ldots,v_d \in \B$ with $\phi(v_i) = x_i$ for $i = 1,\ldots,d$.
\item
$v_{1}^{*}, \ldots,v_{d}^{*}$ is $\GB$-regular.
\end{enumerate}

\end{lemma}
\begin{proof}
The proof of Theorem \ref{equiTT} shows that there exists a \CM \ local ring $(R,\eta)$ of dimension $d+ s$, an  $\eta$-primary ideal
and a ring homomorphism $\psi \colon R \rt A$ such that
\begin{enumerate}[\rm(1)]
\item
$A$ is a finitely generated as a $R$-module (via $\psi$).
\item
\begin{enumerate}[\rm (a)]
  \item
  In the equicharacteristic case $R= k[[V_1,\ldots, V_d, W_1,\ldots, W_s]]$.
  \item
  In the mixed characteristic case $R = D/(f)[[V_1,\ldots,V_d, W_1,\ldots, W_s]]$ with $(D,\Pi)$ a DVR and $f \in (\Pi)$.
\end{enumerate}
\item
In both cases  $\q = (V_1,\ldots,V_d, W_1,\ldots, W_s)$, $\psi(V_i) = x_i$ for $i =1,\ldots,d$ and $\Psi(W_i)= w_i$ for $i = 1,\ldots,s$. Notice
$G_\q(R) = k[V_{1}^{*},\ldots, V_{d}^{*}, W_1^*,\ldots, W_s^*]$ in the equicharacteristic case  and $G_\q(R) = D/(f)[V_{1}^{*},\ldots, V_{d}^{*}, W_1^*,\ldots, W_s^*]$ in the mixed characteristic case.
\end{enumerate}
Set $c = \dim R - \dim A$.
Set $S = R/(V_1, \cdots, V_d)$ and $\ov{\q} = \q /(V_1, \ldots, V_d)$. Note $G_{\ov{\q}}(S) = G_\q(R)/(V_1^*,\ldots, V_d^*)$.
Set $ C = A/(x_1,\ldots,x_d)$

We then choose $\xi_1,\ldots,\xi_c \in \ov{\q}$ such that $\xi_1^*,\ldots, \xi_c^* \in \ann_{G_\q(R)} G_\A(C)$ and that it is a
regular sequence in $G_\q(R)$.
After
raising powers we may assume that $\deg \xi_i^* = r$ for all $i$. Set $M = A$ (note $A$ need not be cyclic as a $R$-module).
Fix $i$. Set $J = \A^{r+1}$. We have $\xi_i M \subseteq \A^{r+1} M + (V_1,\ldots, V_d)M$. By the \emph{determinant trick} there exists a monic polynomial
$f(X) \in A[X]$ such that
\[
 f(X) = X^n + a_1X^{n-1} + \cdots + a_{n -1}X + a_n \quad \text{with} \ a_i \in J^{i} + (V_1,\ldots, V_d).
\]
and $u_i = f(\xi_i) \in \ann M$
It can be eaily checked (as in proof of Lemma \ref{annihilator}) that in $G_\q(S)$ we have $u_i = (\xi_i^n)^* = (\xi_i^* )^n$.
Thus $V_1,^*\cdots, V_d^*, u_1^*,\ldots, u_c^*$ is a $G_\q(R)$-regular sequence. Then note that
$u_1^*,\ldots, u_c^*, V_1,^*\cdots, V_d^*$ is also a $G_\q(R)$-regular sequence.

Set $B = R/(\ub)$, $\B = \q/(\ub)$ and $\phi = \psi \otimes R/(\ub)$.  Notice $\GB = G_\q(R)/(\ub^*)$.
 Set $v_i$ = image of $V_i$ in $B$. Then note that $v_1, ^*, \ldots, v_d^*$ is a $G_\q(B)$-regular sequence and
  $[B,\n,\B,\phi]$, a Gorenstein
approximation of $[A,\m,\A]$.
\end{proof}
%End 101010101010101010101010101010101010101010101010101010101010
%End 101010101010101010101010101010101010101010101010101010101010

\section{Dual Filtrations mod a  superficial sequence}\label{section-dual-super}
The behavior of Dual filtrations mod a superficial sequence is a crucial result in our paper. Unfortunately the proof is  technical and unappealing.
\s \label{set-mod} \textbf{Setup:}
Let $(A,\m)$ be Gorenstein, $\A$-equimultiple with $\GA$ Gorenstein and  let $M$ be a $MCM$ A-module. We assume
 $A/\m$ is an infinite field.
 Let $\xb = x_1,\ldots, x_r$ be a $M \bigoplus A$ superficial sequence \wrt\ $\A$. Set $B = A/(\xb )$, $\B = \A/(\bX)$ and $N = M/\xb M$. The natural map $\GA/(\xb^*) \rt \GB$ is an isomorphism. Set
$$
 \eF_M = \left\{\Hom_A (M ,
\mathfrak{a}^n)\right\}_{\nZ}  \quad \text{and} \quad \eF_N = \left\{\Hom_B (N ,\mathfrak{n}^n)\right\}_{\nZ}.
$$
A natural question is when is
\begin{equation*} \label{sup-eqn2}
\frac{G (\eF_M , M^\dd)}{\xb^* G(\eF_M , M^\dd)} \cong G (\eF_N ,
N^\dd).
\end{equation*}
Theorem \ref{modFiltX}(2) asserts that  this is so when $G_\A(M)$ is \CM.

\textbf{Some Natural Maps:}
\s The natural maps $\Hom_A (M ,
\mathfrak{a}^n) \rt \Hom_B (N ,\mathfrak{b}^n)$ induces a map of $\ra$-modules
\begin{equation}\label{GammaEQ}
\Gamma_\xb^M \colon \R (\eF_M , M^\dd) \rt \R (\eF_N ,N^\dd).
\end{equation}
Set $\q = (x_1t,\ldots x_rt)$.
Clearly  $\q \R(\eF_M , M^\dd) \subseteq \ker \Gamma_\xb^M$. Since  $\ra$ is \CM, we have $\ra/\q\ra = \rb$.
 So we have a map $\rb$-modules
\begin{equation}\label{BGammaEQ}
\overline{ \Gamma_\xb^M} \colon \frac{\R (\eF_M , M^\dd)}{\q \R(\eF_M , M^\dd)  } \rt \R (\eF_N ,N^\dd).
\end{equation}

\s The map \ref{GammaEQ} induces a map of $\GA$-modules
\begin{equation}\label{thetaEQ}
\Theta_{\xb}^M \colon G (\eF_M , M^\dd) \xar  G (\eF_N ,
N^\dd).
\end{equation}
Clearly $\xb^*G(\eF_M , M^\dd) \subseteq \ker \Theta_{\xb}^M $.
 So we have a map $\GB$-modules
\begin{equation} \label{BthetaEQ}
 \ov{\Theta_{\xb}^M} \colon \frac{G (\eF_M , M^\dd)}{\xb^* G(\eF_M , M^\dd)}  \xar  G (\eF_N ,
N^\dd).
\end{equation}

\s \label{PiEQ} The natural map $\ra \rt \rb$ induces a natural map of $\ra$-modules
\[
 \Pi_\xb^M \colon \Hs_{\ra}(\R(\A, M),\ra)  \xar  \Hs_{\R_{\B}(B)}(\R(\B,N),\rb).
\]
Clearly $\q \left( \Hs_{\ra}(\R(\A, M),\ra) \right) \subseteq \ker \Pi_\xb^M$.

\begin{theorem} \label{modFiltX}[with assumptions as in \ref{set-mod}]
 If $G_{\mathfrak{a}}(M)$ is \CM \ then
\begin{enumerate}[\rm (1)]
 \item
$\ov{\Gamma_{\xb}^M}$ is an isomorphism of $\rb$-modules.
\item
$\ov{\Theta_{\xb}^M}$ is an isomorphism of $\GB$-modules.
\end{enumerate}
\end{theorem}
It suffices to prove for $r =1$ case. Clearly  if $\ov{\Gamma_{\xb}^M}$ is an isomorphism then $\ov{\Theta_{\xb}^M}$ is an isomorphism.
The following observation enables us to analyze the case  $r =1$.
\begin{observation}\label{nice}
 Let $x \in \A$. We have a commutative diagram of $\ra$-modules
\[
\xymatrixrowsep{3pc}
\xymatrixcolsep{2.5pc}
\xymatrix{
\R(\eF_M, M^\dd) \ar@{->}[r]^{\Gamma_x^M}
     \ar@{->}[d]_{\Psi_M}
& \R(\eF_N, N^\dd)
     \ar@{->}[d]^{\Psi_N}
\\
\Hs(\R(\A, M), \ra) \ar@{->}[r]_{\Pi_x^M}
     &   \Hs(\R(\B, N), \rb)
}
\]
Here $\Psi_M$ is as in \ref{main-1}. Notice that by \ref{main},  $\Psi_M$ and $\Psi_N$ are isomorphisms.
\end{observation}
\begin{remark}\label{niceR}
 It is convenient to look at \ref{nice} in the following way:

1. By \ref{GammaEQ} and \ref{BGammaEQ} we have a \textit{ complex} of $\ra$-modules:
\[
 \mathcal{C}\mathbf{\colon} \quad \quad 0 \rt \R(\eF_M, M^\dd)(-1) \xrightarrow{xt} \R(\eF_M, M^\dd)
\xrightarrow{\Gamma_x^M} R(\eF_N, N^\dd) \rt 0.
\]
2. Set $\R(M) = \R(\A,M)$, $\R = \ra$ and $\ov{\R} = \rb$. By \ref{PiEQ} we have a \textit{ complex} $\mathcal{D}$ of $\ra$-modules:
\[
  0 \rt \Hs_{\R}(\R(M), \R)(-1) \xrightarrow{xt} \Hs_{\R}(\R(M), \R)
\xrightarrow{\Pi_x^M} \Hs_{\ov{\R}}(\R(N), \ov{\R}) \rt 0.
\]
3. We consider $\mathcal{C}, \mathcal{D}$ as co-chain complexes starting at $0$; i.e.,
\[
 \mathcal{C}\mathbf{\colon} \quad  0 \rt \mathcal{C}^0 \rt \mathcal{C}^1 \rt \mathcal{C}^2 \rt 0 \quad
\text{\&} \quad \mathcal{D}\mathbf{\colon} \quad  0 \rt \mathcal{D}^0 \rt \mathcal{D}^1 \rt \mathcal{D}^2 \rt 0.
\]
4. By \ref{nice}, it follows that we have chain map of complexes of $\ra$-modules
\[
 \Xi \colon \mathcal{C} \xar \mathcal{D}
\]
with $\Xi^i$ an isomorphism for $i = 0,1,2$.

5. It can be easily verified that
$\mathcal{C}$ is exact  \ff \  $\mathcal{D}$ is exact.
\end{remark}
\begin{proof}[Proof of Theorem \ref{modFiltX}]
 It suffices to prove for $r =1$. Notice that (1) $\implies$  (2).

(1) By assumption  $x$ is $M \bigoplus A$-superficial.
 Set $u = x t \in \ra_1$. Since $\GA$ is \CM \ we have   $x^*$ is $\GA$-regular. So $u$ is $\ra$-regular.
It is clear that $\ra/u\ra = \rb$.
Set $\R(M) = \R(\A, M)$, $\R = \ra$. Also set $\ov{\R} = \rb$ and $\R(N) = \R(\B, M)$.
Since $x$ is $M$-superficial we have an isomorphism $\vartheta_M \colon \R(M)/u\R(M) \rt \R(N)$.
The exact sequence
$$0 \xar \R(-1) \xrightarrow{u} \R \xrightarrow{\rho} \ov{\R} \xar 0$$
yields an exact sequence of $\R$-modules
\begin{align*}
 0 &\rt \Hs_\R(\R(M), \R )(-1) \xrightarrow{u} \Hs_\R(\R(M), \R) \xrightarrow{\rho_M}  \Hs(\R(M), \ov{\R}) \\
   &\rt \Es_{\R}^{1}(\R(M), \R )(-1)  \rt \cdots
\end{align*}
Since $\R$ is Gorenstein and $\R(M)$ is a MCM $\R$-module we get $\Es_{\R}^{1}(\R(M), \R ) = 0$. Also notice that
$ \Hs_\R(\R(M), \ov{\R}) = \Hs_{\ov{\R}}(\R(M)/u\R(M), \ov{\R})$. Furthermore the isomorphism  $\vartheta_M $
induces an isomorphism
$$ \nu_M \colon \Hs_{\ov{\R}}(\R(M)/u\R(M), \ov{\R}) \rt \Hs_{\ov{\R}}(\R(N), \ov{\R}).$$
 It can be easily checked that
$ \nu_M  \circ \rho_M = \Pi^{M}_{x}$.

Therefore the complex $\mathcal{D}$ in \ref{niceR}.2 is exact. By  \ref{niceR}.5 we get that the
complex $\mathcal{C}$ in \ref{niceR}.1 is exact. Thus $\Gamma_x^M$ is onto and
$\ker \Gamma_x^M = u\R(\eF_M, M)$. Therefore $\ov{\Gamma_x^M}$ is an isomorphism.
\end{proof}

%%%%%%%%%%%%%%5555555555555555555555555555555555555555555555555555555555555555555
\section{Some Preliminaries to prove Itoh's conjecture}\label{Lprop}
In \cite{Pu5} we introduced a new technique to investigate problems relating to associated graded modules.
In this section we collect all the relevant results which we proved in \cite{Pu5}. Throughout thus section
$(A,\m)$ is a Noetherian local ring with infinite residue field, $M$ is a \emph{\CM }\ module of dimension
$r \geq 1$ and $I$ is an $\m$-primary ideal.

\s \label{mod-struc} Set $\Sc = A[It]$;  the Rees Algebra of $I$. Set
$L^{I}(M) = \bigoplus_{n\geq 0}M/I^{n+1}M$. We note that $L^I(M)(-1) =  M[t]/\R(I, M)$.  So  $L^{I}(M)$ is a $\Sc$-module. Note that $L^I(M)$ is \emph{not} a finitely generated $\Sc$-module.

 \s Set $\M = \m\oplus \Sc_+$. Let $H^{i}(-) = H^{i}_{\M}$ denote the $i^{th}$-local cohomology functor \wrt \ $\M$. Recall a graded $\Sc$-module $L$ is said to be
*-Artinian if
every descending chain of graded submodules of $L$ terminates. For example if $E$ is a finitely generated $\Sc$-module then $H^{i}(E)$ is *-Artinian for all
$i \geq 0$.

\begin{definition}(\cite[sect. 6]{HeL})
   Consider the following chain of submodules of
$M$:
\[
IM \sub (I^2M\colon_M I) \sub (I^3M\colon_M I^2) \sub \ldots \sub(I^{n+1}M \colon_M I^n)\sub \ldots
\]
As $M$ is Noetherian this chain stabilizes.
The stable value is denoted as $\widetilde{IM}$ and is called
 the \emph{Ratliff-Rush submodule of $M$ \wrt \ $I$}.
 The filtration
$\{\wt{I^nM}\}_{n \geq 1}$ is called the \textit{Ratliff-Rush filtration} of $M$
\wrt \ $I$.
\end{definition}

\s It can be shown that if $\depth M >0$ then $\wt{I^nM} = I^nM$ for all $n \gg 0$.
Furthermore if $M = A$ then $\wt{I} \subseteq \ov{I}$,  where $\ov{I}$ is the integral closure of $I$.

\s \label{zero-lc} In \cite[4.7]{Pu5} we proved that
\[
H^{0}(L^I(M)) = \bigoplus_{n\geq 0} \frac{\wt{I^{n+1}M}}{I^{n+1}M}.
\]
\s \label{Artin}
For $L^I(M)$ we proved that for $0 \leq i \leq  r - 1$
\begin{enumerate}[\rm (a)]
\item
$H^{i}(L^I(M))$ are  *-Artinian; see \cite[4.4]{Pu5}.
\item
$H^{i}(L^I(M))_n = 0$ for all $n \gg 0$; see \cite[1.10 ]{Pu5}.
\item
 $H^{i}(L^I(M))_n$  has finite length
for all $n \in \mathbb{Z}$; see \cite[6.4]{Pu5}.
\item
$\lambda(H^{i}(L^I(M))_n)$  coincides with a polynomial for all $n \ll 0$; see \cite[6.4]{Pu5}.
\end{enumerate}

\s \label{I-FES} The natural maps $0\rt I^nM/I^{n+1}M \rt M/I^{n+1}M \rt M/I^nM \rt 0 $ induce an exact
sequence of $R(I)$-modules
\begin{equation}
\label{dag}
0 \xar G_{I}(M) \xar L^I(M) \xrightarrow{\Pi} L^I(M)(-1) \xar 0.
\end{equation}
We call (\ref{dag}) \emph{the first fundamental exact sequence}.  We use (\ref{dag}) also to relate the local cohomology of $G_I(M)$ and $L^I(M)$.

\s \label{II-FES} Let $x$ be  $M$-superficial \wrt \ $I$ and set  $N = M/xM$ and $u =xt \in \Sc_1$. Notice $L^I(M)/u L^I(M) = L^I(N)$.
For each $n \geq 1$ we have the following exact sequence of $A$-modules:
\begin{align*}
0 \xar \frac{I^{n+1}M\colon x}{I^nM} \xar \frac{M}{I^nM} &\xrightarrow{\psi_n} \frac{M}{I^{n+1}M} \xar \frac{N}{I^{n+1}N} \xar 0, \\
\text{where} \quad \psi_n(m + I^nM) &= xm + I^{n+1}M.
\end{align*}
This sequence induces the following  exact sequence of $\Sc$-modules:
\begin{equation}
\label{dagg}
0 \xar \Bcal^{I}(x,M) \xar L^{I}(M)(-1)\xrightarrow{\Psi_u} L^{I}(M) \xrightarrow{\rho^x}  L^{I}(N)\xar 0,
\end{equation}
where $\Psi_u$ is left multiplication by $u$ and
\[
\Bcal^{I}(x,M) = \bigoplus_{n \geq 0}\frac{(I^{n+1}M\colon_M x)}{I^nM}.
\]
We call (\ref{dagg}) the \emph{second fundamental exact sequence. }

\s \label{long-mod} Notice  $\lambda\left(\Bcal^{I}(x,M) \right) < \infty$. A standard trick yields the following long exact sequence connecting
the local cohomology of $L^I(M)$ and
$L^I(N)$:
\begin{equation}
\label{longH}
\begin{split}
0 \xar \Bcal^{I}(x,M) &\xar H^{0}(L^{I}(M))(-1) \xar H^{0}(L^{I}(M)) \xar H^{0}(L^{I}(N)) \\
                  &\xar H^{1}(L^{I}(M))(-1) \xar H^{1}(L^{I}(M)) \xar H^{1}(L^{I}(N)) \\
                 & \cdots \cdots \\
               \end{split}
\end{equation}

\s \label{Artin-vanish} We will use the following well-known result regarding *-Artinian modules quite often:

Let $V$ be a *-Artinian $\Sc$-module.
\begin{enumerate}[\rm (a)]
\item
$V_n = 0$ for all $n \gg 0$
\item
If $\psi \colon V(-1) \rt V$ is a monomorphism then $V = 0$.
\item
If $\phi \colon V \rt V(-1)$ is a monomorphism then $V = 0$.
\end{enumerate}

 \s \label{power-of-I} Set $L = L^I(M)$,
 \begin{align*}
\xi_I(M) &:= \underset{0 \leq i \leq r-1}\min\{ \ i \ \mid H^{i}(L)_{-1} \neq 0 \ \text{or} \ \ell(H^{i}(L)) = \infty \}.\\
\amp_I(M) &:= \max\{\ |n| \ \mid H^{i}(L)_{n-1} \neq 0 \  \text{for} \ i = 0,\ldots, \xi_I(M) - 1 \}.
\end{align*}
In \cite[7.5]{Pu5} we showed that
\[
 \depth G_{I^l}(M) = \xi_I(M) \ \text{for all} \ l > \amp_I(M).
\]

\s \label{normal}
Recall an ideal $I$ is said to be asymptotically normal if $I^n$ is integrally closed for all $n \gg 0$.
If $I$ is a asymptotically normal  $\m$-primary ideal  and $\dim A \geq 2$ then by a result of Huckaba and Huneke \cite[3.8]{HH}, $\depth G_{I^l}(A) \geq 2$
for all $l \gg 0$(also see \cite[7.3]{Pu5}). It  also follows from \cite[9.2]{Pu5} that in this case $H^1(L)_n = 0$ for $n < 0$. In particular $\ell(H^1(L)) < \infty$ (here $L = L^I(M)$).

\s \label{asympCM} We will also need the following fact for general filtration's (and was proved for $I$-adic filtration's in \cite[9.2]{Pu5}). Let $\eF = \{ M_n \}_{n \in \Z}$ be an $I$-stable filtration with $M_n = M$ for $n<0$. If the associated graded module of the Veronese $G(\eF^{< c >}, M)$ has depth $\geq 2$ for some $c \geq 1$ then
$H^1(G(\eF, M))_n = 0$ for $n < 0$.
To see this set $L(\eF, M)= \bigoplus_{n \geq 0} M/M_{n+1}$. Then as before $L(\eF, M)$ is a $\Sc$-module.
We note that for all $l \geq 1$ we have
\[
\left(L(\eF,M)(-1)\right)^{<l>} = \bigoplus_{n \geq 0} M/M_{nl} =  L(\eF^{<l>},M)(-1)
\]
As $\depth G(\eF^{< c >}, M) \geq 2$ it follows from an analogue of \ref{Artin}; \ref{dag} for filtration's and \ref{Artin-vanish} that
$H^i(L(\eF^{<c>},M)) = 0$ for $i = 0,1$. As the Veronese functor commutes with local cohomology we get that
$H^i(L(\eF, M))_{nc-1} = 0$ for all $n \in \Z$. In particular $H^1(L(\eF, M))_{-1} = 0$. Let $x$ be $M$-superficial \wrt \ $\eF$.
Set $N = M/xM$. Then by an analogue of \ref{longH} to filtration's we have  an exact sequence for all $n \in \Z$
\[
H^0(L(\ov{\eF}, N))_n \rt H^1(L(\eF, M))_{n-1} \rt H^1(L(\eF, M))_{n}
\]
As $H^0(L(\ov{\eF}, N))_n = 0$ for $n < 0$ and $H^1(L(\eF, M))_{-1} = 0$;  an easy induction yield's $H^1(L(\eF, M))_{n} = 0$ for $n < 0$.
By the exact sequence \ref{dag} and as $H^0(L(\eF, M)_n = 0$ for $n < 0$ we get that $H^1(G(\eF, M))_n = 0$ for $n < 0$.
\section{ Proof of Itoh's-conjecture}
In this section we give a proof of Itoh's conjecture for normal ideals.
\begin{theorem}\label{itoh}
 Let $(A,\m)$ be a \CM \  local ring of dimension $d \geq 1$ and let $\A$ be a normal $\m$-primary ideal.
 If $e_3^\A(A) = 0$ then $G_\A(A)$ is \CM.
\end{theorem}

\begin{remark}\label{r1-itoh}
The following assertion is well-known.
 If $\A^n$ is integrally closed for all $n$ then $\A^n$ is Ratliff-Rush for all $n \geq 1$.
 It follows that $\depth G_\A(A) \geq 1$.
\end{remark}

We first show that
\begin{lemma}\label{small}
 Theorem \ref{itoh} holds  if $\dim A = 1, 2$.
 \end{lemma}
 This Lemma is certainly well-known to the experts. However we provide a proof
 for the convenience of the reader.
\begin{proof}[Proof of Lemma \ref{small}]
 If $d = 1$ then by \ref{r1-itoh} the result holds.

 Now assume $d = 2$. Note that we may assume that $A$ has an infinite residue field. Let $\C$ be a minimal reduction of $\A$.
 Set $\sigma_i = \lambda(\A^{i+1}/\C \A^i)$. Then it follows from a result of Huneke \cite[2.4]{Hun}
 that the $h$-polynomial of $A$ \wrt \ $\A$ is
 \[
  h(z) =  \ell(A/\A) + (\sigma_0 - \sigma_1)z + (\sigma_1 - \sigma_2)z^2 + \cdots + (\sigma_{s-2} - \sigma_{s-1})z^{s-1} + \sigma_{s-1} z^s.
 \]
It follows that
\begin{align*}
 e_1^\A(A) &= \sum_{i \geq 0} \sigma_i, \\
 e_2^\A(A) &= \sum_{i \geq 1} i \sigma_i, \text{and} \\
 e_3^\A(A) &= \sum_{i \geq 2} \binom{i}{2} \sigma_i.
 \end{align*}
As $e_3^\A(A) = 0$ we have $\sigma_2 = 0$. So we have $\A^3 = \C \A^2$. As $\A$ is integrally closed  we also have
$\A^2 \cap \C = \C \A$. The result follows from Valabrega-Valla Theorem
\cite[2.3]{VV}.
\end{proof}
\s Recall that an ideal $\A$ is said to be \textit{asymptotically normal} if $\A^n$ is integrally closed for all $n \gg 0$.
The crucial result to prove Itoh's conjecture is the following:
\begin{lemma}\label{crucial-itoh}
Let $(A,\m)$ be a \CM \    local ring of dimension $3$ and let $\A$ be an asymptotically  normal $\m$-primary ideal with $e_3^\A(A) = 0$. Set $L^\A(A) = \bigoplus_{n \geq 0}A/\A^{n+1}$ considered as a  module over the Rees algebra $\Sc = A[\A t]$. Let $\M = \m\oplus \Sc_+$ be the maximal homogeneous ideal of $\Sc$. Then the local cohomology modules $H^i_\M(L^\A(A))$ vanish for $i = 1, 2$
\end{lemma}
We will also need to extend Lemma \ref{crucial-itoh} to dimensions $d \geq 4$.
\s \emph{Remark and a Convention:} Note that all the relevant graded modules considered  below are modules over the Rees algebra $\Sc = A[\A t]$. Also all local cohomology will
taken over $\M = \m\oplus \Sc_+$  the maximal homogeneous ideal of $\Sc$.
We also note that if $x$ is $A$-superficial \wrt \ $A$ then the Rees algebra $\Sc^\prime  = A/(x)[\A/(x)t]$ is a quotient of $\Sc$. As we are only interested in vanishing of certain local-cohomology modules, by the independence theorem of local cohomology
it does not matter if we take local cohomology of an $\Sc^\prime$-module \wrt \ $\M^\prime$ or over $\M$ (and considering the module in question as an $\Sc$-module). So throughout we will only write $H^i(-)$ to mean $H^i_\M(-)$.

\begin{lemma}\label{crucial-itoh-2}
Let $(A,\m)$ be a \CM \ local ring of dimension $d \geq 3$ and let $\A$ be an asymptotically  normal $\m$-primary ideal with $e_3^\A(A) = 0$.  Then the local cohomology modules $H^i(L^\A(A))$ vanish for $1 \leq i \leq d -1$.
\end{lemma}
\begin{proof}
We prove the result by induction on $d \geq 3$.
For $d = 3$ the result follows from Lemma \ref{crucial-itoh}. Now assume $d \geq 4$ and the result has been proved for $d-1$.
After passing through a general extension (see \ref{AtoA'}(d)) we may choose $x$, an $A$-superficial element
\wrt  \ $\A$ such that in the ring $B = A/(x)$ the ideal $\B = \A/(x)$ is asymptotically normal. Furthermore note that $e_3^\B(B) = 0$.
Set $L = L^\A(A)$ and $\ov{L} = L^\B(B)$. By induction hypothesis we have $H^i(\ov{L}) = 0$ for $1 \leq i \leq d-2$. By \ref{longH} we get a surjective map
$H^1(L)(-1) \rt H^1(L)$ and for $2 \leq i \leq d-1$ injections
$H^i(L)(-1) \rt H^i(L)$. By \ref{Artin-vanish} it follows that $H^i(L) = 0$ for $i = 2,\cdots d - 1$. By \ref{normal} we get that $H^1(L)$ has finite length. So the surjection
$H^1(L)(-1) \rt H^1(L)$ induces an isomorphism $H^1(L)(-1) \cong H^1(L)$ and this forces $H^1(L) = 0$.
\end{proof}

We now give a proof of Theorem \ref{itoh} assuming Lemma \ref{crucial-itoh}.
\begin{proof}[Proof of Theorem \ref{itoh}]
If $d = \dim A = 1,2$, the result follows from Lemma \ref{small}. If $d \geq 3$
then by Lemma \ref{crucial-itoh-2} we get $H^i(L^\A(A)) = 0$  for $1 \leq i \leq d -1$. Also as $\A$ is normal,  in-particular $\A^n$ is Ratliff-Rush for all
$n \geq 1$. So by \ref{zero-lc} we get $H^0(L^\A(A)) = 0$. By taking cohomology of the first fundamental sequence \ref{dag} we get that $H^i(G_\A(A)) = 0$
for  $0 \leq i \leq d -1$. Thus $G_\A(A)$ is \CM.
\end{proof}

\s \label{red-itoh} Thus to prove Itoh's conjecture all we have to do is to prove Lemma \ref{crucial-itoh}. This requires several preparatory results.
\emph{For the rest of the section we will assume $\dim A = 3$ and $\A$ is an asymptotically normal ideal with $e_3^\A(A) = 0$}.
We first show
\begin{lemma}\label{asymp-itoh}
(with hypotheses as in \ref{red-itoh}.) Then $G_{\A^n}(A)$ is \CM \ for $n \gg 0$
and $\red(\A^n) \leq 2$ for $n \gg 0$.
\end{lemma}
\begin{remark}
Lemma \ref{asymp-itoh} is certainly known to the experts. We provide a proof for the convenience of the reader.
\end{remark}
\begin{proof}[Proof of Lemma \ref{asymp-itoh}]
We may assume that the residue field of $A$ is infinite. Let $\C = (x_1, x_2, x_3)$ be a minimal reduction of $\A$.

Notice that for all $n \geq 1$ we have $e_3^{\A^n}(A) = e_3^\A(A) = 0$. Furthermore $\C^{[n]} = (x_1^n, x_2^n, x_3^n)$ is  a minimal reduction of $\A^n$.
By a result of Huckaba and Huneke \cite[3.8]{HH}  we have $\depth G_{\A^n}(A) \geq 2$
for all $n \gg 0$, say $n \geq n_0$. As $\A$ is asymptotically normal we may also assume $\A^n$ is integrally closed for all $n \geq n_0$

Fix $n \geq n_0$.  Set $\sigma_i = \ell( (\A^n)^{i+1}/ \C^{[n]} (\A^n)^{i})$. Then by a result of Marley \cite[3.9]{Mar}, it follows that
that the $h$-polynomial of $A$ \wrt \ $\A^n$ is
 \[
  h(z) =  \ell(A/\A^n) + (\sigma_0 - \sigma_1)z + (\sigma_1 - \sigma_2)z^2 + \cdots + (\sigma_{s-2} - \sigma_{s-1})z^{s-1} + \sigma_{s-1} z^s.
 \]
It follows that
\begin{align*}
 e_1^{\A^n}(A) &= \sum_{i \geq 0} \sigma_i, \\
 e_2^{\A^n}(A) &= \sum_{i \geq 1} i \sigma_i, \text{and} \\
 e_3^{\A^n}(A) &= \sum_{i \geq 2} \binom{i}{2} \sigma_i.
 \end{align*}
As $e_3^{\A^n}(A) = 0$ we have $\sigma_2 = 0$. So we have $(\A^n)^3 = \C^{[n]} (\A^n)^2$. As $\A^n$ is integrally closed  we also have
$(\A^n)^2 \cap \C^{[n]} = \C^{[n]} \A^n$.  So by Valabrega-Valla Theorem \cite[2.3]{VV} we have that $G_{\A^n}(A)$ is \CM.
\end{proof}
Next we show
\begin{lemma}\label{rachel-itoh}
(with hypotheses as in \ref{red-itoh}.) We have
\begin{enumerate}[ \rm(1)]
\item
$a(G_\A(A)) < 0$.
\item
$H^2(L^\A(A)) = 0$.
\item
$\sum_{i = 0}^{2}(-1)^i\ell(H^i(G_\A(A)))  = 0.$
\item
For $i = 0, 1, 2$ we have $H^i(G_\A(A))_n = 0$ for $n < 0$.
\end{enumerate}
\end{lemma}
\begin{proof}
 (1) We have $\red_{\C^{[n]}}(\A^n) \leq 2$ for all $n \gg 0$. By a result of Hoa, \cite[2.6]{Hoa}, the result follows.

 (2) Set $G = G_\A(A)$ and $L = L^\A(A)$. Also let $S = A[\A t]$ be the Rees algebra of $A$ \wrt \ $\A$ and let $\M$ be its
 maximal homogeneous ideal. Throughout let $H^i(-) = H^i_\M(-)$.
 The first fundamental exact sequence $0 \rt G \rt L \rt L(-1) \rt 0$ yields an exact sequence in cohomology
 \[
 H^2(L)_n \rt H^2(L)_{n-1} \rt H^3(G)_n.
 \]
As $a(G) < 0$ we get for $n \geq 0$, surjections $H^2(L)_{n} \rt H^2(L)_{n-1}$. As $H^2(L)_n = 0$ for $n \gg 0$ we get
$H^2(L)_n = 0$ for $n \geq -1$.

After passing through a general extension we may choose $x$, an $A$-superficial element
\wrt  \ $\A$ such that in the ring $B = A/(x)$ the ideal $\B = \A/(x)$ is asymptotically normal. Also notice $\dim B = 2$.
Set $\ov{L} = L^\B(B)$. By \cite[9.2]{Pu5} we get that $H^1(\ov{L})_n = 0$ for $n < 0$.
By \ref{longH} we have an exact sequence
\[
 H^1(\ov{L})_n \rt H^2(L)_{n-1} \rt H^2(L)_n
\]
By setting $n = -1$ we get $H^2(L)_{-2} = 0$. Iterating we get $H^2(L)_n = 0$ for all $n \leq -2$.

It follows that $H^2(L) = 0$.

(3) We first note that as $G_{\A^n}(A)$ is \CM \ for all $n \gg 0$ the ring $G_\A(A)$ is generalized \CM, see \cite[7.8]{Pu5} We also have $H^0(L)$
and $H^1(L)$ have finite length (see \ref{zero-lc} and \ref{normal}).

The first fundamental exact sequence $0 \rt G \rt L \rt L(-1) \rt 0$ yields an exact sequence in cohomology
 \begin{align*}
  0 &\rt H^0(G) \rt H^0(L) \rt H^0(L)(-1) \\
  &\rt H^1(G) \rt H^1(L) \rt H^1(L)(-1) \\
  &\rt H^2(G) \rt H^2(L) = 0.
 \end{align*}
 Taking lengths, the result follows.

 (4) As $\A$ is normal we have that $H^1(L)_n = 0$ for $n < 0$ (see \cite[9.2]{Pu5}). Also by \ref{zero-lc} we get $H^0(L)_n = 0$
 for $n < 0$. The result follows from the above exact sequence in cohomology.
 \end{proof}

\begin{remark}\label{apply-gor-approx}
 Till now we have not used our theory of complete intersection approximation. We do it now.
 We first complete $A$. Let $\A = (a_1,\ldots, a_n)$.  We take a general extension $A' = A[X_1,\ldots, X_n]_{\m A[X_1,\ldots, X_n]}$. Let $y = \sum_{i=1}^{n}a_iX_i$.
 Then the ideal $J = \A'/(y)$ is asymptotically normal. Set $B = A'/(y)$. By a result of Huckaba and Huneke we get that there exists $c$ such that $G_{J^n}(B)$ is \CM \ for $n \geq c$.  We take a CI-approximation $(R,\n,\B,\psi)$ of $A$ \wrt \ $\A$ (note \emph{not of} $A'$). By \ref{high-red} we may assume $\red(\B, R) \geq c+2$. By construction note that $\B$ is generated by $b_1,\ldots,b_n$ and $\psi(b_i) = a_i$ for all $i$. Now set $R' = R[X_1,\ldots, X_n]_{\n R[X_1,\ldots, X_n]}$. By \ref{lying-above} we get that $A' = A\otimes_R R'$. We note that $(R',\n',\B',\psi')$ is a CI-approximation of $(A',\m',\A')$. Set $z = \sum_{i=1}^{n}b_iX_i$. Then note $\psi'(z) = y$.
Also note that $z$ is $R'$-superficial \wrt \ $\B'$. We now complete $R'$ \wrt \ $\n'$.
 Thus we may assume that our ring $(A,\m, \A)$ has a CI-approximation $[T,\tf,\q,\psi]$ such that
 \begin{enumerate}
 \item Both $A$ and $T$ are complete.
   \item reduction number of $\q$ is $\geq c + 2$.
   \item there exists $z \in \q$ which is $A$-superficial \wrt \ $\A$ such that \\ $G_{\A^n}(A/zA)$ is \CM \ for all $n \geq c$. Furthermore $z$ is $G_\q(T)$-regular.
   \item $A$ has a canonical module $\omega_A$.
   \end{enumerate}
\end{remark}
We first show
\begin{lemma}\label{gcm}
 (with hypotheses as in \ref{apply-gor-approx}.) We have
 \begin{enumerate}[\rm (1)]
  \item For any prime $P$ in $\R_\q(T)$ with $\htt P \leq 3$ we have $\R_\A(A)_P$ is \CM.
  \item
  $H^i_\N(\R_\A(A))$ has finite length for $0 \leq i \leq 3$.
 \end{enumerate}
\end{lemma}
\begin{proof}
 As $\R_\q(T)$ is a Gorenstein ring the assertion (2) follows from (1) by local duality.

 (1) We first consider the case $t^{-1} \notin P$. Then $(\R_\A(A))_P$ is a localization of
 $(R_\A(A))_{t^{-1}} = A[t,t^{-1}]$ which is \CM.

 Next consider the case when $t^{-1} \in P$. Set $Q = P/(t^{-1})$ a prime ideal of height $\leq 2$ in $G_\q(T)$. As
 $G_\A(A)$ is generalized \CM \ and since $G_\q(T)$ is Gorenstein, by local duality we have that $(G_\A(A))_Q$ is \CM.
 It follows that $(\R_\A(A))_P$ is \CM.
\end{proof}
As a consequence we get
\begin{corollary}\label{exact-loc}
(with hypotheses as in \ref{apply-gor-approx}.) We have
\begin{enumerate}[\rm (1)]
  \item $H^3_\N(\R_\A(A)) = 0$
  \item an exact sequence
\[
 0 \rt H^3_\N(G_\A(A)) \rt H^4_\N(\R_\A(A))(+1) \xrightarrow{t^{-1}} H^4_\N(\R_\A(A)) \rt 0.
\]
  \item $\Ext^1_{\R_\q(T)}(\R_\A(A),\R_\q(T)) = 0$.
\end{enumerate}
\end{corollary}
\begin{proof}
 The short exact sequence $0 \rt \R_\A(+1)  \xrightarrow{t^{-1}} H^4_\N(\R_\A)  \rt G_\A(A) \rt 0$, yields an exact sequence
 of local cohomology modules
 \begin{align*}
  0 &\rt H^0_\N(\R_\A(A))(+1) \xrightarrow{t^{-1}} H^0_\N(\R_\A(A)) \rt  H^0_\N(G_\A(A)) \\
  &\rt H^1_\N(\R_\A(A))(+1) \xrightarrow{t^{-1}} H^1_\N(\R_\A(A)) \rt  H^1_\N(G_\A(A)) \\
   &\rt H^2_\N(\R_\A(A))(+1) \xrightarrow{t^{-1}} H^2_\N(\R_\A(A)) \xrightarrow{\pi}  H^2_\N(G_\A(A)) \\
   &\rt H^3_\N(\R_\A(A))(+1) \xrightarrow{t^{-1}} H^3_\N(\R_\A(A))
 \end{align*}
Let $K$ be the cokernel of $\pi$. We note that $H^i_\N(G_\A(A)) \cong H^i(G_\A(A))$. As all the module in the above exact sequence have finite length
we have
\[
 \ell(K) = \sum_{i = 0}^{2} (-1)^i\ell(H^i(G_\A(A))  = 0 \quad \text{by Lemma \ref{rachel-itoh}(3)}.
\]
Thus we have an inclusion
\[
 H^3_\N(\R_\A(A))(+1) \xrightarrow{t^{-1}} H^3_\N(\R_\A(A))
\]
As $\ell(H^3_\N(\R_\A(A))) $ is finite we have an isomorphism
$H^3_\N(\R_\A(A))(+1) \cong H^3_\N(\R_\A(A))$. Again as $\ell(H^3_\N(\R_\A(A))) $ is finite this implies
that $H^3_\N(\R_\A(A)) = 0$.  So (1), (2) follows. Also by local duality we have (3).
\end{proof}
The following result is a crucial ingredient in proving Lemma \ref{crucial-itoh}.
\begin{theorem}\label{ing}(with hypotheses as in \ref{apply-gor-approx}.)
 Let $E_G$ be the injective hull of $k$ considered as a
 $G_\q(T)$-module. Set $W = \Hom_{G_\q(T)}(H^3(G_\A(A)), E_G)$. Then
 \begin{enumerate}[\rm (1)]
  \item
 there exists an $\A$-good filtration $\eF$ on
 $\omega_A$ such that $G_\eF(\omega_A) \cong W[1]$
 \item
 $ \depth G_\eF(\omega_A) \geq 2$. Furthermore $z$ is $G_\eF(\omega_A)$-regular.
 \item
 $h_\eF(z) = a_0 + a_1z + a_2 z^2$.
 \item
 Set $B = A/zA$. Let $\ov{\eF}$ be the quotient filtration on $\omega_B = \omega_A/z\omega_A$. Then $H^1(G(\ov{\eF},\omega_B)_n = 0$ for $n < 0$.
 \end{enumerate}
\end{theorem}
\begin{proof}
(1)Let $E_\R$ be the injective hull of $k$ considered as a $\R_\q(T)$-module.  Set $(-)^\vee = \Hom_{\R_q(T)}(-, E_\R)$.
Dualizing the exact sequence in Lemma \ref{exact-loc} we get a sequence of $\R_\q(T)$-modules
\[
 0 \rt H^4_\N(\R_\A(A))^\vee \xrightarrow{t^{-1}}  H^4_\N(\R_\A(A))^\vee(+1) \rt W \rt 0.
\]
As $R_\q(T)$ is Gorenstein then by local duality we have an exact sequence
\[
 0 \rt \Hom_{\R_q(T)}(\R_\A(A), \R_\q(T)) \xrightarrow{t^{-1}} \Hom_{\R_q(T)}(\R_\A(A), \R_\q(T))(+1) \rt X \rt 0
\]
where $X \cong W$ upto shift.
By Theorem \ref{main} there exists an $\q$-good filtration $\eF$ on $\omega_A \cong \Hom_T(A,T)$ such that
\[
 \R(\eF, \omega_A) \cong \Hom_{\R_q(T)}(\R_\A(A), \R_\q(T))
\]
Also notice that as $\q A = \A$ it follows that $\eF$ is in fact an $\A$-good filtration on $\omega_A$. Note after a shift $G(\eF)= W[1]$.

(2)
We note that upto shift $W \cong \Hom_{G_\q(T)}(G_\A(A), G_\q(T))$. The result follows.

(3) To prove this result we first \\
\emph{Claim:} The Hilbert series of $W = \bigoplus_{n \in \Z}W_n$ is
\[
H_W(z) = \sum_{n \in \Z}\ell(W_n)  = \frac{a_0z + a_1z^2 + a_2z^3}{(1-z)^3}.
\]
 Set $G = G_\A(A)$.
Let $V = H^3(G)$. As $a(G) < 0$ we have that $V_n = 0$ for $n\geq 0$. So $W_m = 0$
for $m \leq 0$.
By Grothendieck-Serre formula we have
\[
H_\A(n) - P_\A(n) = \sum_{i = 0}^{3}(-1)^i \ell(H^i(G)_n).
\]
Where $H_\A$ (and $P_\A(z)$) are Hilbert function(and Hilbert polynomial) of $A$ \wrt \ $\A$. As $H^i(G)_n = 0$ for $n < 0$ for $i = 0,1,2$ we get that
\[
\ell(V_n) = P_\A(n) \quad \text{for all} \ n < 0.
\]
It follows that $\ell(W_m) = P_\A(-m)$ for all $m \geq 1$. So $P_\A(-z)$ is the Hilbert polynomial of $W$ (it should be remarked that $\deg(P_\A(z)) = 2$).
As postulation number of $W$ is one  (and as $W_0 = 0$)  it follows from \cite[4.1.12]{BH} the Hilbert series of
$W$ is
\[
 \frac{a_0z + a_1z^2 + a_2z^3}{(1-z)^3}
\]
As $G_\eF(A) \cong W[1]$ the result follows.

(4) We have an exact sequence
\[
0 \rt \R_\q(T)(-1) \xrightarrow{zt} \R_\q(T) \rt   \R_\q(\ov{T}) \rt 0.
\]
This yields an exact sequence
\begin{align*}
  0 \rt \Hom_{\R_\q(T)}(\R_\A(A), \R_\q(T))(-1) &\xrightarrow{zt} \Hom_{\R_\q(T)}(\R_\A(A), \R_\q(T)) \rt  \\
  \Hom_{\R_\q(T)}(\R_\A(A), \R_\q(\ov{T})) &\rt \Ext^1_{\R_\q(T)}(\R_\A(A), \R_\q(T))(-1) \rt.
\end{align*}
By \ref{exact-loc} we get that $\Ext^1_{\R_\q(T)}(\R_\A(A), \R_\q(T)) = 0$.
We note that
\[
\Hom_{\R_\q(T)}(\R_\A(A), \R_\q(\ov{T})) = \Hom_{\R_\q(T)}(\R_\A(A)/zt \R_\A(A), \R_\q(\ov{T})).
\]
Set  $\B = \A/z A$. We have an exact sequence
\[
0 \rt F = \bigoplus_{n \geq 1}\frac{\A^n \cap z \A^n}{ z\A^{n-1}} \rt \R_\A(A)/zt \R_\A(A) \rt \R_\B(B) \rt 0.
\]
As $z$ is $A$-superficial \wrt \ $\A$  we get that $F$ has finite length. As $\R_\q(\ov{T})$ is \CM \  and has dimension $3$ we get
\[
\Hom_{\R_\q(T)}(\R_\A(A)/zt \R_\A(A), \R_\q(\ov{T})) \cong \Hom_{\R_\q(T)}(\R_\B(B), \R_\q(\ov{T})).
\]
Thus we have an exact sequence
\[
0 \rt \R(\eF, \omega_A)(-1)\xrightarrow{zt} \R(\eF,\omega_A) \rt \R(\eG,\omega_B) \rt 0;
\]
where $\eG$ is the dual filtration on $\omega_B$. Now by our construction $\R_\B^n(B)$ is \CM \ for $n \geq c$. As $s = \red(\q, T) = \red(\ov{\q}, \ov{T}) \geq c + 2$
and $G_{\ov{\q}}(\ov{T})$ is Gorenstein we get by \cite[Theorem 1]{Ooishi-gor-powers}  that $G_{\ov{\q}^{s-1}}(\ov{T})$ is Gorenstein. It follows that $\R_{\q^{s-1}}(\ov{T})$ is Gorenstein.
It follows that $\R(\eG^{<s-1>},\omega_B)$ is \CM. The result now follows from \ref{asympCM}.
\end{proof}

We now give
\begin{proof}[Proof of Lemma \ref{crucial-itoh}]
By  Lemma \ref{rachel-itoh}(2) we get $H^2(L^\A(A)) = 0$. By \ref{dag} we have an exact sequence
\[
H^1(L^\A(A)) \rt H^1(L^\A(A))(-1) \rt H^1(G_\A(A)) \rt H^2(L^\A(A)) = 0
\]
As $H^1(L^\A(A))_n = 0$ for $n < 0$ it follows that $H^1(G_\A(A))_n = 0$ for $n \leq 0$. Note that to prove $ H^1(L^\A(A)) = 0$ it suffices to show $H^1(G_\A(A)) = 0$ (because
$H^1(L^\A(A))$ has finite length).
Set $G = G_\A(A)$ and $\ov{G} = G/ztG$. We have a exact sequence
\[
0 \rt K \rt G(-1) \xrightarrow{zt} G \rt \ov{G} \rt 0;
\]
where $K$ has finite length. Taking local cohomology we get an exact sequence
\begin{align*}
  H^2(G)(-1)&\xrightarrow{zt} H^2(G) \rt H^2(\ov{G}) \\
  H^3(G)(-1)&\xrightarrow{zt} H^3(G) \rt 0.
\end{align*}
Set $Y = H^2(G)/ztH^2(G)$. Note $Y_n = 0$ for $n \leq 0$. Taking Matlis dual's  we get an exact sequence
\[
0 \rt H^3(G)^\vee \xrightarrow{zt} H^3(G)^\vee(1) \rt H^2(\ov{G})^\vee \rt Y^\vee \rt 0.
\]
By \ref{ing} we get an exact sequence
\[
0 \rt G(\ov{\eF}, \omega_B)  \rt H^2(\ov{G})^\vee \rt Y^\vee \rt 0.
\]
We note that $\depth H^2(\ov{G})^\vee \geq 2$ and $H^2(G(\ov{\eF}, \omega_B))_n = 0$ for $n < 0$. So $Y^\vee_n = 0$ for $n < 0$.
It follows that $Y_n = 0$ for $n > 0$. As $Y_n = 0$ for $n \leq 0$ we get $Y = 0$. So $H^2(G)/zt H^2(G) = 0$. By graded Nakayama Lemma we get $H^2(G) = 0$. As discussed before this implies that $H^1(L^\A(A)) = 0$.
\end{proof}

%999999999999999999999999999999999999999999999999999999999999999999999999999999999

\emph{Acknowledgments:}
I thank  Lucho Avramov, Juergen Herzog, Srikanth Iyengar  and Jugal Verma for
 many useful discussions on the subject of this paper.

%\bibliographystyle{amsplain}
%\bibliography{ref}

\begin{thebibliography}{10}

\bibitem{BSh}
M.P. Brodmann and R.Y. Sharp, \emph{Local cohomology: an algebraic introduction
  with geometric applications}, Cambridge Studies in Advanced Mathematics,
  vol.~60, Cambridge University Press, Cambridge, 1998. \MR{MR1613627}

\bibitem{BH}
W.~Bruns and J.~Herzog, \emph{Cohen-{M}acaulay rings}, Cambridge Studies in
  Advanced Mathematics, vol.~39, Cambridge University Press, Cambridge,(Revised edition) 1997.
  \MR{MR1251956}

\bibitem{cocoa}
A.~Capani, G.~Niesi, and L.~Robbiano, \emph{{CoCoA,a system for doing
  Computations in Commutative Algebra}}, 1995, available via anonymous ftp from
  cocoa.dima.unige.it.

\bibitem{Ciu}
Ciuperc\v{a}, C\v{a}t\v{a}lin,
\emph{Integral closure and generic elements},
J. Algebra 328 (2011), 122--131.

\bibitem{CortZar}
T.~Cortadellas and S.~Zarzuela, \emph{On the depth of the fiber cone of
  filtrations}, J. Algebra \textbf{198} (1997), no.~2, 428--445. \MR{MR1489906}

\bibitem{M2}
D.~R. Grayson and M.~E. Stillman, \emph{Macaulay 2, a software system for
  research in algebraic geometry}, Available at
  http://www.math.uiuc.edu/Macaulay2/.

\bibitem{HUO}
U.~Grothe, M.~Herrmann, and U.~Orbanz, \emph{Graded {C}ohen-{M}acaulay rings
  associated to equimultiple ideals}, Math. Z. \textbf{186} (1984), 531--556.
  \MR{MR744964}

\bibitem{Gr94}
A.~Guerrieri, \emph{On the depth of the associated graded ring of an
  {$m$}-primary ideal of a {C}ohen-{M}acaulay local ring}, J. Algebra
  \textbf{167} (1994), 745--757. \MR{MR1287068}

\bibitem{HeL}
W.~Heinzer, B.~Johnston, D.~Lantz, and K.~Shah, \emph{The {R}atliff-{R}ush
  ideals in a {N}oetherian ring: a survey}, Methods in module theory (Colorado
  Springs, CO, 1991), Lecture Notes in Pure and Appl. Math., vol. 140, Dekker,
  New York, 1993, pp.~149--159.


 \bibitem{HH}
S.~Huckaba, and  C.~Huneke,
\emph{Normal ideals in regular rings},
J. Reine Angew. Math. 510 (1999), 63–-82.   \MR{MR1696091}

 \bibitem{Hun}
C.~Huneke,
\emph{Hilbert functions and symbolic powers},
Michigan Math. J. 34 (1987), no. 2, 293–-318. \MR{MR0894879}

\bibitem{Hoa}
L.~T. Hoa, \emph{Reduction numbers and {R}ees algebras of powers of an ideal},
  Proc. Amer. Math. Soc. \textbf{119} (1993), no.~2, 415--422. \MR{MR1152984}

\bibitem{ItN}
S.~Itoh,
\emph{Coefficients of normal Hilbert polynomials},
 J. Algebra 150 (1992) 101–-117. \MR{MR1174891}



\bibitem{Mar}
T.~Marley, \emph{The coefficients of the {H}ilbert polynomial and the reduction
  number of an ideal}, J. London Math. Soc. (2) \textbf{40} (1989), no.~1,
  1--8. \MR{MR1028910}

\bibitem{Mar-Proc}
\bysame, \emph{The reduction number of an ideal and the local cohomology of the
  associated graded ring}, Proc. Amer. Math. Soc. \textbf{117} (1993), no.~2,
  335--341. \MR{MR1112496}

\bibitem{Ma}
H.~Matsumura, \emph{Commutative ring theory}, Cambridge Studies in Advanced
  Mathematics, vol.~8, Cambridge University Press, Cambridge, 1986, Translated
  from the Japanese by M. Reid. \MR{MR879273}

\bibitem{TamM}
S.~Molinelli and G.~Tamone, \emph{On the {H}ilbert function of certain rings of
  monomial curves}, J. Pure Appl. Algebra \textbf{101} (1995), no.~2, 191--206.
  \MR{MR1348035}

\bibitem{Nag}
M.~Nagata, \emph{Local rings}, Interscience Tracts in Pure and Applied
  Mathematics, No. 13, Interscience Publishers a division of John Wiley \&
  Sons\, New York-London, 1962. \MR{MR0155856}

\bibitem{Nar}
M.~Narita, \emph{A note on the coefficients of {H}ilbert characteristic
  functions in semi-regular local rings}, Proc. Cambridge Philos. Soc.
  \textbf{59} (1963), 269--275. \MR{MR0146212}

\bibitem{North}
 D.~G.~Northcott, A note on the coefficients of the abstract Hilbert function, J. London Math. Soc. 35 (1960) 209–-275. \MR{MR0110731}

\bibitem{Ooishi-cann}
A.~Ooishi, \emph{On the associated graded modules of canonical modules}, J.
  Algebra \textbf{141} (1991), no.~1, 143--157. \MR{MR1118320}

\bibitem{Ooishi-gor-powers}
\bysame,
\emph{The Gorenstein property of the associated graded rings of powers of an ideal},
J. Pure Appl. Algebra 81 (1992), no. 2, 191--196.

\bibitem{Oo-GOR}
\bysame, \emph{On the {G}orenstein property of the associated graded ring and
  the {R}ees algebra of an ideal}, J. Algebra \textbf{155} (1993), no.~2,
  397--414. \MR{MR1212236}

\bibitem{Pu1}
T.~J. Puthenpurakal, \emph{Hilbert-coefficients of a {C}ohen-{M}acaulay
  module}, J. Algebra \textbf{264} (2003), no.~1, 82--97. \MR{MR1980687}

\bibitem{Pu2}
\bysame, \emph{The {H}ilbert function of a maximal {C}ohen-{M}acaulay module},
  Math. Z. \textbf{251} (2005), 551--573. \MR{MR2190344}

 \bibitem{Pu5}
 \bysame, \emph{Ratliff-Rush filtration, regularity and depth of higher associated graded modules I},
  J. Pure Appl. Algebra 208 (2007), no. 1, 159–-176.  \MR{MR2269837}


\bibitem{PuEuler}
\bysame, \emph{A short note on the non-negativity of partial {E}uler
  characteristics}, Beitr\"age Algebra Geom. \textbf{46} (2005), 559--560.
  \MR{MR2196937}

\bibitem{rees-Hom}
D.~Rees, \emph{{A theorem of homological algebra}}, Proc. Cambridge Philos.
  Soc. \textbf{52} (1956), 605--610. \MR{MR0080653}

\bibitem{RV}
M.~E.~Rossi and G.~Valla,
\emph{Hilbert functions of filtered modules},
Lecture Notes of the Unione Matematica Italiana, 9. Springer-Verlag, Berlin; UMI, Bologna, 2010.\MR{MR2723038}

\bibitem{Sa2}
J.~D. Sally, \emph{Tangent cones at {G}orenstein singularities}, Compositio
  Math. \textbf{40} (1980), no.~2, 167--175. \MR{MR563540}

\bibitem{Ser}
J-P. Serre, \emph{Local algebra}, Springer Monographs in Mathematics,
  Springer-Verlag, Berlin, 2000, Translated from the French by CheeWhye Chin
  and revised by the author. \MR{MR1771925}

\bibitem{BS}
B.~Singh, \emph{Effect of a permissible blowing-up on the local {H}ilbert
  functions}, Invent. Math. \textbf{26} (1974), 201--212. \MR{MR0352097}

\bibitem{Stanley-Adv}
R.~P. Stanley, \emph{Hilbert functions of graded algebras}, Advances in Math.
  \textbf{28} (1978), no.~1, 57--83. \MR{MR0485835}

\bibitem{TrungProc}
N.~V. Trung, \emph{Reduction exponent and degree bound for the defining
  equations of graded rings}, Proc. Amer. Math. Soc. \textbf{101} (1987),
  no.~2, 229--236. \MR{MR902533}

\bibitem{VV}
P.~Valabrega and G.~Valla, \emph{Form rings and regular sequences}, Nagoya
  Math. J. \textbf{72} (1978), 93--101. \MR{MR514892}

\bibitem{Wang00}
H.~Wang, \emph{Hilbert coefficients and the associated graded rings}, Proc.
  Amer. Math. Soc. \textbf{128} (2000), 963--973. \MR{MR1628432}

\end{thebibliography}
\providecommand{\bysame}{\leavevmode\hbox to3em{\hrulefill}\thinspace}
\providecommand{\MR}{\relax\ifhmode\unskip\space\fi MR }
\providecommand{\MRhref}[2]{%
  \href{http://www.ams.org/mathscinet-getitem?mr=#1}{#2}
}
\providecommand{\href}[2]{#2}

\end{document}